\documentclass[11pt]{amsart}
\usepackage{amsfonts}
\usepackage{amsmath}
\usepackage{amsthm}
\usepackage{mathrsfs}
\usepackage{amssymb}
\usepackage{bm,bbm}
\usepackage{caption}
\usepackage{tikz}
\usepackage{color}
\usepackage{commath}

\usepackage{bigints}

\numberwithin{equation}{section}

\allowdisplaybreaks

\theoremstyle{plain}
\newtheorem{theorem}{Theorem}[section]
\newtheorem{lemma}[theorem]{Lemma}
\newtheorem{proposition}[theorem]{Proposition}

\theoremstyle{definition}

\newtheorem{remark}[theorem]{Remark}

\def\beqn{\begin{equation}}
\def\beqn*{$$}
\def\eeqn{\end{equation}}

\def\ms{\mathsf}

\def\P{\mathbb{P}}
\def\E{\mathbb{E}}
\def\ga{\gamma}
\def\Pn{\mathcal P_n}

\newcommand{\reals}{{\mathbb R}}

\newcommand{\R}{\reals}
\newcommand{\bbn}{{\mathbb N}}

\newcommand{\vep}{\varepsilon}

\newcommand{\B}{\mathcal B}

\newcommand{\ta}{\theta}

\newcommand{\btheta}{{\bm \theta}}

\newcommand{\one}{{\mathbf 1}}

\newcommand{\remove}[1]{}

\newcommand{\A}{\mathcal A}

\newcommand{\al}{\alpha}

\newcommand{\bt}{{\bf t}}
\newcommand{\bs}{{\bf s}}

\newcommand{\mT}{\mathcal T}

\newcommand{\mD}{\mathcal D}
\newcommand{\gaR}{{\gamma R_n}}
\newcommand{\half}{\frac{1}{2}}
\newcommand{\radpdf}{\bar \rho_{n,\al}}
\newcommand{\Ckal}{B_{k,\al}}
\newcommand{\indh}{{2\al/(k-1)}}
\newcommand{\Ta}{\Theta}
\newcommand{\mC}{\mathcal C}
\newcommand{\mainu}{e^{\frac{1}{2}(d-1)u}}
\newcommand{\wholet}{{0\le t_1 \le \cdots \le t_{m-1} \le u}}

\newcommand{\mS}{\mathcal S}
\newcommand{\mQ}{\mathcal Q}
\newcommand{\bphi}{{\bm \phi}}
\newcommand{\bpsi}{{\bm \psi}}
\newcommand{\mU}{\mathcal U}
\newcommand{\mA}{\mathcal A}

\begin{document}

\bibliographystyle{abbrv}

\renewcommand{\baselinestretch}{1.05}

\title[Stable and Fr\'echet limit theorem]
{Stable and Fr\'echet limit theorem for subgraph functionals in the hyperbolic random geometric graph}

\author{Christian Hirsch}
\address{Department of Mathematics\\
Aarhus University \\
Ny Munkegade, 118, 8000, Aarhus C, Denmark}
\email{hirsch@math.au.dk}
\author{Takashi Owada}
\address{Department of Statistics\\
Purdue University \\
IN, 47907, USA}
\email{owada@purdue.edu}
\author{Ruiting Tong}
\address{Department of Statistics\\
Purdue University \\
IN, 47907, USA}
\email{tong125@purdue.edu}

\thanks{Hirsch's research was supported by a research grant (VIL69126) from VILLUM FONDEN. Owada's research was partially supported by the AFOSR grants FA9550-22-1-0238 at Purdue University. }

\subjclass[2020]{Primary 60D05, 60F17.  Secondary 60G52, 60G70.}
\keywords{hyperbolic random geometric graph, stochastic geometry, stable limit theorem, extreme value theory, star shape counts, clique counts}

\begin{abstract}
We study the fluctuations of subgraph counts in hyperbolic random geometric graphs on the $d$-dimensional Poincar\'e ball in the heterogeneous, heavy-tailed degree regime. In a hyperbolic random geometric graph whose vertices are given by a Poisson point process on a growing hyperbolic ball, we consider two basic families of subgraphs: star shape counts and clique counts, and we analyze their global counts and maxima over the vertex set. Working in the parameter regime where a small number of vertices close to the center of the Poincar\'e ball carry very large degrees and act as hubs, we establish joint functional limit theorems for suitably normalized star shape and clique count processes together with the associated maxima processes. The limits are given by a two-dimensional dependent process whose components are a stable L\'evy process and an extremal Fr\'echet process, reflecting the fact that a small number of hubs dominates both the total number of local subgraphs and their extremes. As an application, we derive fluctuation results for the global clustering coefficient, showing that its asymptotic behavior is described by the ratio of the components of a bivariate L\'evy process with perfectly dependent stable jumps.
\end{abstract}

\maketitle

\section{Introduction}  \label{sec:intro}

Random geometric graphs   in negatively curved spaces have become canonical models for complex networks exhibiting  heavy-tailed degree distributions, tree-like structure, high clustering, and small typical distances. A prominent example is the hyperbolic random geometric graph (HRGG), in which vertices are embedded in a hyperbolic space and edges are drawn between points that are close in hyperbolic distance. Introduced in the physics literature by Krioukov et al.~in \cite{krioukov:papadopoulos:kitsak:vahdat:boguna:2010}, this model was shown  to  produce many features of real networks, such as scale-free behavior via power-law degree distributions. Subsequent works have rigorously analyzed connectivity \cite{bode:fountoulakis:muller:2016}, typical distances \cite{abdullah:fountoulakis:bode:2017}, clustering \cite{candellero:fountoulakis:2016}, and component structure \cite{bode:fountoulakis:muller:2015, fountoulakis:muller:2018, kiwi:mitsche:2019} in hyperbolic network models. 

In this paper we work with the  $d$-dimensional Poincar\'e ball with $d\ge 2$  as in \cite{owada:yogeshwaran:2022}:
\begin{equation}  \label{e:poincare.ball}
\B_d := \{ (x_1,\dots,x_d)\in \R^d: x_1^2 + \cdots + x_d^2 < 1 \}, 
\end{equation}
equipped with  the Riemannian metric 
\begin{equation}  \label{e:riemannian.metric}
ds^2 := \frac{4 (\dif x_1^2 + \cdots + \dif x_d^2)}{(1 - x_1^2 - \cdots - x_d^2)^2}. 
\end{equation}
We define a sequence $R_n := \frac{2}{d-1}\log n$ and consider the hyperbolic ball $B(o,R_n)\subset \B_d$, where $o$ is the origin of $\B_d$ and $R_n$ is the hyperbolic radius. The hyperbolic random geometric graph $\ms{HG}(R_n; \al)$ is then defined to have vertex set given by a Poisson point process $\Pn$ on $B(o,R_n)$ with intensity measure $n(\radpdf \otimes \pi)$, where $\radpdf$ is the density of the radial component, \emph{measured from the boundary of $B(o,R_n)$}; that is, 
\begin{equation}  \label{e:radial.pdf}
\radpdf (t) := \frac{\sinh^{d-1}\big(\al(R_n - t)\big)}{\int_0^{R_n}\sinh^{d-1}(\al s)\dif s}, \ \ \ 0 \le t \le R_n. 
\end{equation}
Moreover, $\pi$ denotes the uniform density of the angular component:
\begin{equation}  \label{e:uniform.angular}
\pi(\ta_1,\dots,\ta_{d-1}) := \frac{\prod_{i=1}^{d-2}\sin^{d-i-1}\ta_i}{2\prod_{i=1}^{d-1}\int_0^\pi \sin^{d-i-1}\ta \dif \ta}, \ \ \ (\ta_1,\dots,\ta_{d-1}) \in \mA_d:= [0,\pi]^{d-2}\times [0,2\pi). 
\end{equation}
The edge set of $\ms{HG}(R_n; \al)$ is defined as $\big\{ (p,q) \in \Pn^2 : 0 < d_H(p,q) \le R_n \big\}$, where $d_H$ denotes the hyperbolic distance determined  by \eqref{e:riemannian.metric}.  We refer to \cite[Section 2.1]{owada:yogeshwaran:2022} for the detailed derivation of \eqref{e:radial.pdf} and \eqref{e:uniform.angular}. 
In the special case $d=2$, the densities \eqref{e:radial.pdf} and \eqref{e:uniform.angular}  simplify respectively as 
\begin{equation}  \label{e:2dim.case.density}
\radpdf(t) = \frac{\al \sinh \big( \al(R_n-t) \big)}{\cosh (\al R_n)-1}, \ \ 0 \le t \le R_n, \ \ \ \text{and} \ \ \ \pi(\theta) =\frac{1}{2\pi}, \ \ \theta \in [0,2\pi). 
\end{equation}
In this case, the resulting model is often referred to as the  KPKVB model. 

A natural way to investigate  the local and mesoscopic structure of such graphs is  through subgraph counts. In the Euclidean setting, the asymptotic behavior of subgraph counts in random geometric graphs is by now well understood: depending on the regime, one obtains variance asymptotics and central limit theorems (CLTs) for counts of edges, trees, and more general subgraph patterns (see \cite{penrose:yukich:2001, penrose:yukich:2003} and the monograph \cite{penrose:2003}). In contrast,   much less is known in the hyperbolic setting. Among the early works on fluctuations of subgraph counts in HRGGs is that  of \cite{owada:yogeshwaran:2022}, who studied the number of copies of a fixed tree in the Poincar\'e ball \eqref{e:poincare.ball}. They obtained precise first- and second-order asymptotics and proved a CLT for subtree counts, revealing a rich pattern of phase transitions as the space becomes more hyperbolic. Additionally, \cite{fountoulakis:yukich:2020} established expectation and variance asymptotics and a CLT for the number of isolated and extreme points in the hyperbolic space. Although not directly related to the specific structure of the KPKVB model and its high-dimensional version, \cite{herold:hug:thale:2021} and \cite{hug:last:schulte:2024} develop fluctuation theory for geometric functionals in hyperbolic space: \cite{herold:hug:thale:2021} establishes quantitative CLTs for the $k$-skeleton of Poisson hyperplanes, while \cite{hug:last:schulte:2024} derives quantitative CLTs for geometric characteristics of hyperbolic Boolean models.

Cliques form another important class of subgraphs; they are widely used for community detection and are closely related to clustering coefficients and higher-order connectivity. For the HRGG, Bl{\"a}sius, Friedrich and Krohmer \cite{blasius:friedrich:krohmer:2018} investigated the expected number of cliques, together with the size of the largest clique, and identified a phase transition at a critical power-law exponent. Moreover, \cite{baguley:maus:ruff:skretas:2025} studied the clique number and degeneracy for colouring problems. All of these works are essentially the first-moment or extremal results   and do not explore the second-order behavior or limit theorems for the fluctuations of clique counts. To the best of our knowledge, there are at present no CLTs, stable limit theorems, or functional limit theorems for clique counts in HRGGs. 

A striking phenomenon that is critical for the present work can already  be found  in \cite{fountoulakis:yukich:2020}, where the authors proved that for $\al>1$ in \eqref{e:2dim.case.density}, 
a Gaussian limit appears for the suitably normalized count of isolated vertices, whereas for $\al \in \big(\half, 1\big)$ 
the usual CLT fails. 
The failure of Gaussian fluctuations suggests that in heavy-tailed regimes with smaller $\al$, the functionals of hyperbolic graphs should instead be governed by non-Gaussian stable laws. 

To complement this point, we conduct simulations of $\ms{HG}(R_n; \al)$ in the Poincar\'e disk with $d=2$. For small values of $\al$ (see Figure \ref{fig:simulation} (a)), the simulations exhibit a highly heterogeneous structure: a small number of vertices close to the center, carry very large degrees and act as hubs, while the majority of vertices near the boundary have relatively small degrees. The resulting graph displays high clustering and a  core–periphery structure. As $\al$ increases (see Figure \ref{fig:simulation} (d)), these hubs become less dominant and the network geometry gradually resembles that of a ``classical" random geometric graph. These simulation results are consistent with the failure of the CLT  in \cite{fountoulakis:yukich:2020}, and are in line with related findings in other hyperbolic models: in the setting of Poisson hyperplanes, \cite{herold:hug:thale:2021} also proved the breakdown of Gaussian weak limits.
Motivated by this picture, a main contribution of this paper is to show rigorously that when $\al$ is small, statistics based on subgraph counts exhibit heavy-tailed fluctuations, which naturally lead to stable limit theorems. This is the first work to identify stable weak limits in a Poincar\'e-ball-type setting.

\begin{figure}
\centering
\includegraphics[scale=0.55]{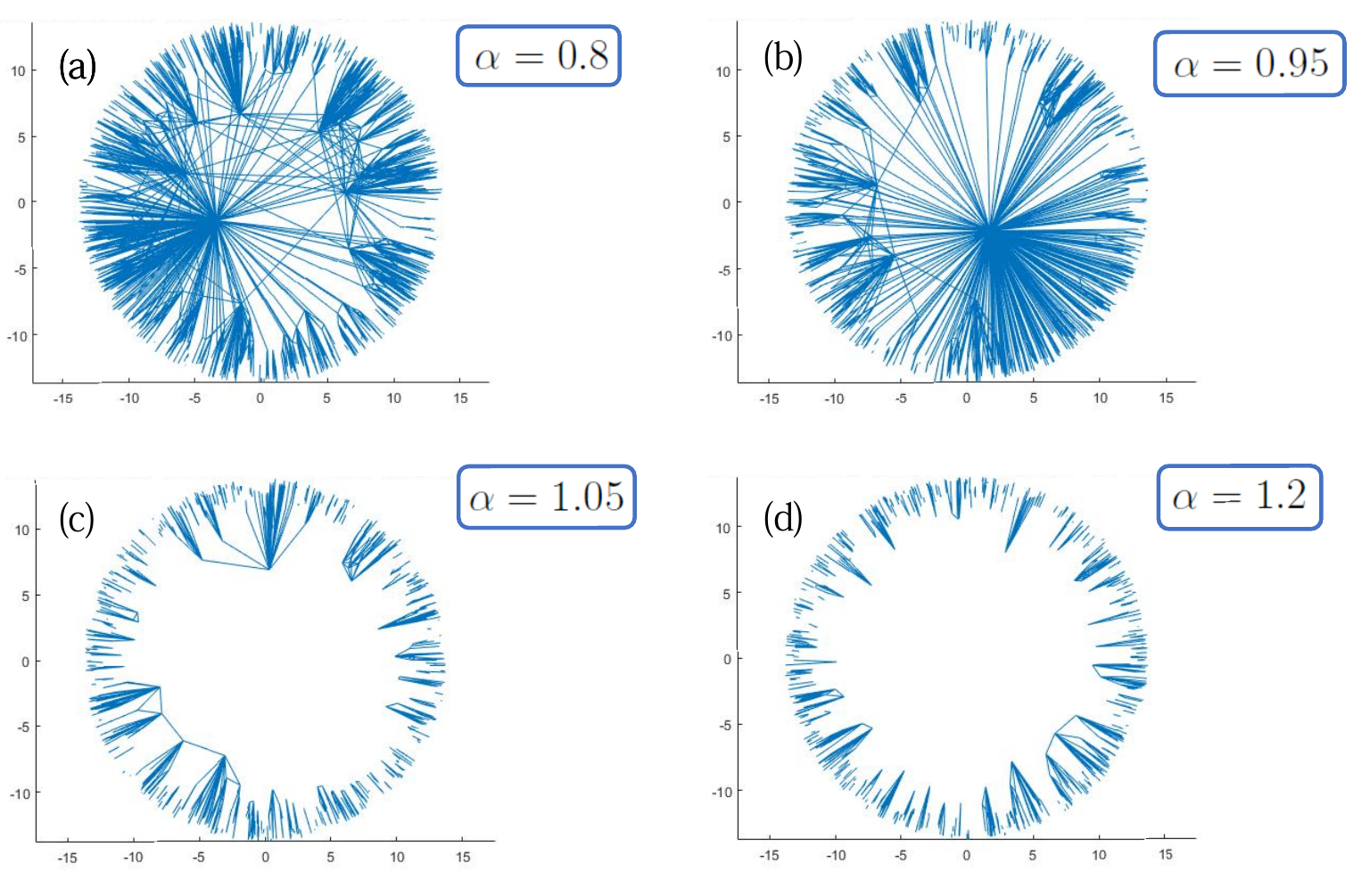}
\caption{\label{fig:simulation} \footnotesize{Simulations of $\ms{HG}(R_n; \al)$ with $d=2$, $R_n = 2 \log 1000 = 13.82$, and different values of $\al$. This figure is taken from \cite{owada:yogeshwaran:2022}. } }
\end{figure}

From a probabilistic viewpoint, heavy-tailed limit theorems for sums and extremes provide a natural framework for our results. It is well known that partial sums of regularly varying sequences lie in the domain of attraction of (non-Gaussian) stable laws. In the functional setting, one typically obtains stable L\'evy processes as the scaling limits of the partial sum processes. Moreover, extremal Fr\'echet processes arise as weak limits of partial maxima processes; see, e.g.,  \cite{resnick:2007}. In the extreme-value literature, a substantial body of work establishes joint weak convergence of partial sum and partial maxima processes for a broad class of regularly varying (and often dependent) sequences; see \cite{chow:teugels:1978, resnick:1986, krizmanic:2020, bai:tang:2024, matsui:mikosch:wintenberger:2025}. According to these results, in heavy-tailed models,  sums and maxima typically grow on the same scale, and a joint analysis is required to capture their intrinsic dependence, which arises from a  small number of extreme vertices dominating the contribution of all the others to the sum.
In the present context, where subgraph counts in hyperbolic graphs are built from vertices whose degrees are heavy-tailed, the hubs that generate the largest numbers of subgraphs also dominate the total number of local subgraph in the entire space. It is therefore natural to expect that the joint fluctuations of subgraph counts and their   maxima are asymptotically governed by a pair of dependent processes: a stable L\'evy process and an extremal Fr\'echet process driven by the same  extreme vertices. 

From a broader point of view, HRGGs are  special cases of more general frameworks for scale-free spatial networks. Geometric inhomogeneous random graphs (GIRGs), introduced by Bringmann, Keusch, and Lengler (\cite{bringmann:keusch:lengler:2019}), model vertices with independent weights and positions in a metric space, and connect pairs with probability depending on their weights and distance. The HRGGs are asymptotically special cases of GIRGs, and more generally as instances of spatial inhomogeneous random graphs (SIRGs). 
For the relation with the GIRGs, at least for $d$=2, the local limits of HRGGs are known to coincide with those of suitable GIRGs; see \cite[Theorem 9.33]{book-hof}. However, local limits capture only first-order, neighborhood-level properties. In contrast, the joint functional stable-Fr\'echet limit theorems established in this paper are driven directly  by  extreme vertices whose influence extends across the entire graph. Such fluctuation phenomena cannot be explored from the existing results on local weak convergence for GIRGs.
Additionally, within the SIRG framework, local weak limits and the scaling of clustering functions have recently been worked out (\cite{hofstad:hoorn:maitra:2023a, hofstad:hoorn:maitra:2023b}; see also \cite{komjathy:lodewijks:2020}), revealing, in particular, universal features of clustering in scale-free spatial networks. Nevertheless, distributional limit theorems for subgraph counts in these models are still  limited. One notable exception  is the age-dependent random connection model (ADRCM), for which Hirsch and Owada \cite{hirsch:owada:2026} established stable limit theorems for the counts of certain trees and cliques in heavy-tailed regimes. These ADRCM results, though non-hyperbolic, give important methodological guidance for the present work. However, the  results in \cite{hirsch:owada:2026} are restricted to one-dimensional (non-functional) limits; functional convergence in Skorokhod spaces and the joint behavior of sums and maxima are not treated there. 

Another key  observable in network analysis is the (global) clustering coefficient $\ms{CC}_n:= 3T_n/W_n$, where $T_n$ is the  number of triangles and $W_n$ the number of wedges in the network of size $n$. In many  random graph models with power-law degrees, the clustering coefficient converges to a positive constant. This has been shown, for example, in spatial preferential attachment models   \cite{jacob:morters:2015}, in the HRGG \cite{candellero:fountoulakis:2016, fountoulakis:hoorn:muller:schepers:2021}, and, in a more general setting, for SIRGs \cite{hofstad:hoorn:maitra:2023a} and for related hyperbolic and other random graph models \cite{stegehuis:hofstad:leeuwaarden:2019}. These works focus mainly on law-of-large-numbers behavior: they identify limiting values of clustering coefficients and  describe how clustering scales with degree, but they do not derive  limit theorems for the fluctuations of $\ms{CC}_n$. Since the numerator and denominator of $\ms{CC}_n$ are  triangle and wedge counts, any fluctuation theory for $\ms{CC}_n$ in hyperbolic graphs must be built on a joint analysis of these counting statistics. In this paper, we develop such a theory for $\ms{CC}_n$ in the HRGG. 

The remainder of  the paper is organized as follows. In Section \ref{sec:star.shape} we establish joint functional limit theorems for star shaped subgraph counts, with dependent stable L\'evy and extremal  Fr\'echet process limits.  Section \ref{sec:clique.counts} establishes analogous functional limit theorems for clique counts. 
A key mechanism there is domination by a small number of extreme vertices near the origin of \eqref{e:poincare.ball}, which simultaneously drive subgraph counts and their extremes. From this view, it is crucial to compute the moments of subgraph statistics when the extreme vertex is fixed near the origin; see Sections \ref{sec:moments.star} and \ref{sec:moments.clique}. We remark that the stable–Fr\'echet joint limits belong to the same universality class as classical i.i.d.~heavy-tailed sums and maxima. However, we would like to emphasize that the novelty of this paper does not lie in the form of the limiting process itself, but in the fact that it arises endogenously from hyperbolic geometry.  
Finally, Section \ref{sec:clustering.coeff} develops the fluctuation theory for the global clustering coefficient by exploiting the main results of the previous sections. 
The key mechanism in this section is  that  extreme vertices near the origin of \eqref{e:poincare.ball} simultaneously drive multiple subgraph statistics, leading to  aligned extreme fluctuations across them. Indeed,  the limiting random variable for  clustering coefficients is given by the ratio of two heavy-tailed random variables with perfectly aligned jumps.

From a technical viewpoint, our proofs combine tools from stochastic geometry and point process theory. We exploit the relationship between HRGGs and GIRGs/SIRGs (in particular the ADRCM), but we must refine these methods to cope with the strong inhomogeneity of the hyperbolic metric. In particular, when dealing with cliques on $m$ vertices (henceforth we call it $m$-cliques), one encounters a delicate regime $\half <\al< \frac{2m-3}{2m-2}$, in which contributions from a small number of extremely central vertices dominate more strongly than in the regime $\frac{2m-3}{2m-2} < \al <1$.  At present, we have not succeeded in decoupling this effect (see Remark \ref{rem:open.q.clique}), and fully treating the case $\half <\al< \frac{2m-3}{2m-2}$ remains an open problem. 

Another important direction for future research is the extension of our main results, namely Theorems \ref{t:joint.sum.max.star} and \ref{t:joint.sum.max.clique}, to more general subgraph counts. Unlike in Euclidean random geometric graphs (see, e.g., the monograph \cite{penrose:2003}), the situation for the HRGG is much more delicate due to the inhomogeneity of the hyperbolic metric, and we expect that substantially different techniques will be required to obtain the analogues of the first- and second-moment estimates in Sections \ref{sec:moments.star} and \ref{sec:moments.clique}.

Throughout the paper, we will use the following notation. The cardinality of a set $A$ is denoted by $|A|$. Given a sequence $(a_i)_{i\ge1}$ of real numbers, we denote $\bigvee_{i=1}^n a_i:= \max_{1\le i \le n} a_i$ and $\bigwedge_{i=1}^n a_i:= \min_{1\le i \le n} a_i$. Moreover, $\Rightarrow$ denotes weak convergence and $\stackrel{\P}{\to}$ is convergence in probability in a given space. Furthermore, $\one\{\cdot\}$ represents an indicator function and $(a)_+ = a$ if $a\ge0$ and $0$ otherwise. Finally, we will use the capital letter $C$ and its variants, say $C^*$ and $c_i$, to denote finite positive constants whose exact values are not important and may change from one appearance to the next.

\section{Star shape counts}   \label{sec:star.shape}

We first denote by $\Pn = \{ X_1,X_2,\dots,X_{N_n} \}$ the Poisson point process in $B(o,R_n)\subset \B_d$, where  $X_1,X_2,\dots$ are i.i.d.~random variables with density $\radpdf \otimes \pi$ and $N_n$ is Poisson distributed with mean $n$ and is independent of $(X_i)_{i\ge1}$. 
For $X, Y\in \Pn$, we write $X\to Y$ if and only if $d_H(X,Y)\le R_n$ and $d_H(o,X)> d_H(o, Y)$. 

Assume $1/2 < \al < k-1$, and 
let $p = (u,\btheta) \in B(o,R_n)$, where $u=R_n-d_H(o,p)$ is the radial component measured from the boundary of $B(o,R_n)$, and $\btheta=(\ta_1, \dots, \ta_{d-1})\in \A_d$ is the angular component. 
For $k\ge 2$, denote by $S_{k-1}$  the star graph on vertices $\{0,1,\dots,k-1\}$ with center $0$ and leaves $1,\dots,k-1$ (i.e., edges $\{0,i\}$ for $i=1,\dots,k-1$). Then, an \emph{injective} graph homomorphism $\phi :S_{k-1}\to \ms{HG}(R_n; \alpha)\cup \{p\}$ with the constraint $\phi(0)=p$ is uniquely determined by an ordered $(k-1)$-tuple of distinct neighbors of $p$ in $\ms{HG}(R_n; \al)$. Hence, the (labeled) star homomorphism count centered at $p$ is defined by 
\begin{equation}  \label{e:def.star.shape.counts}
\mD_{k,n}(p):= \sum_{(Y_1,\dots, Y_{k-1})\in (\Pn)_{\neq}^{k-1}} \one \big\{ \max_{1\le i \le k-1}d_H(Y_i,p)\le R_n \big\}, 
\end{equation}
where 
$(\Pn)_{\neq}^{k-1} := \big\{ (Y_1,\dots, Y_{k-1}) \in \Pn^{k-1}: Y_i \neq Y_j \text{ for } i \neq j \big\}$. 
Note that $\mD_{k,n}(p)$ does not necessarily enumerate induced stars. 
For $u \in (0,R_n)$, write $\mD_{k,n}(u) := \mD_{k,n}\big(u,{\bf 0}\big)$ with ${\bf 0} = (0,\dots,0)\in \mA_d$ and define $\mu_{k,n}(u) := \E \big[ \mD_{k,n}(u,\btheta) \big]$. Note that $\mu_{k,n}(u)$ is invariant with respect to $\btheta\in \mA_d$, due to the uniformity of the density \eqref{e:uniform.angular}. 

Our main theorem describes the joint weak convergence of the process-level star shape counts $\sum_{i=1}^{[N_nt]} \mD_{k,n}(X_i)$ and their maxima process $\bigvee_{i=1}^{[N_nt]}\mD_{k,n}(X_i)$. We note that the boundary case $\al = (k-1)/2$ is excluded from consideration.
It is already known  from \cite[Theorem 3]{owada:yogeshwaran:2022} (see in particular Equ.~(20) therein) that if $\al>k-1$, a (non-functional) CLT holds for the star shape counts. Thus,  when combined with Theorem \ref{t:joint.sum.max.star} below,  we have obtained a complete picture of the asymptotic behavior of the star shape counts, apart from the boundary case.

Before stating the theorem, we introduce the scaling constants $a_{k,n}:=B_{k,\al} n^{(k-1)/(2\al)}$ for $n\ge1$, where 
$$
B_\al := \frac{2^d \al}{(d-1)(2\al-1)\kappa_{d-2}} \ \ \text{ with } \ \ \kappa_{d-2} := \int_0^\pi \sin^{d-2}\ta \dif \ta,
$$
and  $B_{k,\al} := (B_\al)^{k-1}$. 

\begin{theorem}  \label{t:joint.sum.max.star}
$(i)$ If $\half (k-1)<\al< k-1$, then as $n\to\infty$, 
\begin{align}
&\left( \frac{1}{a_{k,n}} \bigg( \sum_{i=1}^{[N_n\cdot\,]}\mD_{k,n}(X_i) - [N_n\cdot\,]  \E [\mD_{k,n}(X_1)] \bigg), \ \frac{1}{a_{k,n}}\bigvee_{i=1}^{[N_n \cdot \,]} \mD_{k,n}(X_i)\right)  \Rightarrow (S_{2\al/(k-1)}(\cdot),\, Y_{2\al/(k-1)}(\cdot)),   \label{e:weak.conv.join.sum.max.star}
\end{align}
in the space $D\big([0,1], \R \times [0,\infty)\big)$ of right-continuous functions from $[0,1]$ to $\R\times [0,\infty)$ with left limits. The weak limit in \eqref{e:weak.conv.join.sum.max.star} is defined by the hybrid characteristic–distribution function
\begin{equation}  \label{e:hybrid1}
\E \big[ e^{ivS_{2\al/(k-1)}(t)} \one \{ Y_{2\al/(k-1)}(t) \le z \} \big] = t\int_0^z e^{ivy} g(v,y)^t e^{-ty^{-2\al/(k-1)}} \ms{m}_\indh (\dif y),  \  \ v\in \R,\,  z\ge 0,  
\end{equation}
where
$$
g(v,y) := \exp \Big\{ \int_0^\infty \big( (e^{ivx}-1)\one \{ x<y \} -ivx \big) \ms{m}_\indh (\dif x)\Big\}, 
$$
and $\ms{m}_{2\al/(k-1)}$ is a Radon measure on $(0,\infty]$ satisfying $\ms{m}_{2\al/(k-1)}\big((y,\infty]\big)=y^{-2\al/(k-1)}$ for $y>0$. \\
Marginally, $S_{2\al/(k-1)}$ is a zero-mean $2\al/(k-1)$-stable L\'evy process whose characteristic function is given by \eqref{e:chf.first.comp.zero.mean.stable} below. Furthermore, $Y_{2\al/(k-1)}$ is an extremal $2\al/(k-1)$-Fr\'echet process; that is, for $0\equiv t_0<t_1<\cdots <t_m<\infty$ and $z_i\ge0$, $i=1,\dots,m$, 
$$
\P\big( Y_{2\al/(k-1)}(t_i)\le z_i, \, i=1,\dots,m \big)=\exp \Big\{ -\sum_{j=1}^m (t_j-t_{j-1}) (z_j \wedge \cdots \wedge z_m)^{-2\al/(k-1)} \Big\}. 
$$
$(ii)$ If $\frac{1}{2}<\al<\half (k-1)$, then as $n\to\infty$, 
$$
\left( \frac{1}{a_{k,n}}  \sum_{i=1}^{[N_n \cdot\, ]}\mD_{k,n}(X_i), \ \frac{1}{a_{k,n}} \bigvee_{i=1}^{[N_n \cdot \,]} \mD_{k,n}(X_i) \right)  \Rightarrow (\widetilde{S}_{2\al/(k-1)}(\cdot), Y_\indh(\cdot)), \ \ \text{in } D\big( [0,1], [0,\infty)^2\big), 
$$
where the weak limit is defined by 
\begin{equation}  \label{e:hybrid2}
\E \big[ e^{iv\widetilde{S}_{2\al/(k-1)}(t)} \one \{ Y_{2\al/(k-1)}(t) \le z \} \big] = t\int_0^z e^{ivy} \, \widetilde{g}(v,y)^t e^{-ty^{-2\al/(k-1)}} \ms{m}_\indh (\dif y),  \ \ \ v\in \R,\,  z\ge 0,  
\end{equation}
with 
$$
\widetilde{g}(v,y) := \exp \Big\{ \int_0^y (e^{ivx}-1) \ms{m}_\indh (\dif x)\Big\}. 
$$
Marginally, $\widetilde S_{2\al/(k-1)}$ is a $2\al/(k-1)$-stable subordinator (i.e., a stable L\'evy process with increasing paths) with its characteristic function given by \eqref{e:chf.first.comp.zero.stable.subordinator} below. 
\end{theorem}

\begin{remark} 
$(i)$ The hybrid characteristic–distribution function $\E\big[ e^{ivX}\one \{Z\le z\} \big]$ combines the characteristic function of $X$ with the distribution function of a non-negative random variable $Z$. This function uniquely determines the joint distribution of $(X,Z)$ (see \cite[Lemma D.2.1]{mikosch:wintenberger:2024} and \cite{chow:teugels:1978}). 

\noindent $(ii)$ Since Lemma \ref{l:marginal.max.star} in Section \ref{sec:main.result.star.count}  applies for all $\half < \al < k-1$, the restriction $\al \neq (k-1)/2$ is not required for the weak convergence of the partial maxima process.
\end{remark}

For the proof of Theorem \ref{t:joint.sum.max.star}, we first analyze the first and second moments $\mu_{k,n}(u)$   and $\text{Var}\big( \mD_{k,n}(u) \big)$ for $u\in (0,R_n)$, in Propositions \ref{p:exp.mu.n} and \ref{p:var.mu.n}, respectively. Based on these moment asymptotics, one of the main goals of Section \ref{sec:pp.conv.star.shape} is to  derive the weak convergence of a point process associated with $\big( a_{k,n}^{-1}\mu_{k,n}(U_i), \, i=1,\dots,N_n \big)$. Section \ref{sec:main.result.star.count} then completes the proof of Theorem \ref{t:joint.sum.max.star}. First, Proposition \ref{p:main.mu.version.star} shows that this point process convergence yields  the weak  convergence of the star shape counts and their maxima associated with $\big( a_{k,n}^{-1}\mu_{k,n}(U_i), \, i=1,\dots,N_n \big)$. Subsequently, Lemmas \ref{l:marginal.max.star} and \ref{l:stable.limit.thm.D} ensure that the discrepancy between the weak convergence for $\big( a_{k,n}^{-1}\mu_{k,n}(U_i), \, i=1,\dots,N_n \big)$ and that for $\big( a_{k,n}^{-1}\mD_{k,n}(X_i), \, i=1,\dots,N_n \big)$ is negligible.

\subsection{Moments}  \label{sec:moments.star}

We begin with the first moment results on $\mu_{k,n}(u)$ for $u\in (0,R_n)$. 

\begin{proposition} \label{p:exp.mu.n}
Let $\gamma\in \big(\frac{1}{2\al},1\big)$. \\
$(i)$ We have, uniformly for $u \in (0,\gaR)$, 
$$
\mu_{k,n}(u) \le \big( 1+o_n(1) \big) B_{k,\al} \bigg( 1+ \sum_{\ell=0}^{k-2}\binom{k-1}{\ell} (e^{\half (d-1)(1-2\al)u} +s_n)^{k-1-\ell} \bigg) e^{\frac{1}{2}(d-1)(k-1)u}, 
$$
where $s_n:= B_\al^{-1} n^{-(2\al-1)(1-\ga)} e^{\al(d-1)\omega_n}$. \\
$(ii)$ We have, uniformly for $u \in (0,\gaR)$, 
$$
\mu_{k,n}(u) \ge \big( 1+o_n(1) \big) B_{k,\al} \big( 1- e^{\half (d-1)(1-2\al)u} \big)^{k-1} e^{\frac{1}{2}(d-1)(k-1)u}. 
$$
$(iii)$ There exists $C\in (0,\infty)$, independent of $n$ and $u$, such that for all $n\ge1$ and $u\in (0,R_n)$,  
$$
\mu_{k,n} (u) \le C e^{\al(d-1)(k-1)\omega_n}\cdot e^{\frac{1}{2}(d-1)(k-1)u}, 
$$ 
where $\omega_n :=\log \log R_n$. 
\end{proposition}

\begin{proof}
Given $p = (u,\btheta) \in B(o,R_n)$ with $u=R_n-d_H(o,p)$ and $\btheta\in \A_d$, 
the star shape counts \eqref{e:def.star.shape.counts} admits the exact partition 
$$
\mD_{k,n}(p) = \sum_{\ell=0}^{k-1} \binom{k-1}{\ell} \sum_{(Y_1,\dots, Y_{k-1})\in (\Pn)_{\neq}^{k-1}} \hspace{-10pt}\one \big\{ \max_{1\le i \le k-1}d_H(Y_i,p)\le R_n, \, \max_{1\le i \le \ell}V_i \le u, \, \min_{\ell+1\le i \le k-1}V_i >u\big\}, 
$$
where $V_i = R_n - d_H(o,Y_i)$ for $i = 1,\dots,k-1$. For $\ell=0$, we remove the condition $\max_{1\le i \le \ell}V_i \le u$, and for $\ell = k-1$, the condition $\min_{\ell+1\le i \le k-1}V_i >u$ is removed. Taking expectations on the both sides, we obtain $\mu_{k,n}(u) = \sum_{\ell=0}^{k-1}\binom{k-1}{\ell}\mu_{k,n}^{(\ell)}(u)$, where
$$
\mu_{k,n}^{(\ell)}(u) := \E \bigg[ \sum_{(Y_1,\dots,Y_{k-1}) \in (\Pn)_{\neq}^{k-1}} \one \big\{ \max_{1\le i \le k-1}d_H(Y_i,p) \le R_n, \, \max_{1\le i \le \ell}V_i \le u, \, \min_{\ell+1\le i \le k-1}V_i >u \big\}\bigg]. 
$$

Recalling that $X_1,\dots,X_{k-1}$ denote   i.i.d.~random variables with density $\radpdf \otimes \pi$, such that $U_i = R_n-d_H(o,X_i)$, $i=1,\dots,k-1$, 
the Mecke formula for Poisson point processes (see, e.g., Chapter 4 in \cite{last:penrose:2017}) yields that 
$$
\mu_{k,n}^{(\ell)}(u) = n^{k-1} \P \big( \max_{1\le i \le k-1}d_H(X_i,p) \le R_n, \, \max_{1\le i \le \ell}U_i \le u, \, \min_{\ell+1\le i \le k-1}U_i >u \big). 
$$
By independence and the conditioning on radial components, 
\begin{align}  
\mu_{k,n}^{(\ell)}(u) &= \Big( n \int_0^u \P \big( d_H(X,p) \le R_n \,|\, t \big) \bar \rho_{n,\al}(t) \dif t \Big)^{\ell} \Big( n \int_u^{R_n} \P \big( d_H(X,p) \le R_n \,|\, t \big) \bar \rho_{n,\al}(t) \dif t \Big)^{k-1-\ell} \label{e:X=(t, Theta)} \\
&=:  \big( A_n(u)+B_n(u) \big)^\ell \big( A_n'(u)+B_n'(u) \big)^{k-1-\ell},  \notag 
\end{align}
where $X = (t,\Theta)$, so that the conditional probabilities above   act only on the angular component $\Theta$,  and 
\begin{align}   \label{e:An(u)}
A_n(u) &:= n\int_{0}^u  \P\big( d_H(X,p) \le R_n\, \big| \, t \big)\one \{ t +u\le R_n-\omega_n  \} \, \bar \rho_{n,\alpha}(t)\dif t, \\
B_n(u) &:= n\int_{0}^u  \P\big( d_H(X,p) \le R_n\, \big| \, t \big) \one \{ t +u >  R_n-\omega_n  \} \, \bar \rho_{n,\alpha}(t)\dif t,  \label{e:Bn(u)} \\
A_n'(u) &:= n\int_{u}^{R_n}  \P\big( d_H(X,p) \le R_n\, \big| \, t \big)\one \{ t +u\le R_n-\omega_n  \} \, \bar \rho_{n,\alpha}(t)\dif t,   \notag  \\ 
B_n'(u) &:= n\int_u^{R_n}  \P\big( d_H(X,p) \le R_n\, \big| \, t \big) \one \{ t +u >  R_n-\omega_n  \} \, \bar \rho_{n,\alpha}(t)\dif t. \notag 
\end{align}
\medskip

\noindent \underline{\emph{Proofs of $(i)$ and $(iii)$}}:
It follows from Lemma \ref{l:Lemma1.OY} $(ii)$ and Lemma \ref{l:Lemma4.OY} in the Appendix that, uniformly for $u\in (0,\gaR)$, 
\begin{equation}  \label{e:initial.Anu}
A_n(u) = \big( 1+o_n(1) \big) \frac{2^{d-1}\al}{\kappa_{d-2}}\, e^{\frac{1}{2}(d-1)u}\int_0^{u} e^{\half (d-1)(1-2\al)t} \one \{ t+u\le R_n-\omega_n \}\dif t. 
\end{equation}
We claim that,  as $n\to \infty$, 
\begin{equation}  \label{e:unif.conv.ratio}
\sup_{u \in (0,\gaR)} \Big( \int_0^u e^{\half (d-1)(1-2\al)t}\dif t \Big)^{-1}\int_0^u e^{\half (d-1)(1-2\al)t} \one \{ t+u> R_n-\omega_n \}\dif t \to 0. 
\end{equation}
For the proof, note first that $t+u>R_n-\omega_n$ implies $u>\half (R_n-\omega_n)$, so we may restrict to $u \in \big( \half (R_n-\omega_n), \gaR \big)$ (here we implicitly assume $\ga\ge 1/2$, since otherwise there is nothing to prove). Now,  \eqref{e:unif.conv.ratio} can be  bounded by
$$
\Big( \int_0^{\half (R_n-\omega_n)}e^{\half (d-1)(1-2\al)t}\dif t \Big)^{-1} \int_{(1-\ga)R_n-\omega_n}^\infty e^{\half (d-1)(1-2\al)t}\dif t \to 0, \ \ \ n\to\infty. 
$$
Thus, we can remove the indicator in \eqref{e:initial.Anu} and conclude that, uniformly for $u\in (0,\gaR)$,
\begin{align}
A_n(u) &= \big( 1+o_n(1) \big) \frac{2^{d-1}\al}{\kappa_{d-2}}\, e^{\frac{1}{2}(d-1)u}\int_0^{u} e^{\half (d-1)(1-2\al)t}\dif t \label{e:def.A.infty.u}\\
&= \big( 1+o_n(1) \big) B_\al \big( 1-e^{\half (d-1)(1-2\al)u} \big) e^{\half (d-1)u}. \notag  
\end{align}

We also need a uniform upper bound for $A_n(u)$ that holds for all $n \ge 1$ and $u \in (0,R_n)$. To this end, we apply Lemma \ref{l:Lemma1.OY} $(i)$, instead of Lemma \ref{l:Lemma1.OY} $(ii)$, together with Lemma \ref{l:Lemma4.OY} as before, to obtain 
\begin{equation}  \label{e:An.u.bound}
A_n(u) \le C   e^{\frac{1}{2}(d-1)u} \int_0^u e^{\frac{1}{2}(d-1)(1-2\al)t}\dif t \le C e^{\frac{1}{2}(d-1)u}, 
\end{equation}
for all $n\ge1$ and $u\in (0,R_n)$, where $C$ is a finite constant independent of both $n$ and $u$. 

Next, we consider $B_n(u)$. First, for $u\in (0,R_n)$, Lemma \ref{l:Lemma1.OY} $(i)$ yields that 
\begin{align*}
B_n(u) &\le n\int_{R_n-\omega_n-u}^\infty \bar \rho_{n,\al}(t) \dif t \le \big( 1+o_n(1) \big) n\int_{R_n-\omega_n-u}^\infty\al (d-1) e^{-\al (d-1)t}\dif t \\
&\le \big( 1+o_n(1) \big)ne^{-\al (d-1)(R_n-\omega_n-u)} = \big( 1+o_n(1) \big)n^{1-2\al} e^{\al(d-1)\omega_n} \cdot e^{\al(d-1)u}.
\end{align*}
If $u\in (0,\gaR)$, then
\begin{align}
B_n(u) &\le \big( 1+o_n(1) \big) n^{1-2\al} e^{\al (d-1)\omega_n} \cdot  e^{\half (d-1)(2\al-1)\gaR}\cdot  e^{\half (d-1)u} =\big( 1+o_n(1) \big)  B_\al s_n  e^{\half (d-1)u}. \label{e:bound.Bnu2}
\end{align}
We also have, for all $n\ge1$ and  $u \in (0,R_n)$ (i.e., $u$ can be greater than $\gaR$), 
\begin{equation}  \label{e:bound.Bnu3}
B_n(u)\le   Cn^{1-2\al} e^{\al (d-1)\omega_n} \cdot e^{\half (d-1)(2\al-1)R_n} \cdot e^{\half (d-1)u} = Ce^{\al (d-1)\omega_n} \cdot e^{\half (d-1)u}. 
\end{equation}

Subsequently, we turn to $A_n'(u)$: by Lemma \ref{l:Lemma1.OY} $(i)$ and Lemma \ref{l:Lemma4.OY}, one can see that uniformly for $u\in(0,R_n)$, 
\begin{align}
A_n'(u) &\le \big( 1+o_n(1) \big) \frac{2^{d-1}\al}{\kappa_{d-2}} \mainu \int_u^\infty  e^{\half (d-1)(1-2\al)t}\dif t  \label{e:An'(u).bdd} \\
&=  \big( 1+o_n(1) \big) B_\al e^{\half (d-1)(1-2\al)u} \cdot \mainu. \notag 
\end{align}
Moreover, since we again have $B_n'(u) \le n\int_{R_n-\omega_n-u}^\infty \radpdf (t)\dif t$,  the same bounds are applicable here as in \eqref{e:bound.Bnu2} and \eqref{e:bound.Bnu3}, namely, 
\begin{align}
\begin{split}  \label{e:B_n'bdd}
B_n'(u) \le \begin{cases}
\big( 1+o_n(1) \big)  B_\al s_n \mainu & \text{ for all } n\ge1 \text{ and } u \in (0,\gaR), \\[3pt]
Ce^{\al(d-1)\omega_n} \cdot \mainu & \text{ for all } n\ge1 \text{ and } u \in (0,R_n). 
\end{cases}
\end{split}
\end{align}

By \eqref{e:def.A.infty.u} and \eqref{e:bound.Bnu2}, it now  holds that $A_n(u) + B_n(u)\le \big( 1+o_n(1) \big)B_\al \mainu$ uniformly for $u\in (0,\gaR)$. Similarly,  by \eqref{e:An'(u).bdd} and \eqref{e:B_n'bdd}, we have, uniformly for $u\in (0,\gaR)$,
$$
A_n'(u) + B_n'(u)\le \big( 1+o_n(1) \big)B_\al (e^{\half (d-1)(1-2\al)u}+s_n )\mainu. 
$$
Substituting these results into \eqref{e:X=(t, Theta)}, we have, uniformly for $u\in (0,\gaR)$, 
$$
\mu_{k,n}^{(\ell)}(u) \le \big( 1+o_n(1) \big)B_{k,\al} (e^{\half (d-1)(1-2\al)u} + s_n)^{k-1-\ell} e^{\half (d-1)(k-1)u}, 
$$
and thus, 
\begin{align*}
\mu_{k,n}(u) &\le \big( 1+o_n(1) \big)B_{k,\al} \bigg( 1+ \sum_{\ell=0}^{k-2}\binom{k-1}{\ell} (e^{\half (d-1)(1-2\al)u} + s_n)^{k-1-\ell} \bigg) e^{\half (d-1)(k-1)u}, 
\end{align*}
which completes the proof of the bound in $(i)$.

For the proof of $(iii)$, if $u\in (0,R_n)$  (so $u$ may be larger than $\gaR$), then \eqref{e:An.u.bound}, \eqref{e:bound.Bnu3}, \eqref{e:An'(u).bdd}, and \eqref{e:B_n'bdd} yield
$$
A_n(u)+B_n(u)\le Ce^{\al(d-1)\omega_n}\cdot \mainu, \ \ \text{ and } \ \ A_n'(u)+B_n'(u)\le Ce^{\al(d-1)\omega_n}\cdot \mainu, 
$$
and therefore, for all $n\ge 1$ and $u \in (0,R_n)$, 
$$
\mu_{k,n}^{(\ell)}(u) \le C e^{\al(d-1)(k-1)\omega_n}\cdot e^{\half (d-1)(k-1)u}. 
$$
This has established the bound in $(iii)$. 
\medskip

\noindent \underline{\emph{Proof of $(ii)$}}: The  desired lower bound follows directly from  
$$
\mu_{k,n}(u) \ge \mu_{k,n}^{(k-1)}(u) \ge A_n(u)^{k-1} = \big( 1+o_n(1) \big) B_{k,\al} \big( 1-e^{\half (d-1)(1-2\al)u} \big)^{k-1}e^{\half (d-1)(k-1)u}, 
$$
uniformly for $u\in (0,\gaR)$, 
where  \eqref{e:def.A.infty.u} is applied for the last equality. 
\end{proof}

The next result provides an upper bound of the variance of $\mD_{k,n}(u)$ for $u\in (0,R_n)$. 

\begin{proposition}  \label{p:var.mu.n} 
There exists a constant $C\in (0,\infty)$ not depending on $n$ and $u$, such that for all $n\ge1$ and $u \in (0,R_n)$, 
$$
\text{Var}\big( \mD_{k,n}(u) \big)\le C e^{\al(d-1)(2k-3)\omega_n} \cdot e^{\half (d-1)(2k-3)u}. 
$$
\end{proposition}

\begin{proof}
For $p, x_1, \dots, x_{k-1}\in B(o,R_n)$, we define $g_n(p,x_1,\dots, x_{k-1}):= \sum_{\ell=0}^{k-1} g_n^{(\ell)} (p, x_1,\dots, x_{k-1})$, where  
$$
g_n^{(\ell)} (p, x_1,\dots, x_{k-1}) := \one \Big\{ \max_{1\le i \le k-1}d_H(x_i, p)\le R_n, \, \sum_{i=1}^{k-1}\one\{u_i \le u\}=\ell  \Big\}, 
$$
with $x_i=(u_i, \ta_i)$,  $u_i =R_n-d_H(o, x_i)$ for $i=1,\dots,k-1$. Then,  the Mecke formula gives that 
\begin{align}
\begin{split}  \label{e:decomp.var}
\text{Var}\big( \mD_{k,n}(u) \big) &= \mu_{k,n}(u) + \sum_{q=1}^{k-2}  \sum_{\ell_1=0}^{k-1}\sum_{\ell_2=0}^{k-1} \E \bigg[ \sum_{(Y_1,\dots,Y_{k-1})\in (\Pn)_{\neq}^{k-1}} \hspace{-10pt}\sum_{\substack{(Z_1,\dots,Z_{k-1})\in (\Pn)_{\neq}^{k-1}, \\ |(Y_1,\dots,Y_{k-1})\cap (Z_1,\dots,Z_{k-1})|=q}} \hspace{-20pt}g_n^{(\ell_1)} (p, Y_1,\dots, Y_{k-1}) \\
&\qquad \qquad \qquad \qquad\qquad \qquad\qquad \qquad\qquad \qquad\qquad \qquad \qquad \times g_n^{(\ell_2)} (p, Z_1,\dots, Z_{k-1})  \bigg]. 
\end{split}
\end{align}
Observe that the leading term in the second term of \eqref{e:decomp.var} corresponds to $q = 1$. 
Applying the Mecke formula once more, the expectation (with $q=1$ and generic $\ell_1,\ell_2$) can be bounded, up to   constants, by 
\begin{align*}
&n^{2k-3} \P \Big( \max_{1\le i \le 2k-3}d_H(X_i,p) \le R_n, \, \sum_{i=1}^{2k-3}\one\{U_i\le u\}=\ell_1+\ell_2-1 \Big) \\
&\qquad + n^{2k-3} \P \Big( \max_{1\le i \le 2k-3}d_H(X_i,p) \le R_n, \, \sum_{i=1}^{2k-3}\one\{U_i\le u\}=\ell_1+\ell_2 \Big), 
\end{align*}
where $X_i = (U_i,\Theta_i)$, $i = 1,\dots,2k-3$, are i.i.d.~random variables with density $\radpdf \otimes \pi$, and $U_i = R_n - d_H(o,X_i)$ for $i = 1,\dots,2k-3$.

The remainder of the argument is identical to the proof of Proposition \ref{p:exp.mu.n} $(iii)$, with the obvious replacement of $k-1$ by $2k-3$. Consequently, the second term in \eqref{e:decomp.var} is at most $C e^{\al(d-1)(2k-3)\omega_n} \cdot e^{\half(d-1)(2k-3)u}$ for all $n\ge 1$ and $u\in (0,R_n)$. Combining this with  Proposition \ref{p:exp.mu.n} $(iii)$, we conclude that for all $n \ge 1$ and $u\in (0,R_n)$, 
\begin{align*}
\text{Var}\big( \mD_{k,n}(u) \big) &\le C \Big( e^{\al(d-1)(k-1)\omega_n}\cdot e^{\half(d-1)(k-1)u} + e^{\al(d-1)(2k-3)\omega_n} \cdot e^{\half(d-1)(2k-3)u} \Big) \\
&\le C e^{\al(d-1)(2k-3)\omega_n} \cdot e^{\half(d-1)(2k-3)u}. 
\end{align*}
\end{proof}

\subsection{Point process convergence}  \label{sec:pp.conv.star.shape}

In what follows, let $\ms{PPP}(\nu)$ denote the Poisson point process with intensity measure $\nu$.  For the proof of our main result,  namely Theorem \ref{t:joint.sum.max.star}, we will rely only on \eqref{e:PPmu_n}. Nevertheless, we also present  the corresponding result for $\big( \mD_{k,n}(X_i), \, i=1,\dots,N_n \big)$, as it is of independent interest.  The proof of  \eqref{e:PPD_n} is deferred to Lemma \ref{l:PPD_n} in the Appendix. 

\begin{proposition}  \label{p:pp.conv.star}
Assume $1/2<\al<k-1$; then, as $n\to\infty$, 
\begin{equation}  \label{e:PPmu_n}
\sum_{i=1}^{N_n} \delta_{(i/N_n, \, a_{k,n}^{-1} \mu_{k,n}(U_i))}  \Rightarrow \ms{PPP}\big(\ms{Leb}\otimes \ms{m}_{2\al/(k-1)}\big), \  \ \text{in } M_p\big( [0,1]\times  (0,\infty ] \big), 
\end{equation}
and 
\begin{equation}  \label{e:PPD_n}
\sum_{i=1}^{N_n} \delta_{(i/N_n, \, a_{k,n}^{-1} \mD_{k,n}(X_i))}  \Rightarrow \ms{PPP}\big(\ms{Leb}\otimes \ms{m}_{2\al/(k-1)}\big), \  \ \text{in } M_p\big( [0,1]\times  (0,\infty ] \big), 
\end{equation}
where $\ms{Leb}$ is Lebesgue measure on $[0,1]$, and 
$M_p\big([0,1]\times (0,\infty]  \big)$ denotes the space of point measures in $[0,1]\times (0,\infty]$. 
\end{proposition}

\begin{proof}[Proof of \eqref{e:PPmu_n}]
As before, let $\ga \in \big(\frac{1}{2\al},1\big)$ be a fixed constant. 
We first claim that the behavior of the point process in \eqref{e:PPmu_n} can be approximated by its restricted version: 
\begin{equation}  \label{e:vagueP3}
d_{\ms{v}} \bigg(\sum_{i=1}^{N_n} \delta_{(i/N_n, \, a_{k,n}^{-1} \mu_{k,n}(U_i))}, \  \sum_{i=1}^{N_n}\one \big\{ U_i\le \gaR \big\} \, \delta_{(i/N_n, \, a_{k,n}^{-1} \mu_{k,n}(U_i))}  \bigg) \stackrel{\P}{\to} 0, \ \ \ n\to\infty, 
\end{equation}
where $d_{\ms{v}}$ denotes the  vague metric (see, e.g., Proposition 3.17 in \cite{resnick:1987} for an explicit representation of $d_\ms{v}$). The convergence in \eqref{e:vagueP3} follows from the the following result:   for every $B\subset [0,1]$ and $f \in C_K^+\big( (0,\infty] \big)$, the space of non-negative continuous functions on $(0,\infty]$ with compact support, 
$$
\sum_{i=1}^{N_n}  \one \{ U_i> \gaR \}  \one \{  i/N_n\in B\}f\big( a_{k,n}^{-1}\mu_{k,n}(U_i) \big)\stackrel{\P}{\to} 0. 
$$
To prove this,  we have, for every $\vep>0$, 
\begin{align*}
\P \Big( \sum_{i=1}^{N_n}  \one \{ U_i> \gaR\} f\big( a_{k,n}^{-1}\mu_{k,n}(U_i) \big)\ge \vep \Big) &\le \frac{\|f\|_\infty}{\vep}\, \E \Big[ \sum_{i=1}^{N_n} \one \{ U_i> \gaR \} \Big] \\
&= \frac{\|f\|_\infty n}{\vep}\, \int_{\gaR}^\infty \bar \rho_{n,\al}(u)\dif u \le Cn^{1-2\al\gamma}\to0, \ \ \ n\to\infty. 
\end{align*}

In view of \eqref{e:vagueP3},  we may restrict attention to the truncated version; hence, to complete the entire proof, it remains to show  that, as $n \to \infty$,
\begin{equation}  \label{e:direct.conv.mu_n}
\sum_{i=1}^{N_n} \one \big\{ U_i\le \gaR \big\} \,\delta_{(i/N_n, \, a_{k,n}^{-1} \mu_{k,n}(U_i))}  \Rightarrow \ms{PPP}(\ms{Leb}\otimes \ms{m}_{2\al/(k-1)}), \ \ \text{in } M_P\big( [0,1]\times (0,\infty] \big). 
\end{equation}
For this purpose it suffices to verify  the following  results: as $n \to \infty$,
\begin{equation}  \label{e:conv.to.PPP}
\sum_{i=1}^n  \one \big\{ U_i\le \gaR \big\}\, \delta_{(i/n, \, a_{k,n}^{-1} \mu_{k,n}(U_i))} \Rightarrow \ms{PPP}(\ms{Leb}\otimes \ms{m}_{2\al/(k-1)}), 
\end{equation}
where $U_1,U_2,\dots$ are i.i.d.~random variables with density $\radpdf$, and also,
\begin{equation}  \label{e:de-Poissonization.PP}
d_\ms{v} \bigg(   \sum_{i=1}^n  \one \big\{ U_i\le \gaR\big\}\, \delta_{(i/n, \, a_{k,n}^{-1} \mu_{k,n}(U_i))}, \ \sum_{i=1}^{N_n} \one \big\{ U_i\le  \gaR \big\}\, 
\delta_{(i/N_n, \, a_{k,n}^{-1}\mu_{k,n}(U_i))} \bigg) \stackrel{\P}{\to} 0. 
\end{equation}

For the proof of \eqref{e:conv.to.PPP}, by Corollary 6.1 and Lemma 6.1 in \cite{resnick:2007} it suffices to verify that for every $y>0$, 
\begin{equation}  \label{e:Cor6.1.and.Lemma6.1}
n \P  \big( \mu_{k,n}(U_1)\ge y a_{k,n}, \,  U_1\le \gaR \big) \to \ms{m}_{2\al/(k-1)}\big((y,\infty]\big)  =y^{-2\al/(k-1)}. 
\end{equation}
By Proposition \ref{p:exp.mu.n} $(i)$ and $(ii)$, and using that $s_n$ is a decreasing sequence, we see that for every $\vep>0$ there exists $N\in \bbn$ such that, for all $n\ge N$ and $u\in (0, \gaR)$,
\begin{align}  
&(1-\vep) B_{k,\al}\big( 1-e^{\half (d-1)(1-2\al)u} \big)^{k-1} e^{\half (d-1)(k-1)u} \label{e:mu_n.upper.lower.bdd} \\ 
&\qquad \qquad \le \mu_{k,n}(u)\le (1+\vep) B_{k,\al} \bigg( 1 + \sum_{\ell=0}^{k-2}\binom{k-1}{\ell} (e^{\half (d-1)(1-2\al)u} + s_N)^{k-1-\ell} \bigg) e^{\half (d-1)(k-1)u},  \notag
\end{align}
and $\sum_{\ell=0}^{k-2}\binom{k-1}{\ell}(2s_N)^{k-1-\ell}\le \vep$. Now, 
choose $M>0$ so large that $e^{\half (d-1)(1-2\al)M} \le s_N$; then, 
 the upper bound in \eqref{e:mu_n.upper.lower.bdd}, as well as Lemma \ref{l:Lemma1.OY} $(ii)$, gives that 
\begin{align}
&n \P  \big( \mu_{k,n}(U_1)\ge y a_{k,n}, \,  U_1\le \gaR \big) \label{e:upper.limit.mun} \\
&\le n \P \Big( (1+\vep)^2 \Ckal e^{\half (d-1)(k-1)U_1} \ge ya_{k,n}, \, M\le U_1\le \gaR  \Big) \notag  \\
&\quad + n \P \bigg( (1+\vep)\Big(1+\sum_{\ell=0}^{k-2}\binom{k-1}{\ell} (1+s_N)^{k-1-\ell}\Big) \Ckal e^{\half (d-1)(k-1)U_1} \ge ya_{k,n}, \,  U_1< M  \bigg) \notag \\
&\sim n \P \Big( (1+\vep)^2 \Ckal e^{\half (d-1)(k-1)U_1} \ge ya_{k,n}, \, M\le U_1\le \gaR  \Big) \notag \\
&= n\int_{\frac{2}{(d-1)(k-1)}\log (ya_{k,n} (1+\vep)^{-2}\Ckal^{-1})}^\gaR \radpdf (u)\dif u \notag \\
&\to \Big( \frac{y}{(1+\vep)^2} \Big)^{-2\al/(k-1)}, \ \ \text{as } n\to\infty. \notag 
\end{align}
Letting $\vep\downarrow 0$, we conclude that 
$$
\limsup_{n\to\infty} n \P  \big( \mu_{k,n}(U_1)\ge y a_{k,n}, \, U_1\le \gaR \big)  \le y^{-2\al/(k-1)}. 
$$

Subsequently, 
the lower bound  in  \eqref{e:mu_n.upper.lower.bdd} assures that for all $n\ge N$, 
\begin{align}
&n \P  \big( \mu_{k,n}(U_1)\ge y a_{k,n}, \,  U_1\le \gaR \big) \label{e:liminf.lower.bdd} \\
&\ge n \P \Big( (1-\vep) \Ckal \big( 1-e^{\half (d-1)(1-2\al)U_1} \big)^{k-1} e^{\half (d-1)(k-1)U_1} \ge y a_{k,n}, \, \omega_n \le U_1 \le \gaR  \Big). \notag 
\end{align}
Then,  there exists $N' \ge N$, such  that for all $n\ge N'$, 
$$
1-e^{\half (d-1)(1-2\al)U_1} \ge  1-e^{\half (d-1)(1-2\al)\omega_n}\ge 1-\vep; 
$$
hence, the right-hand side in \eqref{e:liminf.lower.bdd} is  bounded below by
\begin{align*}
&n\P \Big( e^{\half (d-1)(k-1)U_1}\ge \frac{ya_{k,n}}{(1-\vep)^k \Ckal}, \, \omega_n \le U_1 \le \gaR \Big) \\
&= n \P\bigg( \frac{2}{(d-1)(k-1)}\log \Big( \frac{ya_{k,n}}{(1-\vep)^k\Ckal} \Big) \le U_1 \le \gaR\bigg) \to \Big( \frac{y}{(1-\vep)^k} \Big)^{-2\al/(k-1)}, \ \ \text{as } n\to\infty. 
\end{align*}
Thus, by letting $\vep \downarrow 0$, 
$$
\liminf_{n\to\infty} n \P  \big( \mu_{k,n}(U_1)\ge y a_{k,n}, \,  U_1\le \gaR \big)  \ge y^{-2\al/(k-1)}, 
$$
and \eqref{e:Cor6.1.and.Lemma6.1} is obtained. 

Finally, we turn to the proof of \eqref{e:de-Poissonization.PP}. To this aim, we need to demonstrate that for every $B\subset [0,1]$ with $\ms{Leb}(\partial B)=0$ and $f\in C_K^+\big( (0,\infty] \big)$, 
\begin{align*}
&\sum_{i=1}^{n} \one\big\{ U_i\le\gaR \big\}\, \one \{ i/n\in B \} f\big( a_{k,n}^{-1}\mu_{k,n}(U_i) \big)  \\
&\qquad\qquad - \sum_{i=1}^{N_n} \one\big\{ U_i\le \gaR \big\}\, \one \{ i/N_n\in B \} f\big( a_{k,n}^{-1}\mu_{k,n}(U_i) \big)\stackrel{\P}{\to}0, \ \  \ n\to\infty. 
\end{align*}
Denoting $[n]:=\{ 1,\dots,n \}$, we have 
\begin{align}
&\E \bigg[ \, \bigg| \sum_{i=1}^{n} \one\big\{ U_i\le \gaR \big\}\, \one \{ i/n\in B \} f\big( a_{k,n}^{-1}\mu_{k,n}(U_i) \big) \label{e:E_N_n-E_n}\\
&\qquad \qquad \qquad -\sum_{i=1}^{N_n} \one\big\{ U_i\le\gaR \big\}\, \one \{ i/N_n\in B \} f\big( a_{k,n}^{-1}\mu_{k,n}(U_i) \big)\,  \bigg| \, \bigg] \notag \\
&= \sum_{m=0}^\infty \P(N_n=m) \Big|  \#\big\{ i \in [n]\cap nB  \big\}  -\#\big\{ i \in [m]\cap mB  \big\}\Big| \E\Big[ f\big(a_{k,n}^{-1}\mu_{k,n}(U_1)\big) \one \big\{ U_1 \le \gaR \big\} \Big]. \notag 
\end{align}
Since $f$ has compact support, there exists $\eta>0$ such that the support of $f$ is contained in $[\eta, \infty]$. By Proposition \ref{p:exp.mu.n} $(i)$, 
\begin{align*}
\E\Big[ f\big(a_{k,n}^{-1}\mu_{k,n}(U_1)\big) \one \big\{ U_1 \le \gaR \big\} \Big] &\le \|f\|_\infty \P \big( \mu_{k,n}(U_1)\ge \eta a_{k,n}, \, U_1 \le \gaR \big) \\
&\le C\int_{\frac{2}{(d-1)(k-1)}\log (\eta a_{k,n}/C)}^\infty e^{-\al(d-1)u} \dif u \le Cn^{-1}. 
\end{align*}
Substituting this into \eqref{e:E_N_n-E_n} will lead to the upper bound of the form 
\begin{align*}
&Cn^{-1}  \sum_{m=0}^\infty \P(N_n=m) \Big| \#\big\{ i \in [n]\cap nB  \big\}-  \#\big\{ i \in [m]\cap mB  \big\} \Big| \\
&\le Cn^{-1}  \sum_{m=0}^\infty \P(N_n=m) \big|  \#\big\{ i \in [n]\cap nB  \big\}-  n\ms{Leb}(B) \big| \\
&\quad + Cn^{-1}  \sum_{m=0}^\infty \P(N_n=m) |n-m| \, \ms{Leb}(B) +Cn^{-1}  \sum_{m=0}^\infty \P(N_n=m) \big|   m\ms{Leb}(B) -  \#\big\{ i \in [m]\cap mB  \big\} \big|  \\
&=: I_{1,n} +I_{2,n} + I_{3,n}. 
\end{align*}
Clearly, it holds that $I_{1,n} \to 0$ as $n\to\infty$. By the Cauchy-Schwarz inequality, 
$$
I_{2,n} =C\ms{Leb}(B)n^{-1}\E \big[ |N_n-n| \big] \le C\ms{Leb}(B)n^{-1} \sqrt{\text{Var}(N_n)}= C\ms{Leb}(B)n^{-1/2}\to 0, \ \ \text{as } n\to\infty. 
$$
As for $I_{3,n}$, given $\vep>0$, there exists $N\in\bbn$ such that for all $n\ge N$, 
$$
\big|  \ms{Leb}(B) -m^{-1}  \#\big\{ i \in [m]\cap mB  \big\} \big| \le \vep, 
$$
which implies that $I_{3,n}\le o_n(1)+C\vep$.  Now, $\lim_{\vep\downarrow0}\limsup_{n\to\infty}I_{3,n}=0$ follows. 
\end{proof}
\medskip

\subsection{Proof of Theorem \ref{t:joint.sum.max.star}}  \label{sec:main.result.star.count}

For the proof of Theorem \ref{t:joint.sum.max.star}, our next step is to  establish the corresponding results for the version in which $\mD_{k,n}(X_i)$ in Theorem \ref{t:joint.sum.max.star} is replaced by $\mu_{k,n}(U_i)$, where $U_i=R_n-d_H(o,X_i)$.

\begin{proposition}  \label{p:main.mu.version.star}
$(i)$ If $\half (k-1)<\al< k-1$, then as $n\to\infty$, 
\begin{align*}
&\left( \frac{1}{a_{k,n}} \bigg( \sum_{i=1}^{[N_n\cdot\,]}\mu_{k,n}(U_i) - [N_n\cdot\,]  \E [\mu_{k,n}(U_1)] \bigg), \ \frac{1}{a_{k,n}}\bigvee_{i=1}^{[N_n \cdot \,]} \mu_{k,n}(U_i)\right)  \Rightarrow (S_{2\al/(k-1)}(\cdot),\, Y_{2\al/(k-1)}(\cdot)) 
\end{align*}
in the space $D\big( [0,1], \R\times [0,\infty)\big)$, where the weak limit is given in Theorem \ref{t:joint.sum.max.star} $(i)$. \\
$(ii)$ If $\frac{1}{2}<\al<\half (k-1)$, then as $n\to\infty$, 
$$
\left( \frac{1}{a_{k,n}}  \sum_{i=1}^{[N_n \cdot\, ]}\mu_{k,n}(U_i), \ \frac{1}{a_{k,n}} \bigvee_{i=1}^{[N_n \cdot \,]} \mu_{k,n}(U_i) \right)  \Rightarrow (\widetilde{S}_{2\al/(k-1)}(\cdot), Y_\indh(\cdot)), \ \ \text{in } D\big( [0,1], [0,\infty)^2\big), 
$$ 
where the weak limit is given  in Theorem \ref{t:joint.sum.max.star} $(ii)$. 
\end{proposition}

\begin{proof}[Proof of Proposition \ref{p:main.mu.version.star}]
Throughout the proof, write  $\sum_\ell \delta_{(t_\ell, j_\ell)}$ for the  $\ms{PPP}(\ms{Leb}\otimes \ms{m}_{2\al/(k-1)})$ and fix the constant $\ga \in \big(\frac{1}{2\al},1\big)$. 
\vspace{3pt}

\noindent \underline{\textit{Proof of $(i)$}}: Given $\vep > 0$, consider the map $\hat T_\vep: M_p\big( [0,1]\times (0,\infty] \big) \to D\big( [0,1], [0,\infty)^2\big)$ defined by 
$$
\hat T_\vep \Big( \sum_\ell \delta_{(s_\ell,\, z_\ell)} \Big) = \Big( \sum_{\ell: \, s_\ell \le \,\cdot} z_\ell \one \{ z_\ell > \vep \}, \ \bigvee_{\ell:\, s_\ell \le \, \cdot} z_\ell \Big), 
$$ 
which is known to be continuous with respect to the law of $\ms{PPP}(\ms{Leb}\otimes \ms{m}_{2\al/(k-1)})$ (see Section 7.2.3 in \cite{resnick:2007} and Appendix 3.5 in \cite{resnick:1986}). By combining Proposition \ref{p:pp.conv.star} with the continuous mapping theorem, we obtain, for every $\vep > 0$, 
\begin{align}  
&\left( \frac{1}{a_{k,n}}  \sum_{i=1 }^{[N_n\cdot\,]} \mu_{k,n}(U_i)\one \big\{  \mu_{k,n}(U_i)> \vep a_{k,n}\big\}, \ \frac{1}{a_{k,n}} \bigvee_{i=1}^{[N_n\cdot\,]} \mu_{k,n}(U_i) \right) \label{e:main.comp.random} \\[5pt]
&\qquad \Rightarrow \Big( \sum_{\ell:\, t_\ell \le \, \cdot} j_\ell\one \{ j_\ell > \vep \}, \  \bigvee_{\ell:\, t_\ell \le \, \cdot} j_\ell \, \Big), \ \ \text{in } D\big( [0,1], [0,\infty)^2 \big), \ \ \text{as }n\to\infty. \notag 
\end{align}
Furthermore, we have, as $n\to\infty$, 
\begin{equation}  \label{e:sub.comp.exp}
\frac{[N_n  \cdot\,]}{a_{k,n}}\, \E\Big[ \mu_{k,n}(U_1)\one \big\{ \mu_{k,n}(U_1)>\vep a_{k,n} \big\}  \Big] \Rightarrow (\cdot) \int_\vep^\infty x \ms{m}_{2\al/(k-1)}(\dif x), \ \ \text{in } D\big( [0,1], [0,\infty) \big). 
\end{equation}
It follows from \eqref{e:main.comp.random} and \eqref{e:sub.comp.exp} that 
\begin{align*}
&\left( \frac{1}{a_{k,n}}  \sum_{i=1 }^{[N_n\cdot\,]} \mu_{k,n}(U_i)\one \big\{  \mu_{k,n}(U_i)> \vep a_{k,n}\big\} - \frac{[N_n  \cdot\,]}{a_{k,n}} \E\Big[ \mu_{k,n}(U_1)\one \big\{ \mu_{k,n}(U_1)>\vep a_{k,n} \big\}  \Big],  \frac{1}{a_{k,n}} \bigvee_{i=1}^{[N_n\cdot\,]} \mu_{k,n}(U_i) \right) \\[5pt]
&\qquad \Rightarrow \Big(  \sum_{\ell:\, t_\ell \le \, \cdot} j_\ell\one \{ j_\ell > \vep \} -  (\cdot) \int_\vep^\infty x \ms{m}_{2\al/(k-1)}(\dif x), \  \bigvee_{\ell:\, t_\ell \le \, \cdot} j_\ell \, \Big), \ \ \text{in } D\big( [0,1], \R\times [0,\infty) \big). 
\end{align*}
The first component in the above limit  converges weakly to $S_{2\al/(k-1)}(\cdot)$ as $\vep \downarrow 0$, where, for each $t\ge0$, 
\begin{equation}  \label{e:chf.first.comp.zero.mean.stable}
\E \big[ e^{iv S_{2\al/(k-1)}(t)} \big] = \exp\Big\{ t\int_0^\infty (e^{ivx}-1-ivx) \ms{m}_{2\al/(k-1)} (\dif x)  \Big\}, \ \ \ v\in \R 
\end{equation}
(see, e.g.,  Section 5.5 in \cite{resnick:2007}). 
Furthermore, it holds that $\bigvee_{\ell:\, t_\ell \le \, \cdot} j_\ell \stackrel{d}{=}Y_\indh(\cdot)$. Indeed, for  $0\equiv s_0<s_1<\cdots <s_m<\infty$ and $z_i\ge 0$, $i=1,\dots,m$, 
\begin{align*}
\P \Big( \bigvee_{\ell:\, t_\ell \le s_i} j_\ell\le z_i, \, i=1,\dots,m  \Big) &= \P \Big( \sum_\ell \delta_{(t_\ell, j_\ell)} \big(  (0,s_i]\times (z_i,\infty]\big)=0, \, i=1,\dots,m \Big)\\
&=\prod_{j=1}^m \P  \Big( \sum_\ell \delta_{(t_\ell, j_\ell)} \big(  (s_{j-1},s_j]\times (z_j \wedge \cdots \wedge z_m,\infty]\big)=0 \Big) \\
&=\exp \Big\{ -\sum_{j=1}^m (s_j-s_{j-1}) (z_j \wedge \cdots \wedge z_m)^{-2\al/(k-1)} \Big\}\\
&= \P \big( Y_{2\al/(k-1)}(s_i)\le z_i, \, i=1,\dots,m \big). 
\end{align*}
In order to characterize  the joint law of $(S_\indh(t), Y_\indh(t))$, one can see that for $v\in \R$ and  $z\ge 0$, 
\begin{align*}
&\E \big[ e^{ivS_\indh(t)} \one \{ Y_\indh(t)\le z\} \big] \\
&= \E \Big[ e^{ivY_\indh(t)} \one \{ Y_\indh(t) \le z\} \E \Big( e^{iv(S_\indh(t)-Y_\indh(t))} \Big|  Y_\indh(t)\Big) \Big]. 
\end{align*}
It then follows from   \cite[Proposition 2.3]{resnick:1986} and \eqref{e:chf.first.comp.zero.mean.stable} that 
$$
\E \Big( e^{iv(S_\indh(t)-Y_\indh(t))} \Big|  Y_\indh(t)\Big) = g(v, Y_\indh(t))^t, 
$$
and therefore, 
\begin{align*}
\E \big[ e^{ivS_\indh(t)} \one \{ Y_\indh(t) \le z\} \big] &= \int_0^z e^{ivy}\, g(v,y)^t\, \P \big(Y_\indh(t)\in \dif y\big) \\
&= t\int_0^z e^{ivy} g(v,y)^t e^{-ty^{-2\al/(k-1)}} \ms{m}_\indh (\dif y), 
\end{align*}
yielding \eqref{e:hybrid1}. 

To finish the proof, it remains to verify that, for every $\eta \in (0,1)$,
\begin{align}
&\lim_{\vep\to 0}\limsup_{n\to\infty}\P \bigg(\sup_{0\le t \le 1} \bigg| \sum_{i=1 }^{[N_nt]}\mu_{k,n}(U_i)\one \big\{  \mu_{k,n}(U_i)\le  \vep a_{k,n}\big\}\label{e:prop3.5(1)} \\
&\qquad \qquad \qquad \qquad \qquad \qquad  -[N_nt] \E \Big[  \mu_{k,n}(U_1)\one \big\{  \mu_{k,n}(U_1)\le \vep a_{k,n}\big\} \Big] \bigg|>\eta a_{k,n} \bigg)=0. \notag
\end{align}
Since $\big( \mu_{k,n}(U_i) \big)_{i\ge1}$ are i.i.d.~random variables,  Kolmogorov's maximal inequality ensures that the last probability can be bounded above by 
\begin{align}
&\E\bigg[ \P \bigg(\max_{0\le j \le N_n} \bigg| \sum_{i=1 }^j \Big\{ \mu_{k,n}(U_i)\one \big\{  \mu_{k,n}(U_i)\le  \vep a_{k,n}\big\}  - \E \Big[  \mu_{k,n}(U_i)\one \big\{  \mu_{k,n}(U_i)\le \vep a_{k,n}\big\} \Big]\Big\} \bigg|>\eta a_{k,n} \bigg|  N_n \bigg) \bigg] \label{e:chebyshev} \\ 
&\le \frac{1}{\eta^2 a_{k,n}^2}\, \E\bigg[ \text{Var}\Big( \sum_{i=1}^{N_n} \mu_{k,n}(U_i)\one \big\{  \mu_{k,n}(U_i)\le \vep a_{k,n} \big\} \, \Big|  N_n \Big) \bigg] \notag  \\ 
&\le \frac{1}{\eta^2 a_{k,n}^2}\,\text{Var}\Big( \sum_{i=1}^{N_n} \mu_{k,n}(U_i)\one \big\{  \mu_{k,n}(U_i)\le \vep a_{k,n} \big\} \Big) \notag  \\
&=\frac{n}{\eta^2 a_{k,n}^2}\, \int_{0}^{R_n} \mu_{k,n}(u)^2 \one \big\{ \mu_{k,n}(u) \le \vep a_{k,n} \big\} \radpdf (u) \dif u. \notag
\end{align}
Clearly, $\mu_{k,n}(u)$ is strictly increasing in $u$, and by Proposition \ref{p:exp.mu.n} $(ii)$, we have $\mu_{k,n}(\gaR)\ge \vep a_{k,n}$ for large enough $n$. Therefore, $\mu_{k,n}(u)\le \vep a_{k,n}$ immediately implies $u \le  \gaR$. Consequently, the upper limit of the last integral in \eqref{e:chebyshev} can be restricted from $R_n$ to $\gaR$. Applying Proposition \ref{p:exp.mu.n} $(i)$ and Lemma \ref{l:Lemma1.OY}, the last expression in \eqref{e:chebyshev} is at most
\begin{align}
&\frac{Cn}{\eta^2 a_{k,n}^2}\, \int_0^\gaR e^{(d-1)(k-1-\al)u}\one \big\{ \mu_{k,n}(u) \le \vep a_{k,n} \big\}\dif u \label{e:var.2nd.momemnt} \\
&=\frac{Cn}{\eta^2 a_{k,n}^2}\, \int_{\vep_0}^\gaR e^{(d-1)(k-1-\al)u}\one \big\{ \mu_{k,n}(u) \le \vep a_{k,n} \big\}\dif u + o_n(1), \notag 
\end{align}
where $\vep_0\in (0,1)$. 
It follows from Proposition \ref{p:exp.mu.n} $(ii)$ that there exists $N\in \bbn$ such that for all $n\ge N$ and $\vep_0\le u \le \gaR$, 
$$
\mu_{k,n}(u) \ge (1-\eta) \Ckal \big( 1-e^{\half (d-1)(1-2\al)\vep_0} \big)^{k-1} e^{\half (d-1)(k-1)u}. 
$$
Using this lower bound, the main term in \eqref{e:var.2nd.momemnt} is further bounded from above by 
\begin{equation}  \label{e:lower.mu.n.vep0}
\frac{Cn}{\eta^2 a_{k,n}^2}\, \int_{\vep_0}^\gaR e^{(d-1)(k-1-\al)u} \one \Big\{ e^{\half (d-1)(k-1)u} \le \frac{\vep a_{k,n}}{C_1} \Big\} \dif u, 
\end{equation}
where $C_1:= (1-\eta)\Ckal \big( 1-e^{\half (d-1)(1-2\al)\vep_0} \big)^{k-1}$. Finally, one can bound \eqref{e:lower.mu.n.vep0} from above  by 
\begin{align*}
\frac{Cn}{\eta^2 a_{k,n}^2}\,  \int_0^{\frac{2}{(d-1)(k-1)}\log (\vep C_1^{-1} a_{k,n})} e^{(d-1)(k-1-\al)u} \dif u \le \frac{Cn^{1-\frac{k-1}{\al}}}{\eta^2}\, \Big( \frac{\vep a_{k,n}}{C_1} \Big)^{\frac{2(k-1-\al)}{k-1}} = C \vep^{\frac{2(k-1-\al)}{k-1}}. 
\end{align*}
Since the final term tends to $0$ as $\vep \downarrow 0$,  \eqref{e:prop3.5(1)} has been established. 
\medskip

\noindent \underline{\textit{Proof of $(ii)$}}: By Proposition \ref{p:pp.conv.star} and the continuous mapping theorem based on the map $\hat T_\vep$ as before, we obtain, once again, that as $n \to \infty$,  
\begin{align*}
\left( \frac{1}{a_{k,n}} \sum_{i=1 }^{[N_n \cdot\,]}\mu_{k,n}(U_i)\one \big\{  \mu_{k,n}(U_i)> \vep a_{k,n}\big\}, \ \frac{1}{a_{k,n}}\bigvee_{i=1}^{[N_n \cdot\, ]} \mu_{k,n}(U_i)\right) \Rightarrow \Big( \sum_{\ell:\, t_\ell \le \, \cdot} j_\ell \one \{  j_\ell > \vep\}, \ \bigvee_{\ell:\, t_\ell \le \,\cdot } j_\ell\Big).  
\end{align*}
In this limit, the first component converges weakly to $\widetilde{S}_{2\al/(k-1)}(\cdot)$ as $\vep   \downarrow 0$, where
\begin{equation}  \label{e:chf.first.comp.zero.stable.subordinator}
\E \big[ e^{iv \widetilde{S}_{2\al/(k-1)}(t)} \big] = \exp\Big\{t \int_0^\infty (e^{ivx}-1) \ms{m}_{2\al/(k-1)} (\dif x)  \Big\}, \ \ \ v\in \R 
\end{equation}
(see e.g.,  Section 5.5 in \cite{resnick:2007}). The second component has the same distribution as $Y_\indh(\cdot)$. 
Moreover, the hybrid characteristic-distribution function \eqref{e:hybrid2} can again be derived from \cite[Proposition 2.3]{resnick:1986} and \eqref{e:chf.first.comp.zero.stable.subordinator}, by carrying out the same calculations as in case $(i)$.

To conclude the proof, it remains to demonstrate that for every $\eta\in (0,1)$, 
\begin{equation}  \label{e:2nd.conv.subordinator}
\lim_{\vep\to 0}\limsup_{n\to\infty}\P \bigg( \sup_{0\le t \le 1}\sum_{i=1}^{[N_n t]}\mu_{k,n}(U_i)\one \big\{  \mu_{k,n}(U_i)\le  \vep a_{k,n}\big\} > \eta a_{k,n} \bigg)=0. 
\end{equation}
By Markov's inequality, the   probability above is at most 
\begin{align*}
&\frac{1}{\eta a_{k,n}}\, \E\Big[\sum_{i=1}^{N_n}  \mu_{k,n}(U_i) \one \big\{ \mu_{k,n}(U_i)\le \vep a_{k,n} \big\}  \Big]  = \frac{n}{\eta a_{k,n}}\, \int_{0}^{R_n} \mu_{k,n}(u)\one \big\{ \mu_{k,n}(u)\le \vep a_{k,n} \big\} \radpdf (u)\dif u. 
\end{align*}
Proceeding as in the proof of case $(i)$, this is further bounded above by 
$C\vep^{\frac{k-1-2\al}{k-1}} + o_n(1)$, which clearly satisfies $\lim_{\vep \downarrow0} \limsup_{n\to\infty} \big( C\vep^{\frac{k-1-2\al}{k-1}} + o_n(1) \big)=0$. 
\end{proof}

By Proposition \ref{p:main.mu.version.star}, the proof of Theorem \ref{t:joint.sum.max.star} reduces to establishing the following two lemmas.

\begin{lemma}  \label{l:marginal.max.star}
Let $\half <\al < k-1$; then, as $n\to\infty$, 
$$
\sup_{0\le t\le 1}\bigg| a_{k,n}^{-1}\bigvee_{i=1}^{[N_n t]} \mD_{k,n}(X_i) - a_{k,n}^{-1} \bigvee_{i=1}^{[N_n t]} \mu_{k,n}(U_i) \bigg| \stackrel{\P}{\to} 0. 
$$
\end{lemma}

\begin{lemma}  \label{l:stable.limit.thm.D}
$(i)$ If $\half (k-1)<\al< k-1$, then as $n\to\infty$, it holds that 
$$
\sup_{0\le t\le 1}\bigg|  \frac{1}{a_{k,n}} \Big( \sum_{i=1}^{[N_n t]}\mD_{k,n}(X_i) -[N_n t] \E \big[\mD_{k,n}(X_1) \big]\Big) - \frac{1}{a_{k,n}} \Big( \sum_{i=1}^{[N_n t]}\mu_{k,n}(U_i) -[N_n t]  \E \big[  \mu_{k,n}(U_1) \big]\Big)\bigg| \stackrel{\P}{\to} 0. 
$$
$(ii)$ If $\frac{1}{2}<\al<\half (k-1)$, then as $n\to\infty$, 
$$
\sup_{0\le t\le 1}\bigg|  \frac{1}{a_{k,n}}  \sum_{i=1}^{[N_n t]}\mD_{k,n}(X_i) - \frac{1}{a_{k,n}} \sum_{i=1}^{[N_n t]}\mu_{k,n}(U_i) \bigg|\stackrel{\P}{\to} 0. 
$$
\end{lemma}

\begin{proof}[Proof of Lemma \ref{l:marginal.max.star}]
Note that for every $\vep>0$, 
\begin{align}
&\P \bigg( \sup_{0\le t \le 1}\Big| \bigvee_{i=1}^{[N_nt]}\mD_{k,n}(X_i)  -\bigvee_{i=1}^{[N_nt]}\mu_{k,n}(U_i)\Big| >\vep a_{k,n}  \bigg) \le \P \bigg( \bigcup_{i=1}^{N_n} \Big\{ \big| \mD_{k,n}(X_i) - \mu_{k,n}(U_i)  \big| >\vep a_{k,n} \Big\}\bigg) \label{e:max.difference}\\
&\le \E \Big[ \sum_{i=1}^{N_n} \one \Big\{  \big|  \mD_{k,n}(X_i)-\mu_{k,n}(U_i) \big| >\vep a_{k,n} \Big\} \Big]  \notag \\
&=n\int_0^{R_n} \P \Big( \big| \mD_{k,n}(u)-\mu_{k,n}(u) \big|>\vep a_{k,n} \Big) \radpdf (u) \dif u \to 0, \ \ \ n\to\infty, \notag 
\end{align}
where the last  convergence can be established in the same way as  \eqref{e:A_n.to.0.second} in the Appendix. 
\end{proof}

For the proof of Lemma \ref{l:stable.limit.thm.D}, we begin with case $(ii)$, since its argument is considerably shorter than that of case $(i)$.

\begin{proof}[Proof of Lemma \ref{l:stable.limit.thm.D} $(ii)$]
For any $\vep >0$, 
\begin{equation}  \label{e:diff.D.mu.conv.in.prob}
\P\bigg(  \sup_{0\le t \le 1} \Big| \sum_{i=1}^{[N_n t]} \mD_{k,n}(X_i) - \sum_{i=1}^{[N_n t]} \mu_{k,n}(U_i) \Big| >\vep a_{k,n}\bigg) \le \P \bigg( \sum_{i=1}^{N_n} \big| \mD_{k,n}(X_i) - \mu_{k,n}(U_i) \big| >\vep a_{k,n}\bigg)
\end{equation}
Then, \eqref{e:diff.D.mu.conv.in.prob} is further bounded above by 
\begin{align*}
&\P \bigg( \bigcup_{i=1}^{N_n} \Big\{ \big| \mD_{k,n}(X_i)-\mu_{k,n}(U_i) \big| \ge  a_{k,n} \Big\} \bigg)  \\
&\qquad + \P \bigg(  \sum_{i=1}^{N_n} \big| \mD_{k,n}(X_i)-\mu_{k,n}(U_i) \big|\, \one \Big\{  \big| \mD_{k,n}(X_i)-\mu_{k,n}(U_i) \big| \le a_{k,n} \Big\} >\vep a_{k,n} \bigg) \\
&=: C_n + D_n. 
\end{align*}
From the proof of Lemma \ref{l:marginal.max.star}, we already know  that  $C_n\to0$ as $n\to\infty$. 

By the Markov and Chebyshev inequalities, 
\begin{align*}
D_n &\le \frac{1}{\vep a_{k,n}}\, \E \Big[ \sum_{i=1}^{N_n} \big| \mD_{k,n}(X_i)-\mu_{k,n}(U_i) \big| \, \one \Big\{ \big| \mD_{k,n}(X_i)-\mu_{k,n}(U_i) \big| \le a_{k,n} \Big\}\,\Big] \\
&\le \frac{n}{\vep a_{k,n}}\, \int_0^{\frac{2}{(d-1)(k-1)}\log (Ma_{k,n})} \sqrt{\text{Var}\big( \mD_{k,n}(u) \big)} \, \radpdf (u)\dif u + \frac{n}{\vep}\, \int_{\frac{2}{(d-1)(k-1)}\log (Ma_{k,n})}^\infty \radpdf (u)\dif u \\
&\le Cn^{1-\frac{k-1}{2\al}} e^{\al(d-1)(k-\frac{3}{2})\omega_n} \int_0^{\frac{2}{(d-1)(k-1)}\log (Ma_{k,n})} e^{\half (d-1)(k-\frac{3}{2}-2\al)u}\dif u + CM^{-\frac{2\al}{k-1}}. 
\end{align*}
If $k-\frac{3}{2}-2\al\le 0$, then $D_n\le Cn^{1-\frac{k-1}{2\al}}e^{\al(d-1)(k-\frac{3}{2})\omega_n} \log n + CM^{-\frac{2\al}{k-1}}$, and if $k-\frac{3}{2}-2\al>0$, then $D_n \le CM^{\frac{k-3/2-2\al}{k-1}}e^{\al(d-1)(k-\frac{3}{2})\omega_n}n^{-1/(4\al)} + CM^{-\frac{2\al}{k-1}}$. In both cases, we have $\lim_{M\to\infty}\limsup_{n\to\infty}D_n=0$ as desired. 
\end{proof}

\begin{proof}[Proof of Lemma \ref{l:stable.limit.thm.D} $(i)$]
Choose a constant $c_1$ satisfying 
\begin{equation}  \label{e:constraint.c1}
0 < \frac{2\al-k+1}{2\al (2\al-k+3/2)} < c_1 < \frac{1}{2\al}. 
\end{equation}
For our purposes, it suffices to establish that for every $\eta>0$, 
\begin{align*}
\P \bigg( \sup_{0\le t \le 1} \Big|  \sum_{i=1}^{[N_nt]} \big( \mD_{k,n}(X_i) - \E [\mD_{k,n}(X_i)] \big) - \sum_{i=1}^{[N_nt]} \big( \mu_{k,n}(U_i) - \E [\mu_{k,n}(U_i)] \big) \Big| >4\eta a_{k,n} \bigg)\to0, \  \ n\to\infty. 
\end{align*}
This can be obtained by proving the following results: 
\begin{equation}   \label{e:main(i).cond1}
\P \bigg( \sup_{0\le t \le 1} \bigg| \sum_{i=1}^{[N_nt]} \big(\mD_{k,n}(X_i) -\mu_{k,n}(U_i) \big)\one \{ U_i >c_1 R_n \} \bigg|>\eta a_{k,n} \bigg) \to 0, 
\end{equation}
\begin{equation}  \label{e:main(i).cond2}
\P \bigg( \sup_{0\le t \le 1} [N_nt] \Big| \, \E \big[ \mD_{k,n}(X_1)\one\{ U_1>c_1 R_n \} \big] - \E \big[ \mu_{k,n}(U_1)\one\{ U_1>c_1 R_n \} \big] \, \Big| >\eta a_{k,n} \bigg) \to 0, 
\end{equation}
\begin{equation}  \label{e:main(i).cond3}
\P \bigg( \sup_{0\le t \le 1} \bigg| \sum_{i=1}^{[N_nt]} \Big(\mu_{k,n}(U_i)\one \{ U_i \le c_1 R_n \} -\E \big[ \mu_{k,n}(U_i)\one \{ U_i \le c_1 R_n \} \big] \Big) \bigg|>\eta a_{k,n} \bigg) \to 0, 
\end{equation}
and
\begin{equation}  \label{e:main(i).cond4}
\P \bigg( \sup_{0\le t \le 1} \bigg| \sum_{i=1}^{[N_nt]} \Big(\mD_{k,n}(X_i)\one \{ U_i \le c_1 R_n \} -\E \big[ \mD_{k,n}(X_i)\one \{ U_i \le c_1 R_n \} \big] \Big) \bigg|>\eta a_{k,n} \bigg) \to 0. 
\end{equation}

For the proof of \eqref{e:main(i).cond1}, it follows from the Markov and Chebyshev  inequalities, along with Proposition \ref{p:var.mu.n}, that
\begin{align}
&\P \bigg( \sup_{0\le t \le 1} \bigg| \sum_{i=1}^{[N_nt]} \big(\mD_{k,n}(X_i) -\mu_{k,n}(U_i) \big)\one \{ U_i >c_1 R_n \} \bigg|>\eta a_{k,n} \bigg)  \label{e:diff.big.u}\\
&\le \frac{1}{\eta a_{k,n}}\E \bigg[   \sum_{i=1}^{N_n} \big| \mD_{k,n}(X_i)-\mu_{k,n}(U_i) \big|\, \one \{ U_i>c_1 R_n \} \bigg]\notag  \\
&= \frac{n}{\eta a_{k,n}} \, \int_{c_1R_n}^{R_n} \E \Big[ \big| \mD_{k,n}(u)-\mu_{k,n}(u) \big| \Big] \radpdf (u)\dif u\notag  \\
&\le Cn^{1-\frac{k-1}{2\al}}\, \int_{c_1R_n}^{R_n} \sqrt{\text{Var}\big(\mD_{k,n}(u)  \big)} \, e^{-\al(d-1)u}\dif u \notag \\
&\le Cn^{1-\frac{k-1}{2\al}} e^{\al(d-1)(k-\frac{3}{2})\omega_n} \, \int_{c_1R_n}^{R_n} e^{\half (d-1)(k-\frac{3}{2}-2\al)u} \dif u \notag \\
&\le Cn^{\frac{2\al - k+1}{2\al} - c_1(2\al-k+\frac{3}{2})}e^{\al(d-1)(k-\frac{3}{2})\omega_n}  \to 0, \ \ \text{as } n\to\infty, \notag 
\end{align}
where the last convergence is due to the constraint  in \eqref{e:constraint.c1}. 

Next, we turn to \eqref{e:main(i).cond2} and define $X_1'=(U_1',\Theta_1')$ (with $U_1'=R_n-d_H(o,X_1')$) as  an independent copy of $X_1=(U_1,\Theta_1)$ that is also independent of $\Pn$. Throughout the proof of \eqref{e:main(i).cond2}, we write $\mD_{k,n}(\cdot) = \mD_{k,n}(\cdot, \Pn)$ to emphasize the dependence of $\mD_{k,n}$ on the process $\Pn$. Note that \eqref{e:main(i).cond2} follows if one can show that 
\begin{equation}  \label{e:main(i).cond2-1}
\frac{n}{a_{k,n}}\, \Big|\,\E \big[ \mD_{k,n}(X_1', \Pn) \one \{ U_1' >c_1 R_n\}  \big] - \E \big[ \mu_{k,n}(U_1) \one \{ U_1 >c_1 R_n\}  \big] \Big|\to 0, \ \ \ n\to\infty, 
\end{equation}
and 
\begin{equation}  \label{e:main(i).cond2-2}
\frac{n}{a_{k,n}}\, \Big|\,\E \big[ \mD_{k,n}(X_1', \Pn) \one \{ U_1' >c_1 R_n\}  \big] - \E \big[ \mD_{k,n}(X_1, \Pn) \one \{ U_1 >c_1 R_n\}  \big] \Big|\to 0, \ \ \ n\to\infty. 
\end{equation}
As for \eqref{e:main(i).cond2-1}, since $U_1'$ is an independent copy of $U_1$, 
\begin{align*}
&\frac{n}{a_{k,n}}\, \Big|\,\E \big[ \mD_{k,n}(X_1', \Pn) \one \{ U_1' >c_1 R_n\}  \big] - \E \big[ \mu_{k,n}(U_1) \one \{ U_1 >c_1 R_n\}  \big] \Big| \\
&\le \frac{n}{a_{k,n}}\,  \E \Big[ \,\big| \mD_{k,n}(X_1', \Pn) - \mu_{k,n}(U_1') \big|\,\one \{U_1' >c_1 R_n\}  \Big]\\
&=\frac{n}{a_{k,n}}\, \int_{c_1R_n}^{R_n} \E \Big[ \big| \mD_{k,n}(u)-\mu_{k,n}(u) \big| \Big]\radpdf (u)\dif u. 
\end{align*}
We know from \eqref{e:diff.big.u} that the final term vanishes as $n\to\infty$. 

For the proof of \eqref{e:main(i).cond2-2}, let us  write 
$$
\E \big[ \mD_{k,n}(X_1', \Pn) \one \{ U_1' >c_1 R_n\}  \big] = \int_{c_1R_n}^{R_n} \E \big[ \mD_{k,n}(u,\Pn) \big]\radpdf (u)\dif u, 
$$
and 
$$
\E \big[ \mD_{k,n}(X_1, \Pn) \one \{ U_1 >c_1 R_n\}  \big] = \int_{c_1R_n}^{R_n} \E \big[ \mD_{k,n}(u,\Pn\setminus \{ X_1\}) \big]\radpdf (u)\dif u, 
$$
while noting that by definition, $\mD_{k,n}(u,\Pn)\ge \mD_{k,n}\big( u,\Pn\setminus \{ X_1 \} \big)$ holds. Hence, we have 
\begin{align}
&\frac{n}{a_{k,n}}\, \Big|\,\E \big[ \mD_{k,n}(X_1', \Pn) \one \{ U_1' >c_1 R_n\}  \big] - \E \big[ \mD_{k,n}(X_1, \Pn) \one \{ U_1 >c_1 R_n\}  \big] \Big| \label{e:diff.Poissonization}\\
&= \frac{n}{a_{k,n}}\, \int_{c_1R_n}^{R_n} \E \big[ \mD_{k,n}(u,\Pn) - \mD_{k,n}(u, \Pn\setminus \{X_1\}) \big]\radpdf (u)\dif u.\notag  
\end{align}
Denote the difference in the last term as 
\begin{align}
&\mD_{k,n}(u,\Pn) - \mD_{k,n}(u, \Pn\setminus \{X_1\})   \label{e:difference.D.X1} \\
&= (k-1)\one \big\{  d_H(X_1,p)\le R_n  \big\}\hspace{-15pt}\sum_{(Y_1,\dots, Y_{k-2})\in (\{ X_2,\dots,X_{N_n} \})_{\neq}^{k-2}} \hspace{-15pt}\one \{ \max_{1\le i \le k-2}d_H(Y_i, p)\le R_n \}, \notag 
\end{align}
where $p=(u,{\bf 0})$ with $u=R_n-d_H(o,p)$. Since $X_1$ is independent of $\Pn\setminus \{ X_1 \}$ and $\{ X_2,\dots, X_{N_n} \}$ is contained in $\Pn$, we  derive that 
\begin{align*}
&\E \big[ \mD_{k,n}(u,\Pn) - \mD_{k,n}(u, \Pn\setminus \{X_1\}) \big]\\
&\le (k-1)\P\big(  d_H(X_1,p)\le R_n \big) \E\bigg[ \sum_{(Y_1,\dots, Y_{k-2})\in (\Pn)_{\neq}^{k-2}} \hspace{-15pt}\one \{ \max_{1 \le i \le k-2}d_H(Y_i, p)\le R_n \} \bigg]. 
\end{align*}
By the proof of Proposition \ref{p:exp.mu.n} $(iii)$,  it holds that 
$\P\big(  d_H(X_1,p)\le R_n \big) \le Cn^{-1} e^{\al(d-1)\omega_n} \cdot e^{\half (d-1)u}$, 
and also 
$$
 \E\bigg[ \sum_{(Y_1,\dots, Y_{k-2})\in (\Pn)_{\neq}^{k-2}} \hspace{-15pt}\one \{ \max_{1\le i \le k-2}d_H(Y_i, p)\le R_n \} \bigg]\le C e^{\al(d-1)(k-2)\omega_n}\cdot e^{\half (d-1)(k-2)u}. 
$$
Referring these bounds back to \eqref{e:diff.Poissonization}, 
\begin{align*}
&\frac{n}{a_{k,n}}\, \int_{c_1R_n}^{R_n} \E \big[ \mD_{k,n}(u,\Pn) - \mD_{k,n}(u, \Pn\setminus \{X_1\}) \big]\radpdf (u)\dif u \\
&\le \frac{Cn}{a_{k,n}}\, n^{-1}e^{\al(d-1)(k-1)\omega_n} \int_{c_1R_n}^{R_n} e^{\half (d-1)(k-1-2\al)u}\dif u \\
&\le Ca_{k,n}^{-1}e^{\al(d-1)(k-1)\omega_n}  \to 0, \ \ \text{as } n\to\infty. 
\end{align*}

We now address \eqref{e:main(i).cond3}. Proceeding as in \eqref{e:chebyshev} and applying Kolmogorov's inequality, the probability in \eqref{e:main(i).cond3} can be bounded above by
\begin{align*}
&\frac{1}{\eta^2 a_{k,n}^2} \text{Var} \Big( \sum_{i=1}^{N_n} \mu_{k,n}(U_i) \one\{  U_i\le c_1R_n\}\Big) = \frac{n}{\eta^2 a_{k,n}^2}\, \int_{0}^{c_1R_n} \mu_{k,n}(u)^2 \radpdf (u)\dif u \\
&\qquad \le \frac{Cn}{a_{k,n}^2}\, \int_0^{c_1R_n} e^{(d-1)(k-1-\al)u}\dif u \le Cn^{-(k-1-\al)(1/\al - 2c_1)} \to 0,  \ \ \text{as } n\to\infty. 
\end{align*}

Finally, the proof of \eqref{e:main(i).cond4} is highly technical and  postponed to Lemma \ref{l:Pruss.maximal.inequ} in the Appendix.  
\end{proof}

\section{Clique counts}   \label{sec:clique.counts}

We consider an $m$-clique, $m \ge 3$,  formed by the vertices $X_i = (U_i, \Theta_i)$, $i = 1, \dots, m-1$, and $p=(u,\btheta) \in B(o, R_n)$, where $U_i = R_n - d_H(o, X_i)$ for $i = 1, \dots, m-1$, and $u =R_n-d_H(o, p)$, subject to the   ordering $U_1 \le \cdots \le U_{m-1} \le u$. Specifically, we define the count of $m$-cliques in which the point $p$ is the vertex closest to the origin of $\B_d$: 
\begin{align*}
&\mC_{m,n}(p) := \sum_{(Y_1,\dots,Y_{m-1})\in (\Pn)_{\neq}^{m-1}} \one \big\{ \max_{1\le i, j \le m-1}d_H(Y_i,Y_j)\le R_n, \\
&\qquad \qquad \qquad \qquad \qquad \qquad\qquad\quad  \max_{1\le i \le m-1}d_H(Y_i,p)\le R_n,  \, V_1 \le \cdots \le V_{m-1} \le u \big\}, 
\end{align*}
where $V_i=R_n-d_H(o,Y_i)$ for $i=1,\dots,m-1$. 
For $u\in (0,R_n)$, denote $\mC_{m,n}(u):=\mC_{m,n}(u,{\bf 0})$ with ${\bf 0}\in \mA_d$, and define $\nu_{m,n}(u):= \E \big[ \mC_{m,n}(u,\btheta) \big]$. Note  that $\nu_{m,n}(u)$ is independent of $\btheta$ by the uniformity of an underlying density. Define also 
$$
C_{m,\al} := 2^{d-1} \xi_{d-1} \big( s_{d-1}^{-1}\al (d-1) \big)^{m-1}\ell(\infty), 
$$
where $\xi_{d-1}$ denotes the volume of $(d-1)$-dimensional unit ball, $s_{d-1}$ is the surface area of the $d$-dimensional unit ball in $\R^d$, and $\ell(\infty):=\lim_{u\to\infty}\ell(u)$ is  defined  in \eqref{e:def.ell} below. 

Our main result for the clique counts is, once again, the joint weak convergence of the process-level clique counts $\sum_{i=1}^{[N_nt]} \mC_{m,n}(X_i)$,  representing the total number of $m$-cliques generated by $\Pn$,  together with their maxima process $\bigvee_{i=1}^{[N_nt]}\mC_{m,n}(X_i)$. We introduce  the scaling constants   by $b_{m,n} := C_{m,\al} n^{1/(2\al)}$. 

\begin{theorem}  \label{t:joint.sum.max.clique}
Let $\frac{2m-3}{2m-2}<\al<1$. Then, as $n\to\infty$, 
\begin{align}
&\left( \frac{1}{b_{m,n}} \bigg( \sum_{i=1}^{[N_n\cdot\,]}\mC_{m,n}(X_i) - [N_n\cdot\,]  \E [\mC_{m,n}(X_1)] \bigg), \ \frac{1}{b_{m,n}}\bigvee_{i=1}^{[N_n \cdot \,]} \mC_{m,n}(X_i)\right)  \Rightarrow (S_{2\al}(\cdot),\, Y_{2\al}(\cdot)),  \label{e:joint.weak.clique}
\end{align}
in the space $D\big([0,1], \R \times [0,\infty)\big)$, where the hybrid characteristic-distribution function is given in \eqref{e:hybrid1} with $k=2$. 
\end{theorem}
Since the first component of the weak limit in \eqref{e:joint.weak.clique} is a $2\al$-stable L\'evy process, the upper bound $\al<1$ is optimal. In contrast, the lower bound $\frac{2m-3}{2m-2}$ marks a phase transition, and the remaining regime $\frac{1}{2} <\al< \frac{2m-3}{2m-2}$ is left as an open problem; see Remark \ref{rem:open.q.clique} below.

The structure of the proof of Theorem \ref{t:joint.sum.max.clique}  parallels that of Theorem \ref{t:joint.sum.max.star}.  First, we compute the moments of $\mC_{m,n}(u)$ in Section \ref{sec:moments.clique}, and then establish the corresponding point process convergence in Section \ref{sec:pp.conv.clique}. Subsequently, the weak convergence for the clique counts  and their maxima  associated with $\big( b_{k,n}^{-1}\nu_{k,n}(U_i), \, i=1,\dots,N_n\big)$ is derived in \eqref{e:join.weak.nu.version}.  Finally, Lemmas \ref{l:marginal.max.clique} and  \ref{l:stable.limit.thm.C} prove  the required negligibility between the results for $\big( b_{k,n}^{-1}\nu_{k,n}(U_i), \, i=1,\dots,N_n\big)$ and those for $\big( b_{k,n}^{-1}\mC_{k,n}(X_i), \, i=1,\dots,N_n\big)$. 

In our  analysis,  the main technical effort lies in establishing the moment estimates. The connectivity structure of cliques is substantially more complicated  than that of star shapes, and we must carefully identify which internal connections give the dominant contribution and which are negligible.

\subsection{Moments}  \label{sec:moments.clique}

We begin with the result on $\nu_{k,n}(u)$ for $u\in (0,R_n)$. 

\begin{proposition}  \label{p:exp.clique}
Let $\frac{2m-3}{2m-2} < \al < 1$ and $\gamma \in (\frac{1}{2\al}, 1)$. \\
$(i)$  As $n\to\infty$, 
$$
\nu_{m,n}(u) \le \big( 1+o_n(1) \big) C_{m,\al} e^{\half (d-1)u}, 
$$
uniformly over $u\in (0,\gaR)$. \\
$(ii)$ There exists $C\in (0,\infty)$, independent of $n$ and $u$, such that 
$$
\nu_{m,n}(u)\ge \Big\{ \big( 1+o_n(1) \big) 2^{d-1}\xi_{d-1} \big( s_{d-1}^{-1}\al (d-1) \big)^{m-1} \ell(u) - C \ell'(u)  \Big\}_+e^{\half (d-1)u}, 
$$
uniformly over $u\in (0,\gaR)$, where $\ell'(u)$ is defined in \eqref{e:def.ell'}. \\
$(iii)$ There exist $C, c_0 \in (0,\infty)$, independent of  $n$ and $u$, such that for all $n\ge1$ and $u\in (0,R_n)$, 
$$
\nu_{m,n}(u) \le Ce^{c_0\omega_n} \cdot  \mainu. 
$$
\end{proposition}

\begin{remark}   \label{rem:open.q.clique}
Before proving Proposition \ref{p:exp.clique}, we briefly discuss the lower bound $\frac{2m-3}{2m-2}$ on the range of $\al$. This assumption appears to mark a phase transition rather than a mere technical constraint. To see this, consider how different configurations contribute to the clique counts. The proof machinery in Proposition \ref{p:exp.clique}  proceeds by first choosing a low-degree node connected to $p$, which lies farthest away from the origin of $\B_d$. This low-degree node contributes to the clique count in the order of $O(\mainu)$. We then select the remaining $m-2$ nodes in a small neighborhood of this low-degree node, which themselves contribute to the clique count with the order  $O(1)$. 

Next, let us consider an alternative configuration which places all the $m-1$ nodes (except for the $p$ itself) whose radial components, measured from the boundary of $B(o,R_n)$, are close to the level $u$.  Then, the density for radial components of these  $m-1$ nodes is given as $\big(n\radpdf (u)\big)^{m-1}$, which is approximately of order $O\big(n^{m-1} e^{-\al(d-1)(m-1)u}\big)$ by Lemma \ref{l:Lemma1.OY}. Furthermore, Lemma \ref{l:Lemma4.OY} shows that the probability of these $m-1$ nodes lying within distance $R_n$ from $p$ is of order
$O\big( (e^{-\half (d-1)(R_n-2u)})^{m-1} \big) = O\big( n^{-(m-1)} e^{(d-1)(m-1)u} \big)$. 
Thus, the resulting total contribution to the clique count is of order 
\begin{equation}  \label{e:order.alternative}
O\big(n^{m-1} e^{-\al(d-1)(m-1)u}\big) \times O\big( n^{-(m-1)} e^{(d-1)(m-1)u} \big) = O\big( e^{ (d-1) (m-1)(1-\al)u} \big). 
\end{equation}
Consequently, if   \eqref{e:order.alternative} dominates $O(\mainu)$, equivalently, $\half <\al < \frac{2m-3}{2m-2}$, 
the contribution from the latter alternative configuration dominates that in Proposition \ref{p:exp.clique}. This phenomenon identifies the threshold for the range of $\al$, and the asymptotics of clique counts when $\half <\al<\frac{2m-3}{2m-2}$ remains an open question. 
\end{remark}

\begin{proof}[Proof of Proposition \ref{p:exp.clique}]
By the Mecke formula for Poisson point processes,
\begin{align}
&\nu_{m,n}(u)= n^{m-1} \int_\wholet \P \big( \max_{1\le i, j \le m-1}d_H(X_i, X_j) \le R_n,  \,  \max_{1\le i \le m-1}d_H(X_i, p) \le R_n\, \big| \bt\big)\, \radpdf(\bt), \label{e:Mecke.exp.clique}
\end{align}
where $X_i=(t_i,\Ta_i)$ with $t_i=R_n-d_H(o,X_i)$ for $i=1,\dots, m-1$, and $\radpdf(\bt):= \prod_{i=1}^{m-1}\radpdf (t_i)$ is the product of densities of radial components. We will decompose $\nu_{m,n}(u)$ as follows. 
\begin{align*}
\nu_{m,n}(u) &= n^{m-1} \int_\wholet \P \big( \max_{1\le i, j \le m-1}d_H(X_i, X_j) \le R_n,  \, \max_{1\le i\le m-1}d_H(X_i, p) \le R_n \, \big| \bt\big)\, \radpdf(\bt) \\
&\qquad \quad \times \big( \one_{\mS_n(u)}(\bt) +  \one \{ t_1+t_2 >R_n-\omega_n \} \\
&\qquad \qquad + \sum_{\ell=2}^{m-2}\one \{ t_1 + t_\ell \le R_n-\omega_n, \, t_1 + t_{\ell+1} >R_n-\omega_n     \}  \\
&\qquad \qquad + \one \{ t_1 + t_{m-1} \le R_n-\omega_n, \, t_1 + u >R_n-\omega_n \} \\
&\qquad \qquad + \one \{  t_1+ u \le R_n-\omega_n \text{ and at least one of the conditions in \eqref{e:main.restrictions} fails}  \big\}\big) \\
&=: A_n(u) + B_n(u) + \sum_{\ell=2}^{m-2} C_{n,\ell}(u) + D_n(u) + E_n(u), 
\end{align*}
where 
\begin{equation}  \label{e:main.restrictions}
\mS_n(u):=\{ \bt: \max_{1\le i, j \le m-1} (t_i + t_j) \le R_n-\omega_n,  \ \max_{1\le i \le m-1}(t_i + u) \le R_n-\omega_n \}. 
\end{equation}

From this decomposition  we see that $A_n(u)$ is the leading term for $u\in (0,\gaR)$, whereas all the other terms are negligible as $n \to \infty$.
More precise asymptotics for the remaining terms are given in Lemma \ref{l:negligible.BCDE} in the Appendix. 
With this lemma  available, we turn our attention to $A_n(u)$, which itself can be decomposed as follows:
\begin{align*}
A_n(u) &= n^{m-1} \int_\wholet \P \big( \max_{1\le i, j \le m-1}d_H(X_i, X_j )\le R_n, \,   d_H(X_1, p)\le R_n \, \big| \bt \big) \, \one_{\mS_n(u)}(\bt) \radpdf (\bt) \\
&\quad - n^{m-1} \int_\wholet \P \big( \max_{1\le i, j \le m-1} d_H(X_i, X_j )\le R_n, \,   d_H(X_1, p)\le R_n,  \, \\
&\qquad \qquad\qquad\qquad\qquad\qquad\qquad \max_{2\le \ell \le m-1} d_H(X_\ell, p) >R_n\, \big| \bt \big) \, \one_{\mS_n(u)}(\bt) \radpdf (\bt) \\
&=: A_n^{(1)}(u) - A_n^{(2)}(u). 
\end{align*}
In what follows, we show that $A_n^{(1)}(u)$ is the dominant term, whereas $A_n^{(2)}(u)$ is asymptotically negligible. To prove the negligibility of $A_n^{(2)}(u)$, we need to estimate, for each $\ell = 2, \dots, m-1$, 
\begin{align*}
A_{n,\ell}^{(2)}(u) &:= n^{m-1} \int_\wholet \P \big(\max_{1\le i, j \le m-1} d_H(X_i, X_j )\le R_n,  \,  d_H(X_1, p)\le R_n, d_H(X_\ell, p)>R_n\,  \big| \bt \big)  \\
&\qquad  \qquad\qquad\qquad\qquad\qquad\qquad \qquad \qquad \times  \one_{\mS_n(u)}(\bt) \radpdf (\bt) \\
&\le n^{m-1} \int_{0\le t_1\le \cdots \le t_{m-1}\le (1-\delta)u} \P\big( \max_{2\le i \le m-1}d_H(X_1,X_i)\le R_n,  \, d_H(X_1,p)\le R_n, \, d_H(X_\ell, p)>R_n \, \big|\bt\big)\\
&\qquad  \qquad\qquad\qquad\qquad\qquad\qquad  \qquad \qquad \times  \one \{ t_{m-1}+ u\le R_n-\omega_n \} \radpdf (\bt)  \\
&\quad +n^{m-1} \int_{\wholet, \, t_{m-1}>(1-\delta)u} \hspace{-10pt}\P\big( \max_{2\le i \le m-1}d_H(X_1,X_i)\le R_n, \,  d_H(X_1,p)\le R_n\, \big|\bt\big) \\
&\qquad  \qquad\qquad\qquad\qquad\qquad\qquad\qquad\qquad\qquad \times  \one\{t_1+u\le R_n-\omega_n  \}\radpdf (\bt) \\
&=: A_{n,\ell}^{(2,1)}(u)+ A_{n}^{(2,2)}(u), 
\end{align*}
where $\delta\in (0,1)$ is a constant that will be specified  below; see Figure \ref{fig:main.negligible}. 

\begin{figure}
\centering
\includegraphics[scale=0.45]{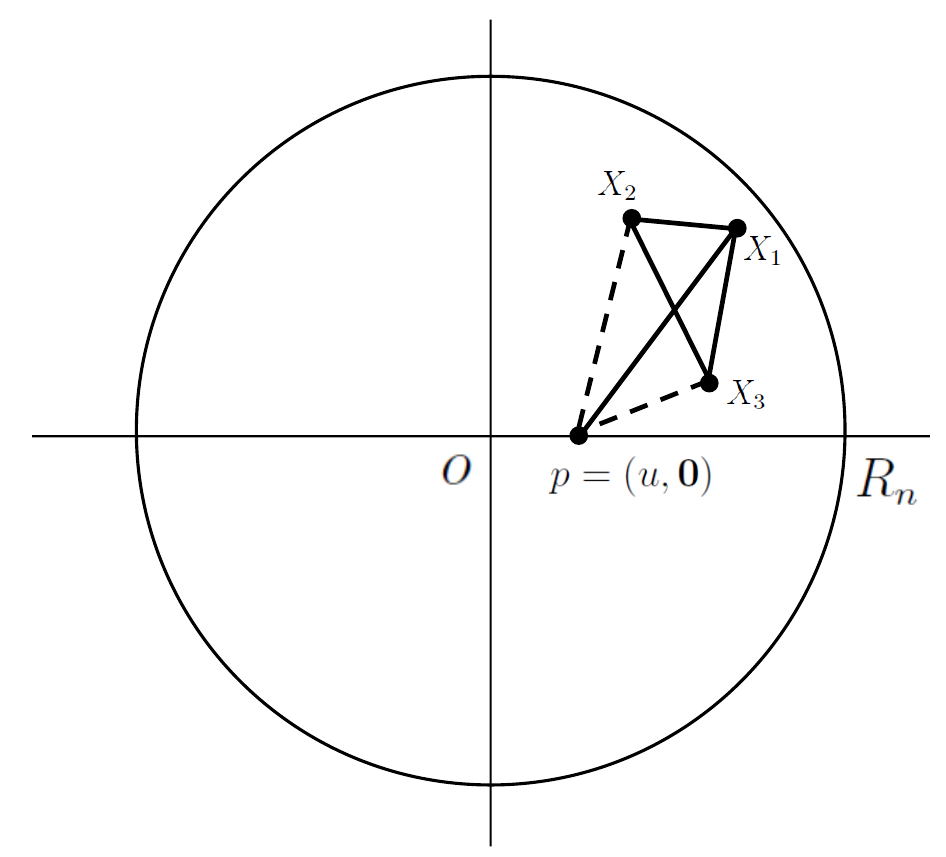}
\caption{\label{fig:main.negligible} \footnotesize{A $4$-clique on the vertices $\{X_1, X_2, X_3, p\}$ with $m=4$ and $d=2$. The asymptotic negligibility of $A_{n,2}^{(2)}(u)$ (resp.~$A_{n,3}^{(2)}(u)$) corresponds to the negligibility of the edge $X_2\to p$ (resp.~$X_3\to p$). Moreover, \eqref{e:rela.angle.main} gives the probability of the configuration with edges $X_1 \to p$, $X_1\to X_i$ for $i=2,3$, and $X_2 \to X_3$. } }
\end{figure}

First, by appealing to Lemmas \ref{l:Lemma1.OY} and \ref{l:Lemma4.OY}, 
\begin{align}
A_{n}^{(2,2)}(u) &\le C \mainu \int_{\wholet, \, t_{m-1}>(1-\delta)u} e^{\half (d-1)(m-1-2\al)t_1} \prod_{i=2}^{m-1} e^{\half (d-1)(1-2\al)t_i} \label{e:bound.An.22.u}  \\
&= C\mainu \int_{t_{m-1}=(1-\delta)u}^u  e^{\half (d-1)(1-2\al)t_{m-1}} \int_{t_1=0}^{t_{m-1}\wedge (R_n-\omega_n-u)}    e^{\half (d-1)(m-1-2\al)t_1} \notag  \\
&\qquad\qquad \qquad  \times \int_{(t_2,\dots,t_{m-2}):\, t_1 \le t_2 \le \cdots \le t_{m-2}\le t_{m-1}} \prod_{i=2}^{m-2} e^{\half (d-1)(1-2\al)t_i} \notag \\
&\le C\mainu \int_{t_{m-1}=(1-\delta)u}^u  e^{\half (d-1)(1-2\al)t_{m-1}} \int_{t_1=0}^{t_{m-1}}    e^{ (d-1)(m-2)(1-\al)t_1} \notag \\
&\le C\mainu \int_{t_{m-1}=(1-\delta)u}^u  e^{\half (d-1)(2m-3-(2m-2)\al)t_{m-1}} \notag \\
&\le C\mainu \cdot e^{\half (d-1)(2m-3-(2m-2)\al)(1-\delta)u}. \notag 
\end{align}

Turning to  the analysis of $A_{n,\ell}^{(2,1)}(u)$, we define $\Ta_{1i}^{(r)}\in [0,\pi]$ to be the relative angle between $\Ta_1$ and $\Ta_i$ for $i=2,\dots,m-1$. Similarly, $\Ta_i^{(r)} \in [0,\pi]$ denotes the relative angle between $\Ta_i$ and ${\bf 0}=(0,\dots, 0)\in \mA_d$. Then, by Lemma \ref{l:Lamma3.OY}, 
\begin{align}
&\P\big( \max_{2\le i \le m-1}d_H(X_1,X_i)\le R_n,  \, d_H(X_1,p) \le R_n,  \, d_H(X_\ell, p)> R_n  \, | \, \bt \big) \label{e:approx.by.Lemma3} \\
&=(1+o_n(1)) \P\big( \max_{2\le i \le m-1} (2R_n -t_1 -t_i + 2\log \sin (\Ta_{1i}^{(r)} / 2))\le R_n,  \notag \\
&\qquad \qquad \qquad \quad   2R_n -t_1 -u + 2\log \sin (\Ta_1^{(r)} / 2)\le  R_n, \,  2R_n -t_\ell -u + 2\log \sin (\Ta_{\ell}^{(r)} / 2)> R_n\big),\notag 
\end{align}
uniformly on the set 
$$
\mS_n'(u):= \big\{  \bt: 0\le t_1 \le \cdots \le t_{m-1}\le (1-\delta)u, \, t_{m-1}+ u \le R_n-\omega_n\big\}. 
$$
As a result of additional approximations 
$$
\sin \Big( \frac{\Theta_{1i}^{(r)}}{2} \Big) \sim  \frac{\Theta_{1i}^{(r)}}{2}, \, i=2,\dots,m-1,  \ \sin \Big( \frac{\Theta_{1}^{(r)}}{2} \Big) \sim  \frac{\Theta_{1}^{(r)}}{2},   \ \ \text{and } \ \ \sin \Big( \frac{\Theta_{\ell}^{(r)}}{2} \Big) \sim  \frac{\Theta_{\ell}^{(r)}}{2}, 
$$ 
uniformly for all $\bt \in \mS_n'(u)$, the expression \eqref{e:approx.by.Lemma3} is approximated uniformly by 
\begin{equation}  \label{e:after.Lemma3}
\P \big( \Ta_{1i}^{(r)} \le 2e^{-\half (R_n-t_1-t_i)}, \, i=2,\dots, m-1, \ \Ta_1^{(r)} \le  2e^{-\half (R_n-t_1-u)}, \, \Ta_\ell^{(r)} > 2e^{-\half (R_n-t_\ell-u)} \big). 
\end{equation}
By the triangle inequality $\Ta_\ell^{(r)} \le \Ta_1^{(r)} + \Ta_{1\ell}^{(r)}$, it holds that 
$$
2e^{-\half (R_n-t_\ell-u)} \le 2e^{-\half (R_n-t_1-u)} +  2e^{-\half (R_n-t_1-t_\ell)}; 
$$
equivalently, $t_1 \ge t_\ell-2\log \big(  1+e^{-\half (u-t_\ell)}\big)$. Since $t_\ell\le t_{m-1}\le (1-\delta)u$, it follows that $t_1 \ge t_\ell -2\log \big(  1+e^{-\frac{\delta u}{2}}\big)$. Now, one can bound \eqref{e:after.Lemma3}   by 
\begin{align*}
&\P \big( \Ta_{1i}^{(r)} \le 2 e^{-\half (R_n-t_1-t_i)}, \, i=2,\dots, m-1, \, \Ta_1^{(r)}\le 2 e^{-\half (R_n-t_1-u)} \big)\one \{ t_1 \ge t_\ell -2\log \big(  1+e^{-\frac{\delta u}{2}}\big) \} \\
&\le C \prod_{i=2}^{m-1} e^{-\half (d-1)(R_n-t_1-t_i)}  \cdot  e^{-\half (d-1)(R_n-t_1-u)}  \one \{ t_1 \ge t_\ell -2\log \big(  1+e^{-\frac{\delta u}{2}}\big) \}\\
&= C n^{-(m-1)} \mainu \cdot e^{\half (d-1) (m-1)t_1} \prod_{i=2}^{m-1} e^{\half (d-1)t_i} \one \{ t_1 \ge t_\ell -2\log \big(  1+e^{-\frac{\delta u}{2}}\big) \}, 
\end{align*}
uniformly for $\bt\in \mS_n'(u)$. Substituting this bound into $A_{n,\ell}^{(2,1)}(u)$, we have, for each $\ell=2,\dots, m-1$, 
\begin{align*}
A_{n,\ell}^{(2,1)}(u) &\le C \mainu \int_{0\le t_1\le \cdots \le t_{m-1}\le (1-\delta)u}  \one \{ t_1 \ge t_\ell -2\log \big(  1+e^{-\frac{\delta u}{2}}\big) \} \\
&\qquad \qquad \qquad \qquad\qquad \qquad\qquad \qquad\times  e^{\half (d-1)(m-1-2\al)t_1} \prod_{i=2}^{m-1} e^{\half (d-1)(1-2\al)t_i} \\
&= C\mainu \int_{0\le t_2\le \cdots \le t_{m-1}\le (1-\delta)u} \prod_{i=2}^{m-1} e^{\half (d-1)(1-2\al)t_i} \int_{t_1=t_\ell -2\log (  1+e^{-\frac{\delta u}{2}})}^{t_\ell} e^{\half (d-1)(m-1-2\al)t_1}  \\
&=C \mainu \big\{ 1-\big( 1+e^{-\frac{\delta u}{2}} \big)^{-(d-1)(m-1-2\al)} \big\} \\
&\qquad \qquad \times \int_{0\le t_2\le \cdots \le t_{m-1}\le (1-\delta)u} \prod_{i=2}^{m-1} e^{\half (d-1)(1-2\al)t_i}  \cdot e^{\half (d-1)(m-1-2\al)t_\ell}. 
\end{align*}
Observe that 
$$
 1-\big( 1+e^{-\frac{\delta u}{2}} \big)^{-(d-1)(m-1-2\al)}   \le (d-1)(m-1-2\al) e^{-\frac{\delta u }{2}}, 
$$
and 
\begin{align*}
& \int_{0\le t_2\le \cdots \le t_{m-1}\le (1-\delta)u} \prod_{i=2}^{m-1} e^{\half (d-1)(1-2\al)t_i}  \cdot e^{\half (d-1)(m-1-2\al)t_\ell} \\
&=\int_{(t_\ell, \dots, t_{m-1}):\, 0\le t_\ell \le \cdots \le t_{m-1}\le (1-\delta)u} \prod_{i=\ell+1}^{m-1} e^{\half (d-1)(1-2\al)t_i}  \cdot e^{\half (d-1)(m-4\al)t_\ell}\\
&\qquad \qquad \times  \int_{(t_2,\dots, t_{\ell-1}):\, 0\le t_2 \le \cdots \le t_{\ell-1}\le t_\ell} \prod_{i=2}^{\ell-1} e^{\half (d-1)(1-2\al)t_i} \\
&\le C \int_{t_\ell=0}^{(1-\delta)u} e^{\half (d-1)(m-4\al)t_\ell} \int_{(t_{\ell+1}, \dots, t_{m-1}):\,  t_\ell\le t_{\ell+1} \le \cdots \le t_{m-1}\le (1-\delta)u} \prod_{i=\ell+1}^{m-1} e^{\half (d-1)(1-2\al)t_i} \\
&\le C\int_{t_\ell=0}^{(1-\delta)u} e^{\half (d-1)(m-4\al)t_\ell}. 
\end{align*}
Hence, 
$$
A_{n,\ell}^{(2,1)}(u)\le C\mainu \cdot e^{-\frac{\delta u}{2}}\int_{t_\ell=0}^{(1-\delta)u} e^{\half (d-1)(m-4\al)t_\ell}. 
$$
If $m-4\al\le 0$, then $A_{n,\ell}^{(2,1)}(u)\le C \mainu \cdot e^{-\frac{\delta u}{4}}$. 
If $m-4\al>0$, we have 
$$
A_{n,\ell}^{(2,1)}(u)\le C\mainu \cdot e^{\half ((d-1)(m-4\al)(1-\delta)-\delta)u}.
$$ 
Now, choose $\delta$ to satisfy 
\begin{equation}  \label{e:constraint.delta}
\frac{(d-1)(m-4\al)_+}{(d-1)(m-4\al)_++1} <\delta <1. 
\end{equation}
Then it follows that 
\begin{equation}  \label{e:bound.Anl.21(u)}
A_{n,\ell}^{(2,1)}(u)\le C \mainu \big( e^{-\frac{\delta u}{4}} + e^{\half ((d-1)(m-4\al)_+(1-\delta)-\delta)u} \big). 
\end{equation}

We conclude from \eqref{e:bound.Anl.21(u)} and \eqref{e:bound.An.22.u} that for all $n\ge1$ and $u\in (0,R_n)$, 
\begin{equation}  \label{e:bound.An2(u)}
A_n^{(2)}(u) \le \sum_{\ell=2}^{m-1} A_{n,\ell}^{(2,1)}(u) + A_{n}^{(2,2)}(u)\le C\mainu \ell'(u), 
\end{equation}
where 
\begin{equation}  \label{e:def.ell'}
\ell'(u):= e^{-\frac{\delta u}{4}} + e^{\half ((d-1)(m-4\al)_+(1-\delta)-\delta)u} + e^{\half (d-1)(2m-3-(2m-2)\al)(1-\delta)u}. 
\end{equation}
Note that $\ell'(u)$ is decreasing in $u$ due to the constraint in  \eqref{e:constraint.delta}, and also, $\ell'(u)\to 0$ as $u\to\infty$. 

It now remains to derive the specific bounds as stated in Proposition \ref{p:exp.clique}, which can be done by simplifying $A_n^{(1)}(u)$.  First, by repeating the argument  as in \eqref{e:approx.by.Lemma3}, 
\begin{align}
&\P \big( \max_{1\le i, j \le m-1}d_H(X_i, X_j) \le R_n,\, d_H(X_1,p)\le R_n \, \big|\bt \big) \label{e:rela.angle.main}\\
&=(1+o_n(1)) \P\big( \Ta_{ij}^{(r)}\le 2 e^{-\half (R_n-t_i-t_j)}, \, i, j=1,\dots, m-1, \ \Ta_{1}^{(r)}\le 2 e^{-\half (R_n-t_1-u)}\big), \notag
\end{align}
uniformly for $\bt \in \mS_n(u)$. Since  $2 e^{-\half (R_n-t_i-t_j)} \le 2e^{-\half \omega_n}\to 0$ and $2 e^{-\half (R_n-t_1-u)} \le 2e^{-\half \omega_n}\to 0$ as $n\to\infty$, all the relative angles  in \eqref{e:rela.angle.main} decay uniformly to zero as $n \to \infty$.

Next, we redefine $\mA_d = [0, \pi]^{d-2} \times [0, 2\pi)$ as the $(d-1)$-dimensional sphere centered at the origin of $\B_d$ in the ambient space $\R^d$. In this ambient space, we then consider a hyperplane $H^{d-1}$ tangent to ${\bf 0} \in \mA_d$. Here, we also regard ${\bf 0} \in \mA_d$ as the origin of $H^{d-1}$. Note that $H^{d-1}$ is homeomorphic to $\R^{d-1}$. For $i=1,\dots,m-1$, let $L_i:= \{ t\Ta_i: t>0 \}$ be the ray in the direction of  $\Ta_i$ and we write \eqref{e:rela.angle.main} as 
\begin{align*}
&\big( 1+o_n(1) \big) \P\big( \Ta_{ij}^{(r)}\le 2 e^{-\half (R_n-t_i-t_j)}, \, i, j=1,\dots, m-1, \\
&\qquad \qquad \qquad \Ta_{1}^{(r)}\le 2 e^{-\half (R_n-t_1-u)} \big| \, L_i \cap H^{d-1}\neq \emptyset, \, i=1,\dots, m-1  \big) \prod_{i=1}^{m-1}\P(L_i \cap H^{d-1}\neq \emptyset) \\
&= \big( 1+o_n(1) \big) 2^{-(m-1)} \P\big( \Ta_{ij}^{(r)}\le 2 e^{-\half (R_n-t_i-t_j)}, \, i, j=1,\dots, m-1, \\
&\qquad \qquad \qquad\qquad \qquad \Ta_{1}^{(r)}\le 2 e^{-\half (R_n-t_1-u)} \big| \, L_i \cap H^{d-1}\neq \emptyset, \, i=1,\dots, m-1  \big). 
\end{align*}

On the event $\bigcap_{i=1}^{m-1}\big\{ L_i \cap H^{d-1}\neq \emptyset \big\}$, let  $S_i$ denote  the vector emanating from the origin of $\B_d$ and intersecting the hyperplane $H^{d-1}$ in the same direction as $\Theta_i$. Let also $T_i$ be the vector in $H^{d-1}$ corresponding to  the intersection point between $S_i$ and $H^{d-1}$; see Figure \ref{fig:hyperplane}. 

\begin{figure}
\centering
\includegraphics[scale=0.4]{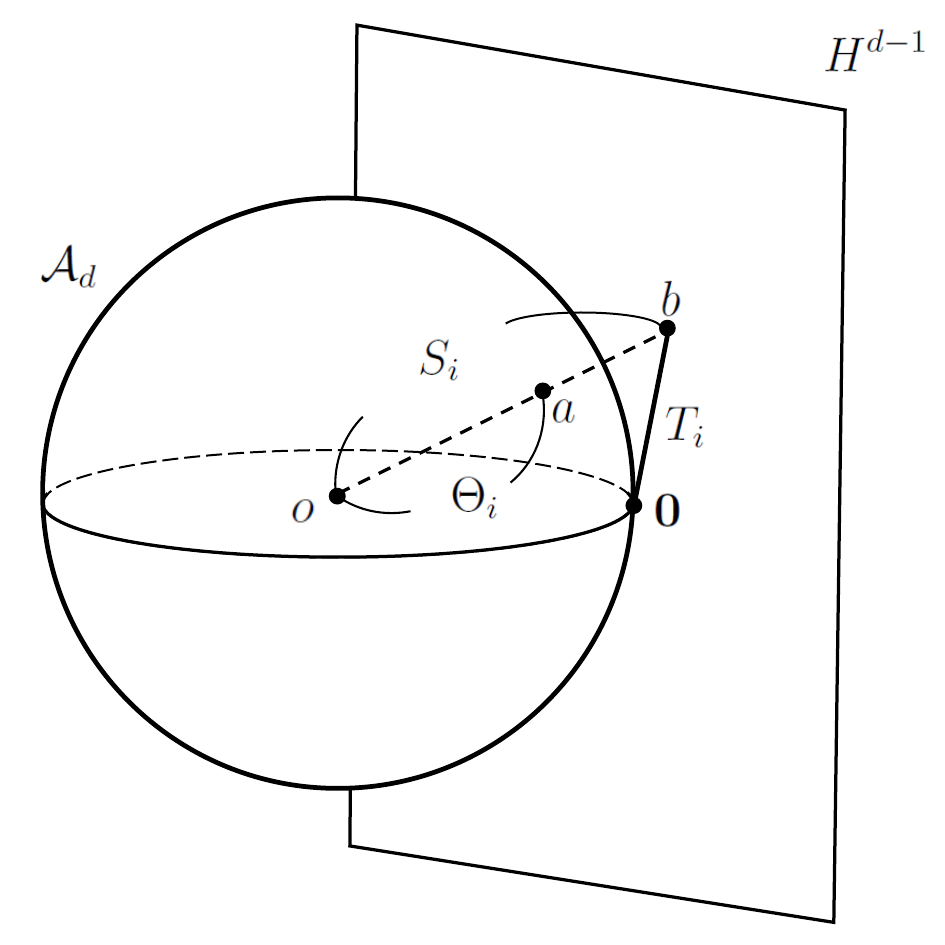}
\caption{\label{fig:hyperplane} \footnotesize{The point $a$ lies in $\A_d$ and $b$ lies in $H^{d-1}$. We set $\Ta_i = \overrightarrow{oa}$, $S_i = \overrightarrow{ob}$, and $T_i = \overrightarrow{{\bf 0}b}$. } }
\end{figure}

 Then the probability density of $T_i$ is $(d-1)$-dimensional multivariate Cauchy, given by 
$$
g(x) :=\frac{2}{s_{d-1}}\, \big( 1+\|x\|^2 \big)^{-d/2}, \  \ \ x\in H^{d-1}. 
$$
By the law of cosines, for $i,j = 1, \dots, m-1$, we have 
$$
\|T_i-T_j\|^2 =\|S_i\|^2 + \|S_j\|^2 - 2 \|S_i\| \|S_j\| \cos (\Ta_{ij}^{(r)}). 
$$
Next, consider the great circle spanned by $S_i$ and ${\bf 0} \in \mA_d$.  
With this circle, we have the following expansions:
\begin{align*}
\|S_i\|^2 &= 1 + \tan^2 (\Ta_{i}^{(r)}) = 1 +(\Ta_{i}^{(r)})^2 + O\big( (\Ta_{i}^{(r)})^4 \big), \\
\cos (\Ta_{ij}^{(r)}) &= 1-\half (\Ta_{ij}^{(r)})^2 +  O\big( (\Ta_{ij}^{(r)})^4 \big). 
\end{align*}
Using these expansions, we derive  that  
$\|T_i-T_j\|= \Ta_{ij}^{(r)}(1+o_n(1))$ 
uniformly for $\bt\in \mS_n(u)$. Similarly,  uniformly for $\bt\in \mS_n(u)$, we have 
$\|T_1\| = \Ta_1^{(r)}(1+o_n(1))$. Hence, one can uniformly approximate \eqref{e:rela.angle.main} by 
\begin{align*}
&2^{-(m-1)}\P\big( \| T_i-T_j \|\le 2e^{-\half (R_n-t_i-t_j)}, \, i,j=1,\dots, m-1, \ \|T_1\|\le  2e^{-\half (R_n-t_1-u)} \big)\\
&=2^{-(m-1)}\int_{(u_1, \dots, u_{m-1})\in (H^{d-1})^{m-1}}\one \big\{ \| u_i-u_j \|\le 2e^{-\half (R_n-t_i-t_j)}, \, i,j=1,\dots, m-1, \\
&\qquad \qquad \qquad \qquad \qquad \qquad \qquad \qquad \qquad \qquad  \|u_1\|\le  2e^{-\half (R_n-t_1-u)}  \big\} \prod_{i=1}^{m-1}g(u_i)\\
&= (2n)^{-(m-1)} \hspace{-3pt}\int_{(v_1,\dots,v_{m-1})\in (H^{d-1})^{m-1}} \one \big\{ \| v_i-v_j \|\le 2e^{\half (t_i+t_j)}, \, i,j=1,\dots, m-1, \\
&\qquad \qquad \qquad \qquad \qquad \qquad \qquad \qquad \qquad \qquad  \|v_1\|\le  2e^{\half  (t_1+u)}  \big\} \prod_{i=1}^{m-1}g(e^{-\half R_n}v_i), 
\end{align*}
where the last line follows from the change of variables $u_i=e^{-\half R_n}v_i$ for $i=1,\dots,m-1$. Now,  $A_n^{(1)}(u)$ can  be expressed   as 
\begin{align}
A_n^{(1)}(u) &= (1+o_n(1))\, 2^{-(m-1)} \label{e:before.radpdf}\\
&\quad \times \int_\wholet \int_{(v_1,\dots,v_{m-1})\in (H^{d-1})^{m-1}}\one \big\{ \| v_i-v_j \|\le 2e^{\half (t_i+t_j)}, \, i,j=1,\dots, m-1, \notag \\
&\qquad \qquad \qquad\qquad \qquad\qquad \qquad \|v_1\|\le  2e^{\half  (t_1+u)}  \big\} \prod_{i=1}^{m-1}g(e^{-\half R_n}v_i)\, \one_{\mS_n(u)}(\bt) \radpdf (\bt). \notag
\end{align}
Since  $e^{-\half R_n}\|v_i\| \le 4e^{-\half \omega_n}\to0$ as $n\to\infty$,  we have, uniformly, 
$$
\prod_{i=1}^{m-1}g(e^{-\half R_n}v_i) \to \Big( \frac{2}{s_{d-1}} \Big)^{m-1} \ \ \text{as } n\to\infty. 
$$
Thus, together with Lemma \ref{l:Lemma1.OY} $(ii)$, $A_n^{(1)}(u)$ asymptotically equals 
\begin{align*}
&(s_{d-1}^{-1}\al(d-1))^{m-1}\int_\wholet \int_{(v_1,\dots,v_{m-1})\in (H^{d-1})^{m-1}} \hspace{-10pt}\one \big\{ \| v_i-v_j \|\le 2e^{\half (t_i+t_j)}, \, i,j=1,\dots, m-1, \,  \\
&\qquad \qquad \qquad \qquad\qquad \qquad\qquad \qquad\qquad \qquad\|v_1\|\le  2e^{\half  (t_1+u)}  \big\}  \prod_{i=1}^{m-1}e^{-\al(d-1)t_i}\,  \one_{\mS_n(u)}(\bt). 
\end{align*}
By performing the change of variables 
$w_1 =v_1$ and $w_i= v_i-v_1$ for $i =2, \dots, m-1$, 
and integrating with respect to $w_1$ over $H^{d-1}$, 
the preceding expression becomes   
\begin{align}
&2^{d-1} \xi_{d-1} (s_{d-1}^{-1}\al(d-1))^{m-1}\mainu \int_\wholet e^{\half (d-1)(1-2\al)t_1} \prod_{i=2}^{m-1}e^{-\al(d-1)t_i}\one_{\mS_n(u)}(\bt) \label{e:An1.final.except.radial.densities}\\
&\quad \times \int_{(w_2,\dots, w_{m-1})\in (H^{d-1})^{m-2}} \one \big\{ \|w_i\| \le 2e^{\half (t_1+t_i)}, \, i=2,\dots,m-1, \notag \\
&\qquad \qquad \qquad \qquad \qquad \qquad  \qquad \qquad \|w_i-w_j\|\le 2e^{\half (t_i + t_j)}, \, i, j=2,\dots,m-1\big\}. \notag
\end{align}

We now claim that,  uniformly for $u \in (0,\gaR)$, 
\begin{equation}  \label{e:exact.An1(u)}
A_n^{(1)}(u) =(1+o_n(1)) \, 2^{d-1} \xi_{d-1} (s_{d-1}^{-1}\al (d-1))^{m-1}\mainu \ell(u), 
\end{equation}
where 
\begin{align} 
\ell(u)&:= \int_\wholet e^{\half (d-1)(1-2\al)t_1} \prod_{i=2}^{m-1} e^{-\al(d-1)t_i}  \label{e:def.ell} \\
&\quad \times \int_{(w_2,\dots, w_{m-1})\in (H^{d-1})^{m-2}} \one \big\{ \|w_i\| \le 2e^{\half (t_1+t_i)}, \, i=2,\dots,m-1, \notag \\
&\qquad \qquad\qquad\qquad\qquad\qquad\qquad \|w_i-w_j\|\le 2e^{\half (t_i + t_j)}, \, i, j=2,\dots,m-1\big\}.   \notag 
\end{align}
Before turning to the proof of \eqref{e:exact.An1(u)}, we record a few properties of the function $\ell$. First, $\ell(u)$ is an increasing function of $u$.  
By removing  the constraint $\|w_i - w_j\|\le 2 e^{\frac{1}{2}(t_i+ t_j)}$ from  \eqref{e:def.ell}, one can see that 
\begin{align*}
\ell(\infty):= \lim_{u\to\infty} \ell(u) \le C \int_0^\infty e^{\half (d-1)(2m-3-(2m-2)\al)t}\dif t<\infty, 
\end{align*}
which, in turn, implies by the dominated convergence theorem that $\ell(u)\to0$ as $u \to 0$. 
For the proof of \eqref{e:exact.An1(u)}, in view of the expression in \eqref{e:An1.final.except.radial.densities}, it suffices to show that, as $n\to\infty$,
\begin{align}
&\sup_{u \in (0,\gaR)}\ell(u)^{-1}\int_\wholet e^{\half (d-1)(1-2\al)t_1} \prod_{i=2}^{m-1}e^{-\al(d-1)t_i}\, \one_{\mS_n(u)^c}(\bt) \label{e:unif.conv.ratio2}\\
&\quad\qquad\qquad \times \int_{(w_2,\dots, w_{m-1})\in (H^{d-1})^{m-2}} \one \big\{ \|w_i\| \le 2e^{\half (t_1+t_i)}, \, i=2,\dots,m-1,  \notag \\
&\qquad \qquad \qquad \qquad \qquad \qquad  \qquad \qquad \|w_i-w_j\|\le 2e^{\half (t_i + t_j)}, \, i, j=2,\dots,m-1\big\}\to 0. \notag 
\end{align}
To prove this, note first that if $u\in (0,\gaR)$ and $\bt \in \mS_n(u)^c$, then $t_{m-1}>(1-\ga)R_n-\omega_n$ and $u>\half (R_n-\omega_n)$. In particular, we may restrict to   $u\in \big( \half (R_n-\omega_n), \gaR \big)$, and \eqref{e:unif.conv.ratio2} can be  bounded from above by 
\begin{align*}
&\ell\Big( \half (R_n-\omega_n) \Big)^{-1} \int_{0\le t_1 \le \cdots \le t_{m-1}<\infty} e^{\half (d-1)(1-2\al)t_1} \prod_{i=2}^{m-1}e^{-\al(d-1)t_i} \, \one\{t_{m-1}>(1-\ga)R_n-\omega_n  \}\\
&\quad\qquad\qquad \times \int_{(w_2,\dots, w_{m-1})\in (H^{d-1})^{m-2}} \one \big\{ \|w_i\| \le 2e^{\half (t_1+t_i)}, \, i=2,\dots,m-1 \big\}\\
&\le C  \int_{0\le t_1 \le \cdots \le t_{m-1}<\infty} e^{\half (d-1)(m-1-2\al)t_1} \prod_{i=2}^{m-1}e^{\half (d-1)(1-2\al)t_i} \, \one\{t_{m-1}>(1-\ga)R_n-\omega_n  \}\\
&\to 0, \ \ \text{as } n\to \infty; 
\end{align*}
hence, \eqref{e:unif.conv.ratio2} has been established. 

Now, combining the first bound in Lemma \ref{l:negligible.BCDE} 
 and \eqref{e:exact.An1(u)}, we have, uniformly over $u\in (0,\gaR)$, 
$$
\nu_{m,n}(u) = (1+o_n(1)) A_n(u) \le (1+o_n(1))  A_n^{(1)}(u) \le (1+o_n(1)) C_{m,\al} \mainu, 
$$
which completes the proof of $(i)$. 
Next, it follows from  \eqref{e:exact.An1(u)} and \eqref{e:bound.An2(u)} that,  uniformly over $u\in (0,\gaR)$, 
\begin{align*}
\nu_{m,n}(u) &\ge A_n(u)\ge  \big( (1+o_n(1))2^{d-1} \xi_{d-1} (s_{d-1}^{-1}\al (d-1))^{m-1}\ell(u) - C\ell'(u) \big)\mainu, 
\end{align*}
which completes the proof of $(ii)$. 

Finally, for $u\in (0, R_n)$, we apply Lemma \ref{l:Lemma1.OY} $(i)$ to 
\eqref{e:before.radpdf} (in place of part $(ii)$), along with an obvious bound $\one_{\mS_n(u)}(\bt)\le 1$; then, 
 uniformly over $u \in (0, R_n)$, we have 
\begin{equation}  \label{e:upper.bound.An1(u)}
A_n^{(1)}(u) \le (1+o_n(1))\, 2^{d-1} \xi_{d-1} (s_{d-1}^{-1}\al (d-1))^{m-1}\mainu \ell(u) \le (1+o_n(1)) C_{m,\al} \mainu. 
\end{equation}
By the second bound of Lemma \ref{l:negligible.BCDE} and \eqref{e:upper.bound.An1(u)}, we have, uniformly over $u\in (0,R_n)$, 
$$
\nu_{m,n}(u)\le Ce^{c_0\omega_n}\cdot \mainu + (1+o_n(1))C_{m,\al} \mainu \le Ce^{c_0\omega_n}\cdot \mainu, 
$$
completing the proof of $(iii)$. 
\end{proof}

\begin{proposition}  \label{p:var.clique}
Let $\frac{2m-3}{2m-2} < \alpha <1$. There exists a constant $C, c_1 \in (0,\infty)$ not depending on $n$ and $u$, such that 
$$
\text{Var}\big(\mC_{m,n}(u)\big) \le Ce^{ c_1\omega_n} \cdot e^{\half (d-1) (2m-3-(2m-4)\alpha)u}, 
$$
for all $n\ge1$ and $u \in (0,R_n)$.
\end{proposition}

\begin{proof}
By the Mecke formula, the variance  can be represented as 
\begin{align}
\begin{split}  \label{e:decomp.var.clique}
\text{Var}\big( \mC_{m,n}(u) \big) &= \nu_{m,n}(u) \\
&\quad + \sum_{q=1}^{m-2}\E \bigg[ \frac{1}{\big( (m-1)! \big)^2}\sum_{(P_1,\dots,P_{m-1})\in (\Pn)_{\neq}^{m-1}} \hspace{-10pt}\sum_{\substack{(Q_1,\dots,Q_{m-1})\in (\Pn)_{\neq}^{m-1}, \\ |(P_1,\dots,P_{m-1})\cap (Q_1,\dots,Q_{m-1})|=q}}\hspace{-10pt} h_n (p, P_1,\dots, P_{m-1}) \\
&\qquad \qquad \qquad \qquad\qquad \qquad\qquad \qquad\qquad \qquad\qquad \qquad \qquad \times h_n (p, Q_1,\dots, Q_{m-1})  \bigg] \\
&=: \nu_{m,n}(u) + \sum_{q=1}^{m-2} I_{n,q}(u), 
\end{split}
\end{align}
where $p=(u,{\bf 0})$ with $u=R_n-d_H(o,p)$ and 
\begin{equation}  \label{e:def.hn}
h_n(p, x_1,\dots,x_{m-1}) := \one \big\{ (x_1,\dots,x_{m-1} ) \text{ forms an } (m-1)\text{-clique}, \, x_i \rightarrow p, \, i=1,\dots,m-1 \big\}. 
\end{equation}
For further analysis, we consider two $m$-cliques, denoted $\mC_1$ and $\mC_2$ and  both constructed from $\Pn$, which share $q+1$ common vertices for some $q \in \{1,\dots, m-2\}$.  
Among these $q+1$ shared vertices, $p$ is the point lying closest to the origin of $\B_d$. 
Let $Z=(W,\Psi)$ with $W=R_n-d_H(o,Z)$, be  the shared vertex that lies closest to the boundary of $B(0,R_n)$ among all shared vertices.
Let $Y_1=(V_1,\Ta_1)$, $V_1=R_n-d_H(o,Y_1)$, be the vertex in $\mC_1$ that is not shared by $\mC_2$ and lies closest to the boundary of $B(0,R_n)$ among all non-shared vertices of $\mC_1$. Similarly, let $Z_1=(W_1,\Phi_1)$, $W_1=R_n-d_H(o,Z_1)$ be the vertex in $\mC_2$ that is not shared by $\mC_1$ and lies closest to the boundary of $B(0,R_n)$ among all non-shared vertices of $\mC_2$.

We consider three possible cases according to the size of radial components of $Z$, $Y_1$, and $Z_1$: $(i)$ $W\ge V_1 \vee W_1$, $(ii)$ $V_1 \wedge W_1 \le W < V_1 \vee W_1$, and $(iii)$ $W< V_1 \wedge W_1$. By Lemma \ref{l:Lemma4.OY}, the probability of two edges $\{ Y_1,Z \}$ and $\{Z_1, Z\}$ is asymptotically given as a constant multiple of $n^{-3}e^{\half (d-1) (V_1 + 2W+ W_1)}$. Note that $W$ appears with coefficient $2$, giving it twice the influence of $V_1$ and $W_1$. Hence, Case $(i)$ leads to  the largest edge probability and is therefore the most demanding configuration. Accordingly, we restrict our attention to Case $(i)$ below.

More precisely,   we consider the following configuration built over $\mC_1$ and $\mC_2$. First, define the point sets $\mQ_i := V(\mC_i)\setminus V(\mC_1 \cap\mC_2)$, $i=1,2$, where $V(\mC_i)$ denotes the vertex set of $\mC_i$.  Next, select a node $S$ from $\mQ_1\cup \mQ_2$ that lies closest to the origin of  $\B_d$ among all nodes in $\mQ_1 \cup \mQ_2$. 
For ease of description, suppose without loss of generality that $Y_1 \in \mC_1$, $Z_1\in \mC_2$, and $S\in \mC_2$. 
We   then take  the path formed by the nodes in $\mC_1$:
\begin{equation} \label{e:path.no.removal}
Y_1 \to Z, \ \ Y_1\to Y_2 \to \cdots \to Y_{m-2} \to p=(u,{\bf 0}),
\end{equation}
where $Y_i = (V_i,\Ta_i)$, $V_i=R_n-d_H(o,Y_i)$, $i=1,\dots,m-2$,  satisfy $V_1 \le \cdots \le V_{m-2} \le u$. Note that all $m$ nodes in \eqref{e:path.no.removal} constitute $\mC_1$. 

We next turn to $\mC_2$ and consider the path in $\mC_2$: 
\begin{equation} \label{e:path.removal}
Z_1 \to Z, \ \ Z_1\to Z_2 \to\cdots \to Z_{m-q-1},
\end{equation}
where $Z_i = (W_i, \Phi_i)$, $W_i=R_n-d_H(o,Z_i)$, $i=1,\dots,m-q-1$ such that $W_1\le \cdots \le W_{m-q-1} \le u$.  By construction, $Z_{m-q-1} = S$, and the nodes in \eqref{e:path.removal}, together with the remaining $q$ shared nodes (including the point $p$), constitute $\mC_2$.     Since $Z_{m-q-1}(=S)$ is, by choice, the closest  point to the origin of $\B_d$ among all non-shared vertices of the two cliques, we must have $V_{m-2} \le W_{m-q-1}$.

Based on the configuration in \eqref{e:path.no.removal} and \eqref{e:path.removal}, one can bound  $I_{n,q}(u)$ in \eqref{e:decomp.var.clique} by 
\begin{align}
A_{n,q}(u) &:= n^{2m-q-2}\int_{\mT (\bt, \bs, w)} \P \big( \max_{1\le i \le m-3}d_H(Y_i, Y_{i+1})\le R_n,  \, d_H(Y_{m-2}, p)\le R_n,\label{e:main.variance.clique}\\
&\qquad \qquad \qquad\qquad\qquad \max_{1 \le i \le m-q-2} d_H(Z_i, Z_{i+1})\le R_n, \,  d_H(Y_1, Z)\le R_n,  \notag \\
&\qquad \qquad \qquad\qquad \qquad \qquad d_H(Z_1, Z)\le R_n \, \big|\, \bt, \bs, w \big) \, \radpdf (\bt)\radpdf (\bs)\radpdf (w), \notag 
\end{align}
where we have set, by abuse of notation,  $Y_i=(t_i, \Ta_i)$, $i=1,\dots,m-2$, $Z_i=(s_i, \Phi_i)$, $i=1,\dots,m-q-1$, and $Z=(w,\Psi)$, and 
$$
\mT (\bt, \bs, w) := \big\{ (\bt,\bs,w): 0\le t_1 \le \cdots \le t_{m-2}\le s_{m-q-1},\,  0\le s_1 \le \cdots \le s_{m-q-1} \le u, \, t_1 \vee s_1 \le w \big\}. 
$$
Moreover, $\radpdf (\bt)$, $\radpdf (\bs)$ are the product of  densities of radial components  as   in \eqref{e:Mecke.exp.clique}. 

Since the configurations in \eqref{e:path.no.removal} and \eqref{e:path.removal} do not involve any cycles, the independence of the angular components   allows us to express the conditional  probability in  \eqref{e:main.variance.clique} in a product form. Specifically, it is equal to
\begin{align} 
&\prod_{i=1}^{m-3}\P \big( d_H(Y_i, Y_{i+1})\le R_n \big|\, t_i, t_{i+1} \big) \P\big( d_H(Y_{m-2}, p)\le R_n \big| \, t_{m-2} \big)\label{e:prod.form.var.clique} \\
&\qquad  \times  \prod_{i=1}^{m-q-2} \P\big( d_H(Z_i, Z_{i+1})\le R_n \big|\, s_i, s_{i+1} \big)  \P\big( d_H(Y_1, Z)\le R_n \big|\, t_1, w \big) \P\big( d_H(Z_1, Z)\le R_n \big|\, s_1, w \big). \notag 
\end{align}
According to Lemma \ref{l:Lemma4.OY}, 
\begin{align}
&\P \big( d_H(Y_i, Y_{i+1})\le R_n \big|\, t_i, t_{i+1} \big) \label{e:ignore.restriction.omega}\\
&\le \P \big( d_H(Y_i, Y_{i+1})\le R_n \big|\, t_i, t_{i+1} \big) \one \{ t_i+t_{i+1}\le R_n-\omega_n\} + \one \{ t_i+t_{i+1}> R_n-\omega_n\}  \notag \\
&\le Ce^{\half (d-1)\omega_n} \cdot e^{-\half (d-1)(R_n-t_i-t_{i+1})}. \notag 
\end{align}
Appealing to bounds of the same type as above,  \eqref{e:prod.form.var.clique} can be bounded  by
\begin{align*}
&C e^{\half (d-1) (2m-q-2)\omega_n} \prod_{i=1}^{m-3} e^{-\half (d-1) (R_n-t_i-t_{i+1})} \cdot e^{-\half (d-1)(R_n-t_{m+2}-u)} \\
&\quad \times \prod_{i=1}^{m-q-2} e^{-\half (d-1) (R_n-s_i-s_{i+1})} \cdot e^{-\half (d-1) (R_n-t_1-w)} \cdot  e^{-\half (d-1)(R_n-s_1-w)}  \\
&= C (n^{-1}e^{\half (d-1)\omega_n})^{2m-q-2} e^{\half (d-1)u} \prod_{i=1}^{m-2} e^{(d-1)t_i} \prod_{i=1}^{m-q-2}  e^{(d-1)s_i} \cdot e^{\half (d-1)s_{m-q-1}} \cdot e^{(d-1)w}. 
\end{align*}
Applying the final term back to \eqref{e:main.variance.clique}, while using Lemma \ref{l:Lemma1.OY}, 
\begin{align}
A_{n,q}(u) &\le C e^{\half (d-1)(2m-3)\omega_n}\cdot e^{\half(d-1) u} \int_{\mT(\bt,\bs,v)} \prod_{i=1}^{m-2} e^{(d-1)(1-\al)t_i} \prod_{i=1}^{m-q-2} e^{(d-1)(1-\al)s_i}   \label{e:main.variance.clique.2}\\
&\qquad\qquad\qquad\qquad\qquad\qquad\qquad \qquad  \times e^{\half (d-1)(1-2\al)s_{m-q-1}} \cdot e^{(d-1)(1-\al)w} \notag \\
&=C e^{\half (d-1)(2m-3)\omega_n}\cdot e^{\half(d-1) u} \int_{w=0}^u e^{(d-1)(1-\al)w}\int_{t_1=0}^w e^{(d-1)(1-\al)t_1}\int_{s_1=0}^w e^{(d-1)(1-\al)s_1} \notag \\
&\qquad \qquad \times \int_{(s_2,\dots,s_{m-q-1}):\,s_1 \le s_2 \le \cdots \le s_{m-q-1} \le u} \prod_{i=2}^{m-q-2} e^{(d-1)(1-\al)s_i}  \cdot e^{\half (d-1)(1-2\al)s_{m-q-1}}  \notag \\
&\qquad \qquad \times \int_{(t_2,\dots,t_{m-2}) :\, t_1 \le t_2 \le \cdots \le t_{m-2}\le s_{m-q-1}} \prod_{i=2}^{m-2} e^{(d-1)(1-\al)t_i}. \notag
\end{align}
Observe that 
\begin{align*}
\int_{(t_2,\dots,t_{m-2}) :\, t_1 \le t_2 \le \cdots \le t_{m-2}\le s_{m-q-1}} \prod_{i=2}^{m-2} e^{(d-1)(1-\al)t_i} &\le \Big( \int_{t_1}^{s_{m-q-1}} e^{(d-1)(1-\alpha)t} \dif t  \Big)^{m-3} \\
&\le Ce^{(d-1)(m-3)(1-\alpha)s_{m-q-1}}, 
\end{align*}
and also, 
\begin{align*}
&\int_{(s_2,\dots,s_{m-q-1}):\,s_1 \le s_2 \le \cdots \le s_{m-q-1} \le u} \prod_{i=2}^{m-q-2} e^{(d-1)(1-\al)s_i}  \cdot e^{\half (d-1)(1-2\al)s_{m-q-1}} \cdot e^{(d-1)(m-3)(1-\alpha)s_{m-q-1}} \\
&= \int_{s_{m-q-1}=s_1}^u e^{\half (d-1) (2m-5-(2m-4)\alpha)s_{m-q-1}} \int_{(s_2,\dots,s_{m-q-2}):\,s_1 \le s_2 \le \cdots \le s_{m-q-2}\le s_{m-q-1}} \prod_{i=2}^{m-q-2} e^{(d-1)(1-\al)s_i} \\
&\le Ce^{(d-1)(m-4)(1-\al)u} \cdot e^{\half (d-1)(2m-5-(2m-4)\al)s_1}. 
\end{align*}
Applying the last term to \eqref{e:main.variance.clique.2}, we conclude that 
\begin{align*}
A_{n,q}(u) &\le C e^{\half (d-1)(2m-3)\omega_n} \cdot e^{\half (d-1) (2m-7-(2m-8)\al)u} \\
&\quad \times \int_{w=0}^u e^{(d-1)(1-\al)w}\int_{t_1=0}^w e^{(d-1)(1-\al)t_1}\int_{s_1=0}^w e^{\half (d-1)(2m-3-(2m-2)\al)s_1} \\
&\le C e^{\half (d-1)(2m-3)\omega_n} \cdot e^{\half (d-1) (2m-3-(2m-4)\al)u}. 
\end{align*}
Referring this result back to the second term in \eqref{e:decomp.var.clique}, while  applying Proposition \ref{p:exp.clique} $(iii)$, 
\begin{align*}
\text{Var}\big( \mC_{m,n}(u) \big) &\le C \big( e^{c_0\omega_n}\cdot \mainu + e^{\half (d-1)(2m-3)\omega_n} \cdot e^{\half(d-1) (2m-3-(2m-4)\alpha)u}  \big)  \\
&\le C e^{c_1\omega_n}\cdot e^{\half(d-1) (2m-3-(2m-4)\alpha)u}. 
\end{align*}
\end{proof}

\subsection{Point process convergence}  \label{sec:pp.conv.clique}

The next result concerns the weak  convergence of the point process associated with $m$-clique counts constructed from the variables $X_i=(U_i, \Ta_i)$, $U_i=R_n-d_H(o,X_i)$, for $i=1,\dots,N_n$. 

\begin{proposition} \label{p:pp.conv.clique}
Let $\frac{2m-3}{2m-2}<\al<1$; then, as $n\to\infty$, 
$$
\sum_{i=1}^{N_n} \delta_{(i/N_n, \, b_{m,n}^{-1} \nu_{m,n}(U_i))}  \Rightarrow \ms{PPP}\big(\ms{Leb}\otimes \ms{m}_{2\al}\big), \  \ \text{in } M_p\big( [0,1]\times  (0,\infty ] \big), 
$$
and 
$$
\sum_{i=1}^{N_n} \delta_{(i/N_n, \, b_{m,n}^{-1} \mC_{m,n}(X_i))}  \Rightarrow \ms{PPP}\big(\ms{Leb}\otimes \ms{m}_{2\al}\big), \  \ \text{in } M_p\big( [0,1]\times  (0,\infty ] \big), 
$$
where the measure $\ms{m}_{2\al}$ is defined in Theorem \ref{t:joint.sum.max.star}. 
\end{proposition}

\begin{proof}
Since most of the argument  parallels the proofs of Proposition \ref{p:pp.conv.star} and Lemma \ref{l:PPD_n}, we state only  the proof of a statement analogous to \eqref{e:Cor6.1.and.Lemma6.1}:
$$
n\P \big( \nu_{m,n}(U_1)\ge yb_{m,n}, \, U_1 \le \gaR \big) \to \ms{m}_{2\al} \big( (y,\infty] \big)=y^{-2\al}, \ \ \ y >0, 
$$
where $\gamma\in \big(\frac{1}{2\al},1\big)$. 

It follows from Proposition \ref{p:exp.clique} $(i)$ that for every $\vep>0$, there exists $N\in \bbn$ such that for all $n\ge N $ and $u\in (0,\gaR)$, 
$$
\nu_{m,n}(u)\le (1+\vep) C_{m,\al} \mainu. 
$$
Using this bound and proceeding as in \eqref{e:upper.limit.mun}, we have as $n\to\infty$, 
$$
n\P \big( \nu_{m,n}(U_1)\ge yb_{m,n}, \, U_1 \le \gaR \big) \le n\int_{\frac{2}{d-1}\log (yb_{m,n} (1+\vep)^{-1}C_{m,\al}^{-1}) }^\gaR \radpdf (u)\dif u \sim \Big( \frac{y}{1+\vep} \Big)^{-2\al}, 
$$
which, in turn,  implies that 
$\limsup_{n\to\infty} n\P \big( \nu_{m,n}(U_1)\ge yb_{m,n}, \, U_1 \le \gaR \big) \le y^{-2\al}$ by letting $\vep \downarrow0$. 

Subsequently, Proposition \ref{p:exp.clique} $(ii)$ implies that  for any $\vep>0$, there exists $N\in \bbn$ such that for all $n\ge N$ and $u\in (0,\gaR)$, 
\begin{equation}  \label{e:lower.nu.n.exact.asym}
\nu_{m,n}(u) \ge \big(  (1-\vep) C_{m,\al}' \ell(u) -C\ell'(u)\big)_+ \mainu, 
\end{equation}
where $C_{m,\al}':= C_{m,\al} / \ell(\infty)$. From this it follows that 
\begin{align*}
&n\P \big( \nu_{m,n}(U_1)\ge yb_{m,n}, \, U_1 \le \gaR \big) \\
&\quad  \ge n \P \Big( \big(  (1-\vep) C_{m,\al}' \ell(U_1) -C\ell'(U_1)\big)_+ e^{\half (d-1) U_1}\ge yb_{m,n}, \, \omega_n \le U_1 \le \gaR \Big). 
\end{align*}
Recall that $\ell(u)$ (resp.~$\ell'(u)$) is increasing (resp.~decreasing) in $u$ such that  $\lim_{u\to\infty} \ell(u) = \ell(\infty) \in (0,\infty)$ and $\lim_{u\to\infty}\ell'(u)=0$; hence, there exists $N'\ge N$ so that for all $n\ge N'$, 
$$
\big(  (1-\vep) C_{m,\al}' \ell(U_1) -C\ell'(U_1)\big)_+ \ge (1-\vep) C_{m,\al}' \ell(\omega_n ) - C\ell'(\omega_n) \ge (1-2\vep)C_{m,\al}. 
$$
Thus, as $n\to\infty$, 
\begin{align*}
n\P \big( \nu_{m,n}(U_1)\ge yb_{m,n}, \, U_1 \le \gaR \big)  &\ge n\P \big( (1-2\vep)C_{m,\al} e^{\half (d-1)U_1} \ge yb_{m,n}, \, \omega_n \le U_1\le \gaR \big) \\
&= n\int_{\frac{2}{d-1}\log (yb_{m,n} (1-2\vep)^{-1}C_{m,\al}^{-1}) }^\gaR \radpdf (u)\dif u \sim \Big( \frac{y}{1-2\vep} \Big)^{-2\al}, 
\end{align*}
and, we obtain   $\liminf_{n\to\infty} n\P \big( \nu_{m,n}(U_1)\ge yb_{m,n}, \, U_1 \le \gaR \big) \ge  y^{-2\al}$ as desired. 
\end{proof}

\subsection{Proof of Theorem \ref{t:joint.sum.max.clique}}

As discussed in Section \ref{sec:main.result.star.count}, the first step toward proving \eqref{e:joint.weak.clique}  is to show that,  as $n\to\infty$,
\begin{align}
&\left( \frac{1}{b_{m,n}} \bigg( \sum_{i=1}^{[N_n\cdot\,]}\nu_{m,n}(U_i) - [N_n\cdot\,]  \E [\nu_{m,n}(U_1)] \bigg), \ \frac{1}{b_{m,n}}\bigvee_{i=1}^{[N_n \cdot \,]} \nu_{m,n}(U_i)\right)  \Rightarrow (S_{2\al}(\cdot),\, Y_{2\al}(\cdot)), \label{e:join.weak.nu.version}
\end{align}
in the space $D\big( [0,1], \R \times [0,\infty) \big)$. Given similarities between \eqref{e:join.weak.nu.version} and Proposition \ref{p:main.mu.version.star} $(i)$, our remaining task is to verify the analogue of \eqref{e:prop3.5(1)}, namely, for any $\eta>0$, 
\begin{align}
&\lim_{\vep\to 0}\limsup_{n\to\infty}\P \bigg(\sup_{0\le t \le 1} \bigg| \sum_{i=1 }^{[N_nt]}\nu_{m,n}(U_i)\one \big\{  \nu_{m,n}(U_i)\le  \vep b_{m,n}\big\}\label{e:prop3.5(1).clique} \\
&\qquad \qquad \qquad \qquad \qquad \qquad  -[N_nt] \E \Big[  \nu_{m,n}(U_1)\one \big\{  \nu_{m,n}(U_1)\le \vep b_{m,n}\big\} \Big] \bigg|>\eta b_{m,n} \bigg)=0. \notag
\end{align}
To show this, as discussed in \eqref{e:chebyshev}, we need to prove that 
\begin{equation}  \label{e:analog.second.moment.convergence}
\lim_{\vep\downarrow0} \limsup_{n\to\infty} \frac{n}{b_{m,n}^2}\, \int_0^{R_n} \nu_{m,n}(u)^2 \one \big\{ \nu_{m,n}(u)\le \vep b_{m,n}\big\} \radpdf (u)\dif u =0. 
\end{equation}
First, fix the constant $\gamma\in \big( \frac{1}{2\al}, 1 \big)$.   By Proposition \ref{p:exp.clique} $(ii)$, we see that 
$\nu_{m,n}(\gaR)\ge 2^{-1}C_{m,\al} e^{\half (d-1)\gaR} = 2^{-1} C_{m,\al} n^\gamma$ 
for sufficiently large $n$, and thus, $\nu_{m,n}(\gaR)/b_{m,n}\to\infty$ as $n\to\infty$. Now, since $\nu_{m,n}(u)$ is increasing in $u$, the condition $\nu_{m,n}(u)\le \vep b_{m,n} $ in \eqref{e:analog.second.moment.convergence} immediately implies $u\le \gaR$. Hence, we can replace the upper limit ``$R_n$" of the integral in \eqref{e:analog.second.moment.convergence} with ``$\gaR$."

For every $M>0$, we can see that as $n\to\infty$, 
$$
\frac{n}{b_{m,n}^2}\, \int_0^{M} \nu_{m,n}(u)^2 \one \big\{ \nu_{m,n}(u)\le \vep b_{m,n} \big\}  \radpdf (u)\dif u \le \frac{Cn}{b_{m,n}^2} \int_0^M e^{(d-1)(1-\al)u} \dif u \le Cn^{1-1/\al} \to 0. 
$$
Furthermore, Proposition \ref{p:exp.clique} $(ii)$ ensures that there exists $N\in \bbn$ such that for all $n\ge N$ and $u\in (M,\gaR)$, 
$$
\nu_{m,n}(u) \ge C_{m,\al}''(M) \mainu, 
$$
where we have chosen $M>0$ so large that 
$C_{m,\al}''(M):= 2^{d-2} \xi_{d-1} \big( s_{d-1}^{-1}\al (d-1) \big)^{m-1} \ell(M)  -C \ell'(M)>0$. It then follows that 
\begin{align*}
&\frac{n}{b_{m,n}^2}\, \int_M^{\gaR} \nu_{m,n}(u)^2 \one \big\{ \nu_{m,n}(u)\le \vep b_{m,n} \big\}  \radpdf (u)\dif u \\
&\le Cn^{1-1/\al} \int_0^\gaR e^{(d-1)(1-\al)u}\one \big\{ C_{m,\al}''(M) \mainu \le \vep b_{m,n} \big\}\dif u\\
&\le Cn^{1-1/\al} \Big( \frac{\vep b_{m,n}}{C_{m,\al}''(M)} \Big)^{2(1-\al)} = C \Big(  \frac{\vep C_{m,\al}}{C_{m,\al}''(M)}\Big)^{2(1-\al)}. 
\end{align*}
The final term vanishes as $\vep\downarrow0$. 

As in Section \ref{sec:main.result.star.count}, the proof of Theorem \ref{t:joint.sum.max.clique} can be completed as a result of the following lemmas. 

\begin{lemma}  \label{l:marginal.max.clique}
As $n\to\infty$, 
$$
\sup_{0\le t \le 1}\bigg| b_{m,n}^{-1}\bigvee_{i=1}^{[N_n t]} \mC_{m,n}(X_i) - b_{m,n}^{-1} \bigvee_{i=1}^{[N_n t]} \nu_{m,n}(U_i)\bigg| \stackrel{\P}{\to} 0. 
$$
\end{lemma}

\begin{lemma}  \label{l:stable.limit.thm.C}
As $n\to\infty$, it holds that 
$$
\sup_{0\le t \le 1}\bigg|  \frac{1}{b_{m,n}} \Big( \sum_{i=1}^{[N_n t]}\mC_{m,n}(X_i) -[N_n t] \E \big[\mC_{m,n}(X_1) \big]\Big) - \frac{1}{b_{m,n}} \Big( \sum_{i=1}^{[N_n t]}\nu_{m,n}(U_i) -[N_nt]  \E \big[  \nu_{m,n}(U_1) \big]\Big) \bigg| \stackrel{\P}{\to} 0. 
$$
\end{lemma}

\begin{proof}[Proof of Lemma \ref{l:marginal.max.clique}]
As indicated by \eqref{e:max.difference}, it remains to verify that for any $\vep>0$,
$$
n\int_0^{R_n}\P\Big(\big| \mC_{m,n}(u) - \nu_{m,n}(u) \big| >\vep b_{m,n}  \Big)\radpdf (u) \dif u \to 0, \ \ \ n\to\infty. 
$$
By Chebyshev's inequality, Proposition \ref{p:var.clique}, and Lemma \ref{l:Lemma1.OY}, the expression above  is bounded by 
\begin{align*}
\frac{n}{\vep^2 b_{m,n}^2}\, \int_0^{R_n} \text{Var} \big(  \mC_{m,n}(u)\big)\radpdf (u)\dif u &\le Cn^{1-1/\al} e^{\half (d-1)(2m-3)\omega_n} \int_0^{R_n} e^{\half (d-1)(2m-3-(2m-2)\al)u}\dif u \\
&\le Cn^{1-1/\al}  e^{\half (d-1)(2m-3)\omega_n} \to 0, \ \ \text{as } n\to\infty. 
\end{align*}
\end{proof}

\begin{proof}[Proof of Lemma \ref{l:stable.limit.thm.C}]
First, choose a constant $c_1$ so that 
\begin{equation}  \label{e:constraint.c1.clique}
0 < \frac{2\al-1}{\al(2m\al-2m+3)} <c_1< \frac{1}{2\al}. 
\end{equation}
Using this constant, it suffices to prove a series of conditions analogous to \eqref{e:main(i).cond1}--\eqref{e:main(i).cond4}. Specifically, for every $\eta>0$, we claim that as $n\to\infty$, 
\begin{equation}   \label{e:main(i).cond1.clique}
\P \bigg( \sup_{0\le t \le 1} \bigg| \sum_{i=1}^{[N_nt]} \big(\mC_{m,n}(X_i) -\nu_{m,n}(U_i) \big)\one \{ U_i >c_1 R_n \} \bigg|>\eta b_{m,n} \bigg) \to 0, 
\end{equation}
\begin{equation}  \label{e:main(i).cond2.clique}
\P \bigg( \sup_{0\le t \le 1} [N_nt] \Big| \, \E \big[ \mC_{m,n}(X_1)\one\{ U_1>c_1 R_n \} \big] - \E \big[ \nu_{m,n}(U_1)\one\{ U_1>c_1 R_n \} \big] \, \Big| >\eta b_{m,n} \bigg) \to 0, 
\end{equation}
\begin{equation}  \label{e:main(i).cond3.clique}
\P \bigg( \sup_{0\le t \le 1} \bigg| \sum_{i=1}^{[N_nt]} \Big(\nu_{m,n}(U_i)\one \{ U_i \le c_1 R_n \} -\E \big[ \nu_{m,n}(U_i)\one \{ U_i \le c_1 R_n \} \big] \Big) \bigg|>\eta b_{m,n} \bigg) \to 0, 
\end{equation}
and
\begin{equation}  \label{e:main(i).cond4.clique}
\P \bigg( \sup_{0\le t \le 1} \bigg| \sum_{i=1}^{[N_nt]} \Big(\mC_{m,n}(X_i)\one \{ U_i \le c_1 R_n \} -\E \big[ \mC_{m,n}(X_i)\one \{ U_i \le c_1 R_n \} \big] \Big) \bigg|>\eta b_{m,n} \bigg) \to 0. 
\end{equation}
First, by proceeding as in \eqref{e:diff.big.u}, the probability in \eqref{e:main(i).cond1.clique} can be bounded above  by 
$$
Cn^{1-\frac{1}{2\al}} \int_{c_1R_n}^{R_n} \sqrt{\text{Var}\big( \mC_{m,n}(u) \big)} e^{-\al(d-1)u} \dif u. 
$$
By Proposition \ref{p:var.clique}, this term is further bounded by 
\begin{align*}
&Cn^{1-\frac{1}{2\al}} e^{\half (d-1)(m-\frac{3}{2})\omega_n} \int_{c_1R_n}^{R_n} e^{\half (d-1)(m-\frac{3}{2}-m\al)u} \dif u \le Cn^{1-\frac{1}{2\al}+c_1 (m-\frac{3}{2}-m\al)} e^{\half (d-1)(m-\frac{3}{2})\omega_n} \to 0,
\end{align*}
where the last convergence follows from \eqref{e:constraint.c1.clique}. 

Next, the proof of  \eqref{e:main(i).cond3.clique} is straightforward and  proceeds in essentially the same manner    as the proof of    \eqref{e:main(i).cond3}, so we omit  the details. 

The argument for \eqref{e:main(i).cond2.clique} is  likewise analogous to that for \eqref{e:main(i).cond2}. Accordingly, we present only the proof of
$$
\frac{n}{b_{m,n}} \int_{c_1R_n}^{R_n} \E \big[ \mC_{m,n}(u,\Pn) -  \mC_{m,n} (u, \Pn \setminus \{X_1\}) \big]\radpdf (u)\dif u \to 0, \ \ \ n\to\infty, 
$$
which corresponds to \eqref{e:diff.Poissonization} but requires some  modification.
In particular, the expression \eqref{e:difference.D.X1} must be replaced by the bound
\begin{align*}
&\mC_{m,n}(u,\Pn) - \mC_{m,n}(u, \Pn\setminus \{X_1\}) \\
&\le  \one \big\{  d_H(X_1,p)\le R_n, \, U_1 \le u \big\}\hspace{-25pt} \sum_{(Y_1,\dots, Y_{m-2})\in (\{ X_2,\dots,X_{N_n} \})_{\neq}^{m-2}} \hspace{-30pt}\one \{ (Y_1,\dots, Y_{m-2}, p) \text{ forms an }(m-1)\text{-clique}, \\
&\qquad \qquad \qquad \qquad \qquad \qquad \qquad \qquad \qquad \qquad \qquad V_1 \le \cdots \le V_{m-2}\le u\}, 
\end{align*}
where $V_i = R_n-d_H(o,Y_i)$ for $i=1,\dots,m-2$. Taking expectations on both sides, 
\begin{align}
&\E \big[ \mC_{m,n}(u,\Pn) - \mC_{m,n}(u, \Pn\setminus \{X_1\}) \big] \label{e:exp.diff.C}\\
&\le \P\big(  d_H(X_1,p)\le R_n, \, U_1 \le u \big) \E\bigg[ \sum_{(Y_1,\dots, Y_{m-2})\in (\Pn)_{\neq}^{m-2}} \hspace{-15pt}\one \{  (Y_1,\dots, Y_{m-2}, p) \text{ forms an }(m-1)\text{-clique}, \notag \\
&\qquad \qquad \qquad \qquad \qquad \qquad \qquad \qquad \qquad \qquad\qquad \qquad V_1 \le \cdots \le V_{m-2}\le u\} \bigg]. \notag 
\end{align}
By \eqref{e:ignore.restriction.omega} and Lemma \ref{l:Lemma1.OY}, 
\begin{align*}
\P\big(  d_H(X_1,p)\le R_n, \, U_1 \le u \big) &\le C\int_0^u e^{\half (d-1)\omega_n} \cdot e^{-\half (d-1)(R_n-t-u)}\cdot e^{-\al(d-1)t}\dif t\\
&\le Cn^{-1} e^{\half (d-1)\omega_n}\cdot \mainu. 
\end{align*}
On the other hand,  the expectation term on the right-hand side of \eqref{e:exp.diff.C} represents  the expected number of $(m-1)$-cliques in which $p$ is a fixed  vertex closest to the  origin of $\B_d$. Hence, by Proposition \ref{p:exp.clique} $(iii)$,  this can be bounded by $Ce^{c_0\omega_n}\cdot \mainu$ for some constants $c_0, C \in (0,\infty)$. Combining these bounds, 
\begin{align*}
&\frac{n}{b_{m,n}} \int_{c_1R_n}^{R_n} \E \big[ \mC_{m,n}(u,\Pn) -  \mC_{m,n} (u, \Pn \setminus \{X_1\}) \big]\radpdf (u)\dif u\\
&\le Cn^{1-\frac{1}{2\al}} \int_{c_1R_n}^{R_n} n^{-1} e^{\half (d-1)\omega_n}\cdot \mainu \cdot e^{c_0\omega_n}\cdot \mainu \cdot e^{-\al(d-1)u}\dif u \\
&=C n^{-\frac{1}{2\al}} e^{((d-1)/2+ c_0)\omega_n} \int_{c_1R_n}^{R_n} e^{(d-1)(1-\al)u}\dif u\\
&\le C n^{-\frac{2}{\al}(\al-\half)^2} e^{((d-1)/2+ c_0)\omega_n} \to 0, \ \ \text{as } n\to\infty. 
\end{align*}

Finally, the proof of \eqref{e:main(i).cond4.clique} requires a separate argument and is  moved to  Lemma \ref{l:Pruss.maximal.inequ.clique}.
\end{proof}
\medskip

\section{Limit theorem for the global clustering coefficient}   \label{sec:clustering.coeff}

In this short section we discuss the limit theorem for the \emph{global clustering coefficient} defined by 
$$
\ms{CC}_n :=  \frac{3\sum_{i=1}^{N_n} \mC_{3,n}(X_i)}{\sum_{i=1}^{N_n} \mD_{3,n}(X_i)}, 
$$
where $\mD_{3,n}(X_i)$  was already defined in Section \ref{sec:star.shape} as the number of pairs $(Y,Z)\in (\Pn)_{\neq}^2$ satisfying either $Y\to X_i$, $Z\to X_i$ or $Z\to X_i \to Y$ or $X_i \to Y$, $X_i \to Z$, without any condition on whether  $Y$ and $Z$ are connected. 
Moreover, $\sum_{i=1}^{N_n} \mC_{3,n}(X_i)$ is the number of $3$-cliques, that is, the triangle count, so that $\mC_{3,n}(X_i)$ is the number of pairs $(Y,Z)\in (\Pn)_{\neq}^2$  such that $Y\to X_i$, $Z\to X_i$, and  $Y\to Z$. 

\begin{theorem}  \label{t:clustering.coeff}
Let $\frac{3}{4}<\al<1$. As $n\to\infty$, 
$$
B_{3,\al} C_{3,\al}^{-1}n^{\frac{1}{2\al}} \bigg( \ms{CC}_n - \frac{3N_n \E[\mC_{3,n}(X_1)]}{\sum_{i=1}^{N_n}\mD_{3,n}(X_i)} \bigg) \Rightarrow \frac{3S_{2\al}(1)}{\widetilde S_\al (1)},
$$
where $B_{3,\al}$ and $C_{3,\al}$ are constants defined before Theorems \ref{t:joint.sum.max.star} and \ref{t:joint.sum.max.clique}, and the process $(\widetilde S_\al(\cdot), S_{2\al}(\cdot))$ is a {bivariate L\'evy process with perfectly dependent stable jumps}, defined by 
\begin{align}
\E \big[ e^{i(v_1\widetilde S_\al(t) + v_2 S_{2\al}(t))} \big] &=  \exp\Big\{  t\int_{(0,\infty)^2} \big( e^{i (v_1 x_1 + v_2 x_2)} -1- iv_2 x_2 \big) \kappa (\dif x_1, \dif x_2) \Big\}  \label{e:ch.f.bivariate.perfect.jumps}\\
&= \exp\Big\{  t\int_0^\infty \big( e^{i (v_1 y + v_2 y^{1/2})} -1- iv_2 y^{1/2} \big) \ms{m}_\al (\dif y) \Big\},  \ \ \ t \ge 0, \ v_i \in \R, \, i=1,2. \notag  
\end{align}
Here,  $\kappa$ is a L\'evy measure concentrated on the curve $\{(y,y^{1/2}):y>0\}$, and it satisfies   $\kappa\big( (y_1,\infty] \times (y_2,\infty] \big) = y_1^{-\al}\wedge y_2^{-2\al}$ for  $y_1, y_2 >0$. 
\end{theorem}

\begin{remark}
From \eqref{e:ch.f.bivariate.perfect.jumps}, it is immediate that $\widetilde S_\al(\cdot)$ is an $\al$-stable subordinator and $S_{2\al}(\cdot)$ is a zero-mean $2\al$-stable L\'evy process. In particular, the second line in \eqref{e:ch.f.bivariate.perfect.jumps} follows from the relation $\kappa= \ms{m}_\al \circ T^{-1}$, where $T(y)=(y,y^{1/2})$ for $y>0$. 
\end{remark}

\begin{remark}
The centering in Theorem \ref{t:clustering.coeff} is random because the denominator $\sum_{i=1}^{N_n} \mD_{3,n}(X_i)$ lies in the domain of attraction of an $\al$-stable subordinator and therefore does not admit a deterministic centering compatible with the scaling $a_{3,n}$. The random centering isolates the leading self-normalized term of $\ms{CC}_n$ and yields a non-degenerate stable limit. From a statistical perspective, this does not necessarily pose a difficulty, especially when the centering term can be estimated from the data. 
\end{remark}

The weak convergence in Theorem \ref{t:clustering.coeff} follows from the propositions stated below.

\begin{proposition}  \label{p:pp.conv.clustering.coeff}
As $n\to\infty$, 
$$
\sum_{i=1}^{N_n} \delta_{(i/N_n,\, a_{3,n}^{-1}\mu_{3,n}(U_i), \, b_{3,n}^{-1}\nu_{3,n}(U_i))} \Rightarrow \ms{PPP} (\ms{Leb}\otimes \kappa), \ \ \text{in } M_p\big( [0,1]\times \big([0,\infty]^2 \setminus \{ {\bf 0}\}\big) \big). 
$$
\end{proposition}

\begin{proposition}  \label{p:joint.weak.clustering.coeff}
We have, in the space $D\big( [0,1], [0,\infty) \times \R \big)$, 
\begin{align*}
&\left( \frac{1}{a_{3,n}}\sum_{i=1}^{[N_n\cdot\,]}\mD_{3,n}(X_i), \  \frac{1}{b_{3,n}} \bigg( \sum_{i=1}^{[N_n\cdot\,]}\mC_{3,n}(X_i) - [N_n\cdot\,]  \E [\mC_{3,n}(X_1)] \bigg) \right) \Rightarrow \big( \widetilde S_\al(\cdot), S_{2\al}(\cdot) \big), \ \ \ n\to\infty. 
\end{align*}
\end{proposition}

Assuming   Proposition \ref{p:joint.weak.clustering.coeff} for  the moment, we  first complete the proof of Theorem \ref{t:clustering.coeff}.

\begin{proof}[Proof of Theorem \ref{t:clustering.coeff}]
Specializing Proposition \ref{p:joint.weak.clustering.coeff} to the case $t=1$ and applying the  continuous mapping theorem,
$$
\frac{a_{3,n}}{b_{3,n}}\bigg( \ms{CC}_n - \frac{3N_n \E[\mC_{3,n}(X_1)]}{\sum_{i=1}^{N_n}\mD_{3,n}(X_i)} \bigg) \Rightarrow \frac{3S_{2\al}(1)}{\widetilde S_\al (1)}, \ \ \text{as } n\to\infty. 
$$
Note that by definition,  $a_{3,n} b_{3,n}^{-1} = B_{3,\al}C_{3,\al}^{-1}n^{\frac{1}{2\al}}$. 
\end{proof}

\begin{proof}[Proof of Proposition \ref{p:pp.conv.clustering.coeff}]
Arguing analogously to the proofs of Propositions \ref{p:pp.conv.star} and \ref{p:pp.conv.clique},  it suffices to show that for $\gamma\in \big( \frac{1}{2\al}, 1 \big)$, 
$$
\sum_{i=1}^{n}\one\big\{ U_i \le \gaR \big\} \delta_{(i/n,\, a_{3,n}^{-1}\mu_{3,n}(U_i), \, b_{3,n}^{-1}\nu_{3,n}(U_i))} \Rightarrow \ms{PPP} (\ms{Leb}\otimes \kappa), \ \ \text{as } n\to\infty. 
$$
By Corollary 6.1 and Lemma 6.1 in \cite{resnick:2007}, this weak  convergence is equivalent to 
$$
n\P\big( \mu_{3,n}(U_1)\ge y_1 a_{3,n}, \, \nu_{3,n}(U_1)\ge y_2 b_{3,n}, \,  U_1 \le \gaR \big) \to \kappa\big( (y_1,\infty]\times (y_2,\infty] \big) = y_1^{-\al}\wedge y_2^{-2\al}
$$
for all $y_1, y_2 >0$. 

According to Proposition \ref{p:exp.mu.n} $(i)$ and Proposition \ref{p:exp.clique} $(i)$ (with $k=m=3$), we obtain that,  for any $\vep>0$, there exists  $N\in \bbn$, such that for all $n\ge N$ and $u \in (0,\gaR)$, 
$$
\mu_{3,n}(u) \le (1+\vep) B_{3,\al} \Big( 1+ \sum_{\ell=0}^1 \binom{2}{\ell} (e^{\half (d-1)(1-2\al)u} + s_N)^{2-\ell} \Big)e^{(d-1)u}, \  \ \ \nu_{3,n}(u) \le (1+\vep) C_{3,\al}\mainu, 
$$
and $\sum_{\ell=0}^1\binom{2}{\ell}(2s_N)^{2-\ell}\le \vep$. 
Choosing $M\in (0,\infty)$ so large that $e^{\half (d-1)(1-2\al)M} \le s_N$, it follows that  for all $n\ge N$, 
\begin{align*}
&n\P\big( \mu_{3,n}(U_1)\ge y_1 a_{3,n}, \, \nu_{3,n}(U_1)\ge y_2 b_{3,n}, \,  U_1 \le \gaR \big) \\
&\le n\P\big( (1+\vep)^2 B_{3,\al} e^{(d-1)U_1} \ge y_1 a_{3,n}, \, (1+\vep) C_{3,\al} e^{\half (d-1)U_1} \ge y_2 b_{3,n}, \, M\le U_1 \le \gaR \big)\\
&= n\P\Big( \frac{1}{d-1} \log \Big\{ n^{1/\al} \Big( \frac{y_1}{(1+\vep)^2}\vee \frac{y_2^2}{(1+\vep)^2}  \Big) \Big\}\le U_1 \le \gaR \Big) \\
&=n \int_{\frac{1}{d-1} \log \{ n^{1/\al} (\frac{y_1}{(1+\vep)^2} \vee \frac{y_2^2}{(1+\vep)^2}) \}}^\gaR \radpdf(u)\dif u \to \Big( \frac{y_1}{(1+\vep)^2} \Big)^{-\al} \wedge
 \Big( \frac{y_2^2}{(1+\vep)^2} \Big)^{-\al}, \ \ \ n\to\infty. 
\end{align*}
Thus, by letting $\vep\downarrow 0$, 
$$
\limsup_{n\to\infty} n\P\big( \mu_{3,n}(U_1)\ge y_1 a_{3,n}, \, \nu_{3,n}(U_1)\ge y_2 b_{3,n}, \,  U_1 \le \gaR \big)  \le y_1^{-\al} \wedge y_2^{-2\al}. 
$$ 

For the lower bound, by appealing to Proposition \ref{p:exp.mu.n} $(ii)$ and Proposition \ref{p:exp.clique} $(ii)$ (specifically, \eqref{e:mu_n.upper.lower.bdd} and \eqref{e:lower.nu.n.exact.asym} with $k=m=3$), and repeating the same argument, we also obtain
$$
\liminf_{n\to\infty} n\P\big( \mu_{3,n}(U_1)\ge y_1 a_{3,n}, \, \nu_{3,n}(U_1)\ge y_2 b_{3,n}, \,  U_1 \le \gaR \big)  \ge y_1^{-\al} \wedge y_2^{-2\al}. 
$$ 
\end{proof}

\begin{proof}[Proof of Proposition \ref{p:joint.weak.clustering.coeff}]
Write $\sum_{\ell} \delta_{(t_\ell, \, j_\ell^{(1)}, \, j_\ell^{(2)})}$ be the $\ms{PPP}(\ms{Leb}\otimes \kappa)$. Given $\vep>0$, we define the continuous map $\hat U_\vep :  M_p\big( [0,1]\times \big( [0,\infty]^2\setminus \{{\bf 0}\} \big) \big) \to D\big( [0,1], [0,\infty)^2 \big)$ by 
$$
\hat U_\vep\Big( \sum_{\ell} \delta_{(s_\ell, \, z_\ell^{(1)}, \, z_\ell^{(2)})} \Big) = \bigg( \sum_{\ell: \, s_\ell \le \, \cdot} z_\ell^{(1)} \one \{z_\ell^{(1)}> \vep\}, \  \sum_{\ell: \, s_\ell \le \, \cdot} z_\ell^{(2)} \one \{z_\ell^{(2)}> \vep\}\bigg). 
$$
Applying the continuous mapping theorem to the weak convergence in Proposition \ref{p:pp.conv.clustering.coeff}, while using the convergence 
$$
\frac{[N_n  \cdot\,]}{b_{3,n}}\, \E\Big[ \nu_{3,n}(U_1)\one \big\{ \nu_{3,n}(U_1)>\vep b_{3,n} \big\}  \Big] \Rightarrow (\cdot) \int_\vep^\infty x \ms{m}_{2\al}(\dif x), \ \ \text{in } D\big( [0,1], [0,\infty) \big), 
$$
we have that 
\begin{align*}
&\bigg( \frac{1}{a_{3,n}}\sum_{i=1}^{[N_n\cdot\,]}\mu_{3,n}(U_i)\one \big\{ \mu_{3,n}(U_i)>\vep a_{3,n} \big\}, \\
&\qquad \qquad \qquad  \frac{1}{b_{3,n}}  \sum_{i=1}^{[N_n\cdot\,]}\nu_{3,n}(U_i) \one \big\{ \nu_{3,n}(U_i)> \vep b_{3,n} \big\} - \frac{[N_n  \cdot\,]}{b_{3,n}}\, \E\Big[ \nu_{3,n}(U_1)\one \big\{ \nu_{3,n}(U_1)>\vep b_{3,n} \big\}  \Big]\bigg) \\
&\Rightarrow \Big(  \sum_{\ell: \, t_\ell \le \, \cdot} j_\ell^{(1)} \one \{j_\ell^{(1)}> \vep\}, \  \sum_{\ell: \, t_\ell \le \, \cdot} j_\ell^{(2)} \one \{j_\ell^{(2)}> \vep\} -(\cdot) \int_\vep^\infty x \ms{m}_{2\al}(\dif x) \Big), \\
&\quad = \Big(  \sum_{\ell: \, t_\ell \le \, \cdot} (j_\ell^{(1)},\,  j_\ell^{(2)}) \one \{j_\ell^{(1)} \wedge j_\ell^{(2)} > \vep\} -(\cdot) \int_{(\vep,\infty)^2} (0,x_2) \kappa(\dif x_1, \dif x_2) \Big) \ \ \text{in } D \big( [0,1], [0,\infty)\times \R \big). 
\end{align*}
The weak limit above converges weakly to the process $\big( \widetilde S_\al (\cdot), S_{2\al}(\cdot) \big)$ in the space $D \big( [0,1], [0,\infty)\times \R \big)$. More precisely, this follows from the fact that,   for any $v=(v_1,v_2)\in \R^2$ and $t\ge0$, 
\begin{align*}
&\E \big[ e^{i \langle v, \, \lim_{\vep\downarrow0} \sum_{\ell: \, t_\ell \le t}(j_\ell^{(1)},\,  j_\ell^{(2)}) \one \{j_\ell^{(1)} \wedge j_\ell^{(2)} > \vep\} -t \int_{(\vep,\infty)^2} (0,x_2) \kappa(\dif x_1, \dif x_2)  \rangle } \big] \\
&= \lim_{\vep\downarrow0} \E \big[ e^{i \langle v, \, \sum_{\ell: \, t_\ell \le t} (j_\ell^{(1)},\,  j_\ell^{(2)}) \one \{j_\ell^{(1)} \wedge j_\ell^{(2)} > \vep\}  \rangle } \big] e^{-t\int_{(\vep,\infty)^2} iv_2 x_2 \kappa(\dif x_1, \dif x_2) }\\
&= \exp \Big\{ t\int_{(0,\infty)^2} \big( e^{i(v_1 x_1+v_2x_2)}-1-iv_2x_2 \big) \kappa(\dif x_1, \dif x_2) \Big\} = \E \big[ e^{i(v_1\widetilde S_\al(t) + v_2 S_{2\al}(t))} \big]. 
\end{align*}
Now,  \eqref{e:2nd.conv.subordinator} and  \eqref{e:prop3.5(1).clique} imply that
\begin{align} 
&\bigg( \frac{1}{a_{3,n}}\sum_{i=1}^{[N_n\cdot\,]}\mu_{3,n}(U_i), \  \frac{1}{b_{3,n}}  \sum_{i=1}^{[N_n\cdot\,]}\nu_{3,n}(U_i) - \frac{[N_n  \cdot\,]}{b_{3,n}}\, \E\big[ \nu_{3,n}(U_1)  \big]\bigg)  \label{e:weak.conv.to.perfect.stable.jumps} \\
&\quad \Rightarrow \big( \widetilde S_\al (\cdot), S_{2\al}(\cdot) \big),  \, \text{in } D \big( [0,1], [0,\infty)\times \R \big). \notag 
\end{align}
Finally, by Lemma \ref{l:stable.limit.thm.D} $(ii)$ and Lemma  \ref{l:stable.limit.thm.C}, we may replace $\mu_{3,n}(U_i)$ and $\nu_{3,n}(U_i)$ in \eqref{e:weak.conv.to.perfect.stable.jumps} by $\mD_{3,n}(X_i)$ and $\mC_{3,n}(X_i)$, respectively. This yields the desired weak convergence. 
\end{proof}
\medskip

\section{Appendix}

\subsection{Fundamental properties of hyperbolic random geometric graphs}

\begin{lemma}[Lemma 1 in \cite{owada:yogeshwaran:2022}]  \label{l:Lemma1.OY}
$(i)$ As $n\to\infty$, we have 
$$
\bar \rho_{n,\al} (t) \le \big( 1+o_n(1) \big) \al (d-1)e^{-\al(d-1)t}
$$
uniformly for $t\in (0,R_n)$. \\ 
$(ii)$ For every $0<\gamma<1$, we have, as $n\to\infty$, 
$$
\bar \rho_{n,\al} (t) = \big( 1+o_n(1) \big) \al (d-1)e^{-\al(d-1)t}
$$
uniformly for $t\in (0,\gaR)$. 
\end{lemma}

\begin{lemma}[Lemma 3 in \cite{owada:yogeshwaran:2022}] \label{l:Lamma3.OY} Given $x_1, x_2\in B(o,R_n)$, let $\ta_{12}\in [0,\pi]$ be the relative angle between  $x_1$ and $x_2$, which are  viewed respectively as vectors emanating from the origin  of $\B_d$. Define $t_i=R_n-d_H(0,x_i)$, $i=1,2$, and set $\hat \ta_{12} = \big( e^{-2(R_n-t_1)} +e^{-2(R_n-t_2)} \big)^{1/2}$. If $\hat \ta_{12}/\ta_{12}$ vanishes as $n\to\infty$, then 
$$
d_H(x_1, x_2) = 2R_n -(t_1+t_2) +2 \log \sin (\ta_{12}/2) + O\big( (\hat \ta_{12} / \ta_{12})^2 \big), \ \ \ n\to\infty, 
$$
uniformly for all $x_1, x_2$ with $t_1+t_2\le R_n-\omega_n$. 
\end{lemma}

\begin{lemma}[Lemma 4 in \cite{owada:yogeshwaran:2022}]  \label{l:Lemma4.OY}
Let $X_1=(t_1,\Theta_1)$, $X_2=(t_2, \Theta_2)$, where $\Theta_1$, $\Theta_2$ are i.i.d.~random vectors on $\mA_d$ with density \eqref{e:uniform.angular} and $t_1, t_2$ are deterministic, representing the hyperbolic distance from the boundary of $B(o,R_n)$. Then, 
$$
\P \big( d_H(X_1,X_2)\le R_n \big) = \big( 1+o_n(1) \big)\frac{2^{d-1}}{(d-1)\kappa_{d-2}}\, e^{-\frac{1}{2}(d-1)(R_n-t_1-t_2)}, \ \ \ n\to\infty, 
$$
uniformly on $\{ (t_1, t_2): t_1 + t_2\le R_n-\omega_n\}$. 
\end{lemma}
\medskip

\subsection{Proof of \eqref{e:PPD_n}}

\begin{lemma}  \label{l:PPD_n}
Let $\half <\al< k-1$; then, as $n\to\infty$, 
$$
\sum_{i=1}^{N_n} \delta_{(i/N_n, \, a_{k,n}^{-1} \mD_{k,n}(X_i))}  \Rightarrow \ms{PPP}\big(\ms{Leb}\otimes \ms{m}_{2\al/(k-1)}\big), \  \ \text{in } M_p\big( [0,1]\times  (0,\infty ] \big). 
$$
\end{lemma}

\begin{proof}
We  fix the constant $\gamma\in \big(\frac{1}{2\al},1\big)$. By virtue of \eqref{e:direct.conv.mu_n}, it is sufficient to prove that as $n\to\infty$, 
\begin{align}
&d_{\ms{v}} \bigg(\sum_{i=1}^{N_n} \delta_{(i/N_n, \, a_{k,n}^{-1} \mD_{k,n}(X_i))},  \ \sum_{i=1}^{N_n}\one \big\{ U_i\le \gaR \big\} \, \delta_{(i/N_n, \, a_{k,n}^{-1} \mD_{k,n}(X_i))} \bigg) \stackrel{\P}{\to} 0, \label{e:vagueP1} \\
&d_{\ms{v}} \bigg(\sum_{i=1}^{N_n} \one \big\{ U_i\le \gaR \big\}\, \delta_{(i/N_n, \, a_{k,n}^{-1} \mD_{k,n}(X_i))}, \ \sum_{i=1}^{N_n} \one \big\{ U_i\le \gaR \big\} \, \delta_{(i/N_n, \, a_{k,n}^{-1} \mu_{k,n}(U_i))}  \bigg) \stackrel{\P}{\to} 0.  \label{e:vagueP2} 
\end{align}
Since \eqref{e:vagueP1} can be established similarly to  \eqref{e:vagueP3}, we can omit its  proof.  To show \eqref{e:vagueP2}, our goal is to verify that 
\begin{equation}  \label{e:Dn.mu.n.PP}
\sum_{i=1}^{N_n} \one \big\{ U_i\le \gaR \big\} \, \one\{ i/N_n\in B \}\Big(  f\big( a_{k,n}^{-1}\mD_{k,n}(X_i)\big)- f\big( a_{k,n}^{-1}\mu_{k,n}(U_i) \big) \Big)\stackrel{\P}{\to} 0, 
\end{equation}
for every $B\subset [0,1]$ with $\ms{Leb}(\partial B)=0$, and any $f\in C_K^+\big( (0,\infty] \big)$. Let $\eta_0$ be a positive constant such that the support of $f$ is contained in $[\eta_0,\infty]$. Choose $\eta' \in (0,\eta_0/2)$. For every $\eta \in (0,\eta_0/2)$, we have
\begin{align*}
&\E \bigg[ \, \bigg|\sum_{i=1}^{N_n} \one \big\{ U_i\le \gaR \big\} \, \one\{ i/N_n\in B \}\Big(  f\big( a_{k,n}^{-1}\mD_{k,n}(X_i)\big)- f\big( a_{k,n}^{-1}\mu_{k,n}(U_i) \big) \Big) \, \bigg| \bigg] \\
&\le n\int_{0}^{\gaR} \E \Big[ \Big| f\big(a_{k,n}^{-1}\mD_{k,n}(u)\big) -f\big( a_{k,n}^{-1} \mu_{k,n}(u) \big)  \Big| \Big] \, \bar \rho_{n,\al}(u) \dif u \\
&\le 2\| f \|_\infty n \int_0^{\gaR} \P  \Big( \big| \mD_{k,n}(u)-\mu_{k,n}(u) \big| >\eta a_{k,n} \Big) \radpdf (u) \dif u \\
&\qquad + n\int_0^\gaR \E\bigg[ \Big| f\big(a_{k,n}^{-1} \mD_{k,n}(u)\big) - f\big(a_{k,n}^{-1} \mu_{k,n}(u)\big) \Big|\, \one \Big\{ \big| \mD_{k,n}(u) -\mu_{k,n}(u) \big| \le \eta a_{k,n} \Big\} \bigg] \radpdf (u) \dif u  \\
&=: A_n + B_n, 
\end{align*}
where $\|f\|_\infty:= \sup_xf(x)$. 

We start by  showing  that $A_n \to 0$ as $n \to \infty$. 
Since $\al>1/(2\gamma)$, we have, for any $M>0$, 
$$
\gaR >\frac{2}{(d-1)(k-1)} \, \log (Ma_{k,n}), \ \ \text{for large enough } n. 
$$
It then follows from Lemma \ref{l:Lemma1.OY}, Chebyshev's inequality, and Proposition \ref{p:var.mu.n} that 
\begin{align}
A_n &\le \frac{Cn}{ a_{k,n}^2}\, \int_0^{\frac{2}{(d-1)(k-1)}\log (Ma_{k,n})} \text{Var}\big( \mD_{k,n}(u) \big) e^{-\al(d-1)u}\dif u + Cn \int_{\frac{2}{(d-1)(k-1)}\log (Ma_{k,n})}^\infty e^{-\al(d-1)u} \dif u \label{e:A_n.to.0.second} \\
&\le Cn^{1-\frac{k-1}{\al}} e^{\al(d-1)(2k-3)\omega_n}\int_0^{\frac{2}{(d-1)(k-1)}\log (Ma_{k,n})}e^{\half (d-1)(2k-3-2\al)u} \dif u + CM^{-2\al/(k-1)}. \notag 
\end{align}
If $k - \frac{3}{2} \le \al < k - 1$, then 
$A_n \le Cn^{1-\frac{k-1}{\al}} e^{\al(d-1)(2k-3)\omega_n}\log n + CM^{-2\al/(k-1)}$. On the other hand, if $\frac{1}{2\gamma}< \al < k-\frac{3}{2}$, then  
$A_n \le CM^{\frac{2k-3-2\al}{k-1}} n^{-1/(2\al)}e^{\al(d-1)(2k-3)\omega_n} + CM^{-2\al/(k-1)}$. 
In both regimes, we have $\limsup_{n\to\infty}A_n\le CM^{-2\al/(k-1)}$  for every $M>0$.  Letting  $M\to\infty$ therefore yields $A_n\to0$ as $n\to\infty$.

For the term $B_n$, note that if $\mu_{k,n}(u) < \eta' a_{k,n}$, then $f\big( a_{k,n}^{-1} \mu_{k,n}(u) \big) = 0$, and
$$
\mD_{k,n}(u) \le \big| \mD_{k,n}(u) - \mu_{k,n}(u) \big| + \mu_{k,n}(u) < \eta a_{k,n} + \eta' a_{k,n} < \eta_0 a_{k,n},
$$
which, in turn, implies $f\big( a_{k,n}^{-1} \mD_{k,n}(u) \big) = 0$. Thus, we can insert an additional indicator $\one \big\{ \mu_{k,n}(u) \ge \eta' a_{k,n} \big\}$ under the integral sign in $B_n$ and obtain
$$
B_n \le \omega_f(\eta) n \int_0^\gaR \one \big\{ \mu_{k,n}(u) \ge \eta' a_{k,n} \big\} \radpdf(u) \dif u,
$$
where $\omega_f(\eta)$ is the modulus of continuity of $f$. By  Proposition \ref{p:exp.mu.n} $(i)$ and Lemma \ref{l:Lemma1.OY},
$$
B_n \le C \omega_f(\eta) n \int_{\frac{2}{(d-1)(k-1)} \log(\eta' a_{k,n} / C)}^\infty e^{-\al(d-1)u} \dif u \le C \omega_f(\eta) \to 0, \ \ \text{as } \eta \to 0,
$$
and \eqref{e:Dn.mu.n.PP} is thus established.
\end{proof}

\medskip

\subsection{Proof of \eqref{e:main(i).cond4}}

\begin{lemma}    \label{l:Pruss.maximal.inequ}
Let $\half (k-1)<\al<k-1$; then,  as $n\to\infty$, 
$$
\P \bigg( \sup_{0\le t \le 1} \bigg| \sum_{i=1}^{[N_nt]} \Big(\mD_{k,n}(X_i)\one \{ U_i \le c_1 R_n \} -\E \big[ \mD_{k,n}(X_i)\one \{ U_i \le c_1 R_n \} \big] \Big) \bigg|>\eta a_{k,n} \bigg) \to 0. 
$$
\end{lemma}

\begin{proof}
The summands above form a sequence of exchangeable random variables. Hence, Pruss's maximal inequality for exchangeable random variables (see Corollary 2 in \cite{pruss:1998}) implies that  
\begin{align*}
&\P \bigg( \sup_{0\le t \le 1} \bigg| \sum_{i=1}^{[N_nt]} \Big(\mD_{k,n}(X_i)\one \{ U_i \le c_1 R_n \} -\E \big[ \mD_{k,n}(X_i)\one \{ U_i \le c_1 R_n \} \big] \Big) \bigg|>\eta a_{k,n} \bigg)\\
&= \E \bigg[  \P \bigg( \max_{0\le j \le N_n} \bigg| \sum_{i=1}^{j} \Big(\mD_{k,n}(X_i)\one \{ U_i \le c_1 R_n \} -\E \big[ \mD_{k,n}(X_i)\one \{ U_i \le c_1 R_n \} \big] \Big) \bigg|>\eta a_{k,n} \, \bigg|\,  N_n \bigg) \bigg] \\
&\le C^*  \E \bigg[  \P \bigg(  \, \bigg| \sum_{i=1}^{\lceil N_n/2\rceil} \Big(\mD_{k,n}(X_i)\one \{ U_i \le c_1 R_n \} -\E \big[ \mD_{k,n}(X_i)\one \{ U_i \le c_1 R_n \} \big] \Big) \bigg|>\frac{\eta a_{k,n}}{C^*} \, \bigg|\,  N_n \bigg) \bigg] \\
&\le \frac{(C^*)^3}{\eta^2 a_{k,n}^2}\, \E \bigg[ \text{Var} \Big( \sum_{i=1}^{\lceil N_n/2\rceil} \mD_{k,n}(X_i)\one \{ U_i \le c_1 R_n \} \, \Big| \, N_n\Big) \bigg]\\
&\le \frac{(C^*)^3}{\eta^2 a_{k,n}^2}\, \text{Var} \Big( \sum_{i=1}^{\lceil N_n/2\rceil} \mD_{k,n}(X_i)\one \{ U_i \le c_1 R_n \} \Big), 
\end{align*}
for some universal constant $C^*\in (0,\infty)$. 

Appealing to the Mecke formula to the last term, 
\begin{align}
&\frac{1}{a_{k,n}^2} \text{Var} \Big( \sum_{i=1}^{\lceil N_n/2\rceil} \mD_{k,n}(X_i)\one \big\{ U_i \le c_1 \R_n \} \Big) = \frac{n}{2a_{k,n}^2}\, \int_{0}^{c_1R_n} \E \big[ \mD_{k,n}(u)^2 \big] \radpdf (u) \dif u \label{e:var.Dn(X)} \\
&\qquad + \Big(\frac{n}{2a_{k,n}}  \Big)^2 \int_{\mA_d}\int_{\mA_d} \int_{0}^{c_1R_n}  \int_{0}^{c_1R_n} \text{Cov} \big( \mD_{k,n}(u,\btheta), \, \mD_{k,n}(v,\bphi) \big) \notag \\
&\qquad \qquad \qquad \qquad\qquad \qquad \times \radpdf (u)\radpdf (v) \pi(\btheta)\pi(\bphi)\dif v \dif u \dif  \bphi \dif \btheta +R_n, \notag 
\end{align}
where $R_n$ is the negligible term as $n\to\infty$. We  note that by Proposition \ref{p:exp.mu.n} $(i)$ and Proposition \ref{p:var.mu.n}, 
\begin{align*}
 \E \big[ \mD_{k,n}(u)^2 \big] &= \text{Var} \big( \mD_{k,n}(u) \big) + \mu_{k,n}(u)^2 \\
&\le Ce^{\al(d-1)(2k-3)\omega_n}\cdot e^{\half(d-1)(2k-3)u} + C e^{(d-1)(k-1)u} \le Ce^{\al(d-1)(2k-3)\omega_n}\cdot e^{(d-1)(k-1)u} 
\end{align*}
for all $n\ge1$ and $0<u\le c_1 R_n$. Therefore, 
\begin{align*}
\frac{n}{a_{k,n}^2}\, \int_{0}^{c_1R_n} \E \big[ \mD_{k,n}(u)^2 \big] \radpdf (u) \dif u &\le Cn^{1-\frac{k-1}{\al}} e^{\al(d-1)(2k-3)\omega_n}\int_0^{c_1R_n} e^{(d-1)(k-1-\al)u}\dif u \\
&\le Cn^{-(k-1-\al)(\frac{1}{\al}-2c_1)}e^{\al(d-1)(2k-3)\omega_n} \to 0, \ \ \text{as } n\to\infty. 
\end{align*}

One more application of the Mecke formula shows that the second term in \eqref{e:var.Dn(X)} is equal to 
\begin{align}
 \sum_{\ell=1}^{k-1}&\sum_{s_1=0}^{k-1} \sum_{s_2=0}^{k-1}  \Big( \frac{n}{2a_{k,n}} \Big)^2 \int_{\mA_d}\int_{\mA_d} \int_{0}^{c_1R_n}  \int_{0}^{c_1R_n} \label{e:cov.K.nell}\\
&\quad \E \bigg[ \sum_{(Y_1,\dots,Y_{k-1})\in (\Pn)_{\neq}^{k-1}} \hspace{-10pt}\sum_{\substack{(Z_1,\dots,Z_{k-1})\in (\Pn)_{\neq}^{k-1} \\ |(Y_1, \dots, Y_{k-1})\cap (Z_1,\dots, Z_{k-1})|=\ell}} \hspace{-25pt} g_n^{(s_1)}(p,Y_1,\dots,Y_{k-1}) g_n^{(s_2)}(q,Z_1,\dots,Z_{k-1})\bigg] \notag \\
&\qquad \qquad \qquad \qquad\qquad \qquad\qquad \qquad \times  \radpdf (u)\radpdf (v) \pi(\btheta)\pi(\bphi)\dif v \dif u \dif  \bphi \dif \btheta \notag \\
&=: \sum_{\ell=1}^{k-1}\sum_{s_1=0}^{k-1} \sum_{s_2=0}^{k-1} \Big( \frac{n}{2a_{k,n}} \Big)^2 \notag \\
&\qquad \qquad  \times\int_{\mA_d}\int_{\mA_d} \int_{0}^{c_1R_n}  \int_{0}^{c_1R_n}  K_{n,\ell}^{(s_1, s_2)}(p,q) \, \radpdf (u)\radpdf (v) \pi(\btheta)\pi(\bphi) \dif v \dif u \dif  \bphi \dif \btheta, \notag 
\end{align}
where we set $p=(u,\btheta)$ and $q=(v,\bphi)$ (with $u$ and $v$ being radial components measured from the boundary of $B(o,R_n)$, and $\btheta, \bphi\in \A_d$). We now consider, without loss of generality, the case $s_1=s_2=k-1$, and suppress the superscript $(s_1, s_2)$ from $K_{n,\ell}^{(s_1, s_2)}(p,q)$.  Recalling that $X_i = (U_i,\Theta_i)$, $i=1,\dots,2(k-1)-\ell$, are i.i.d.~random variables with density $\radpdf \otimes \pi$ such that $U_i=R_n-d_H(o,X_i)$, we obtain, by the Mecke formula,  that 
\begin{align*}
K_{n,\ell}(p,q) &= n^{2(k-1)-\ell} \P \big( d_H(X_i,p) \le R_n, \, U_i \le u,\, i=1,\dots,k-1,   \\
&\qquad \qquad \qquad  \quad d_H(X_j, q)\le R_n, \, U_j \le v, \, j=1,\dots, \ell, k, \dots, 2(k-1)-\ell\big) \\
&= n^{2(k-1)-\ell} \Big( \int_0^{u \wedge v} \P \big( d_H(X,p)\le R_n, \, d_H(X,q)\le R_n \, \big| \, t \big) \, \radpdf (t) \dif t\Big)^\ell  \\
&\quad \times  \Big( \int_0^{u} \P \big( d_H(X,p)\le R_n \, \big| \, t \big) \, \radpdf (t) \dif t\Big)^{k-1-\ell}\Big( \int_0^{v} \P \big( d_H(X,q)\le R_n \, \big| \, t \big) \, \radpdf (t) \dif t\Big)^{k-1-\ell}, 
\end{align*}
where $X=(t,\Theta)$ with $t=R_n-d_H(o,X)$. 
According to Lemma \ref{l:Lemma4.OY}, 
\begin{align}
\P \big( d_H(X, p)\le R_n \big|\, t \big)&\le \P \big( d_H(X, p)\le R_n \big|\, t \big) \one \{ t+u\le R_n-\omega_n\} + \one \{ t+u> R_n-\omega_n\}   \label{e:plus.omega} \\
&\le Ce^{\half (d-1)\omega_n} \cdot e^{-\half (d-1)(R_n-t-u)}. \notag 
\end{align}
Combining Lemma \ref{l:Lemma1.OY} with \eqref{e:plus.omega}, and applying an analogous bound to $\P( d_H(X,q)\le R_n \,\big|\, t )$, we see that $K_{n,\ell}(p,q)$ is, up to a constant factor, bounded by

\begin{align*}
&n^\ell e^{(d-1)(k-1-\ell)\omega_n}\cdot e^{\half (d-1)(k-1-\ell)(u+v)} \bigg( \int_0^{c_1R_n} \P \big( d_H(X,p)\le R_n, \, d_H(X, q) \le R_n\, |\, t \big) \\
&\qquad\qquad\qquad\qquad\qquad\qquad\qquad \qquad \qquad \qquad\times  \one \{ t\le u \wedge v \} e^{-\al(d-1)t} \dif t\bigg)^\ell \\
&\qquad \qquad \qquad \qquad \qquad \qquad \times \Big( \int_0^u e^{\half(d-1)(1-2\al)t} \dif t \Big)^{k-1-\ell} \Big( \int_0^v e^{\half(d-1)(1-2\al)t} \dif t \Big)^{k-1-\ell}  \\
&\le n^\ell e^{(d-1)(k-1-\ell)\omega_n}\cdot e^{\half (d-1)(k-1-\ell)(u+v)} \bigg( \int_0^{c_1R_n} \P \big( d_H(X,p)\le R_n, \, d_H(X, q) \le R_n\, |\, t \big) \\
&\qquad\qquad\qquad\qquad\qquad\qquad\qquad\qquad\qquad\qquad \times  \one \{ t\le u \wedge v \} e^{-\al(d-1)t} \dif t\bigg)^\ell. 
\end{align*}

Applying  \eqref{e:plus.omega} once more, the above integral  is upper bounded as follows: 
\begin{align*}
&\int_0^{c_1R_n} \P \big( d_H(X,p)\le R_n, \, d_H(X, q) \le R_n\, |\, t \big) \one \{ t\le u \wedge v \} e^{-\al(d-1)t} \dif t \\
&\le \int_0^{c_1R_n} \P \big( d_H(X,p)\le R_n \, | \, t \big)  e^{-\al(d-1)t}\dif t \\
&\le Ce^{\half (d-1)\omega_n}\int_0^{c_1R_n} e^{-\half(d-1)(R_n-t-u)} \cdot e^{-\al(d-1)t} \dif t \\
&= Cn^{-1} e^{\half (d-1)\omega_n}\cdot e^{\half(d-1)u} \int_0^{c_1R_n} e^{\half(d-1)(1-2\al)t}\dif t \le Cn^{-1} e^{\half (d-1)\omega_n}\cdot e^{\half(d-1)u}. 
\end{align*}
By substituting this bound, 
\begin{align*}
K_{n,\ell}(p,q) &\le C n^\ell e^{(d-1)(k-1-\ell)\omega_n}\cdot e^{\half (d-1)(k-1-\ell)(u+v)} \big( n^{-1} e^{\half (d-1)\omega_n} \cdot e^{\half(d-1)u} \big)^{\ell-1} \\
&\qquad \times \int_0^{c_1R_n} \P \big( d_H(X,p)\le R_n, \, d_H(X, q) \le R_n\, |\, t \big)  \one \{ t\le u \wedge v \} e^{-\al(d-1)t} \dif t  \\
&\le Cn e^{ (d-1)(k-2)\omega_n}\cdot e^{\half(d-1)(k-2)(u+v)} \int_0^{c_1R_n} \P \big( d_H(X,p)\le R_n, \, d_H(X, q) \le R_n\, |\, t \big) \\
&\qquad \qquad \qquad \qquad\qquad \qquad\qquad\qquad \qquad\qquad \qquad  \times  \one \{ t\le u \wedge v \} e^{-\al(d-1)t} \dif t. 
\end{align*}
Since the randomness in the probability $\P \big( d_H(X,p)\le R_n, \, d_H(X,q)\le R_n \,\big|\, t \big)$ lies only in the angular component of $X$, it can be written as
\begin{align*}
&\P \big( d_H(X,p)\le R_n, \, d_H(X, q) \le R_n\, |\, t \big)  = \int_{\mA_d} \one \big\{ d_H(x,p)\le R_n, \, d_H(x,q)\le R_n \big\} \pi(\bpsi) \dif \bpsi, 
\end{align*}
where $x=(t,\bpsi)$ with $t=R_n-d_H(o,x)$ and $\bpsi\in \mA_d$. Consequently, 
\begin{align*}
K_{n,\ell}(p,q) &\le Cn  e^{ (d-1)(k-2)\omega_n}\cdot  e^{\half (d-1)(k-2)(u+v)} \int_0^{c_1R_n}\int_{\mA_d} \one \big\{ d_H(x,p)\le R_n, \, d_H(x,q)\le R_n \big\} \, \pi(\bpsi) \\
&\qquad \qquad \qquad  \qquad \qquad \qquad \qquad\qquad \qquad \qquad\qquad \qquad  \times \one \{ t\le u \wedge v \} e^{-\al(d-1)t} \dif \bpsi \dif t. 
\end{align*}
Note that the right-hand side above does not depend on $\ell$. Substituting this bound into \eqref{e:cov.K.nell} and using Fubini's theorem, we obtain 
\begin{align*}
\eqref{e:cov.K.nell} &\le C \Big( \frac{n}{a_{k,n}} \Big)^2 ne^{ (d-1)(k-2)\omega_n}  \int_{\bpsi\in \mA_d} \pi(\bpsi)\int_{t=0}^{c_1R_n} e^{-\al(d-1)t}  \\
&\quad \times  \int_{t}^{c_1R_n} e^{\half(d-1)(k-2-2\al)u}  \Big(\int_{\mA_d} \one \{ d_H(x,p)\le R_n \}\,  \pi(\btheta) \dif \btheta \Big) \dif u  \\
&\quad \times  \int_{t}^{c_1R_n} e^{\half(d-1)(k-2-2\al)v} \Big(\int_{\mA_d} \one \{ d_H(x,q)\le R_n \}\,  \pi(\bphi) \dif \bphi \Big) \dif v. 
\end{align*}
By Lemma \ref{l:Lemma4.OY} and the bound in \eqref{e:plus.omega}, 
\begin{align*}
&\int_{t}^{c_1R_n} e^{\half(d-1)(k-2-2\al)u}  \Big(\int_{\mA_d} \one \{ d_H(x,p)\le R_n \}\,  \pi(\btheta) \dif \btheta \Big) \dif u \\
&\le C e^{\half (d-1)\omega_n} \int_t^{c_1R_n} e^{\half(d-1)(k-2-2\al)u}  \cdot e^{-\half (d-1)(R_n-t-u)} \dif u \\
&\le Cn^{-1} e^{\half (d-1)\omega_n} \cdot e^{\half (d-1)(k-2\al)t}, 
\end{align*}
and hence, 
\begin{align*}  
\eqref{e:cov.K.nell} &\le C \Big( \frac{n}{a_{k,n}} \Big)^2 n e^{ (d-1)(k-2)\omega_n}\int_{\bpsi\in \mA_d} \pi(\bpsi)\int_{t=0}^{c_1R_n} e^{-\al(d-1)t} \Big(  n^{-1} e^{\half (d-1)\omega_n} \cdot e^{\half (d-1)(k-2\al)t} \Big)^2  \\
&=C n^{1-\frac{k-1}{\al}}e^{ (d-1)(k-1)\omega_n}\int_0^{c_1R_n} e^{(d-1)(k-3\al)t} \dif t. 
\end{align*}
If $k\ge3$, then $k-3\al<0$, so that $\eqref{e:cov.K.nell}\le Cn^{1-\frac{k-1}{\al}}e^{ (d-1)(k-1)\omega_n}\to 0$ as $n\to\infty$. If $k=2$ and $k-3\al=2-3\al \le 0$, then $\eqref{e:cov.K.nell}\le Cn^{1-1/\al} e^{(d-1)\omega_n}R_n\to0$ as $n\to\infty$. Finally, suppose that $k=2$ and $k-3\al=2-3\al>0$. Then, it must be that $1/2<\al < 2/3$, so that 
$$
n^{1-\frac{1}{\al}} e^{(d-1)\omega_n} \int_0^{c_1R_n} e^{(d-1)(2-3\al)t}\dif t\le Cn^{1-\frac{1}{\al}+2c_1(2-3\al)}e^{(d-1)\omega_n}. 
$$
Since $c_1 < 1/(2\al)$ and $1/2 < \al < 2/3$, it is straightforward to check that $1-\frac{1}{\al}+2c_1(2-3\al) < 0$, and therefore, \eqref{e:cov.K.nell} tends to $0$ as $n \to \infty$.
\end{proof}
\medskip

\subsection{Negligibility of $B_n(u)$, $C_{n,\ell}(u)$, $D_n(u)$, and $E_n(u)$}

\begin{lemma} \label{l:negligible.BCDE}
In the setting of Proposition \ref{p:exp.clique}, there exist $C, c_0 \in (0,\infty)$, both of which are independent of $n$ and $u$, such that 
\begin{equation*}
B_n(u) + \sum_{\ell=2}^{m-2} C_{n,\ell}(u) + D_n(u) + E_n(u) \le  \begin{cases}
C \beta_n e^{\half (d-1)u}, & \text{for all } n \ge 1 \text{ and } u \in (0,\gaR), \\[3pt]
Ce^{c_0\omega_n} \cdot e^{\half (d-1)u}, & \text{for all } n \ge 1 \text{ and } u \in (0,R_n), 
\end{cases}
\end{equation*}
for some sequence $\beta_n\to0$. 
\end{lemma}

\begin{proof}
For clarity of exposition, we shall omit $\omega_n$ throughout the proof. Since  each of the  conditions imposed by the indicators  in $B_n(u)$, $C_{n,\ell}(u)$, $D_n(u)$, and $E_n(u)$ implies $u \in (R_n/2, R_n)$, we may (and do) assume that $u \in (R_n/2, R_n)$. 
By this feature, we have for any $0 \le \vep < 1$, 
\begin{equation}  \label{e:bound.of.1}
1= e^{-\half (d-1-\vep)u} \cdot e^{\half (d-1-\vep)u} \le n^{-\frac{1}{2}(1-\vep/(d-1))} e^{\half (d-1-\vep)u}. 
\end{equation}

We begin with $B_n(u)$: For $u \in (R_n/2, R_n)$, it follows from Lemma \ref{l:Lemma1.OY} $(i)$ that 
\begin{align*}
B_n(u) &\le Cn^{m-1} \int_{R_n/2 \le t_2 \le \cdots \le t_{m-1}\le u }\prod_{i=2}^{m-1} e^{-\al(d-1)t_i} \int_{t_1=R_n-t_2}^{t_2} e^{-\al(d-1)t_1} \\
&\le  C n^{m-1-2\al} \int_{R_n/2 \le t_2 \le \cdots \le t_{m-1}\le u }\prod_{i=3}^{m-1} e^{-\al(d-1)t_i}\\
&\le C u n^{m-1-2\al}  \int_{R_n/2 \le t_3 \le \cdots \le t_{m-1}\le u }\prod_{i=3}^{m-1} e^{-\al(d-1)t_i} \\
&\le C u n^{m-1-2\al}  \big( e^{-\half \al(d-1)R_n} \big)^{m-3} = Cun^{(m-1)(1-\al)}. 
\end{align*}
By \eqref{e:bound.of.1} with positive and sufficiently small $\vep$, 
\begin{align}
B_n(u) &\le C un^{(m-1)(1-\al)} n^{-\frac{1}{2}(1-\vep/(d-1))} e^{\half (d-1-\vep)u}\label{e:Bn(u).on1} \\
&= Cn^{(m-1)(1-\al)-\frac{1}{2} + \frac{\vep}{2(d-1)}} \mainu =o_n(1) \mainu. \notag 
\end{align}
The last equality follows from the assumption $\al>(2m-3)/(2m-2)$. 

For $m\ge4$ and $\ell\in \{2,\dots,m-2\}$, we next deal with $C_{n,\ell}(u)$. For $u\in (R_n/2, R_n)$, 
\begin{align*}
C_{n,\ell}(u) &\le n^{m-1} \int_\wholet \P \big( d_H(X_1,X_i) \le R_n, \, i=1,\dots,\ell \, \big| \bt \big) \\
&\qquad \qquad \qquad \qquad \qquad\times \one \{ t_1 + t_\ell \le R_n, \, t_1 + t_{\ell+1} >R_n \} \radpdf (\bt). 
\end{align*}
By Lemma \ref{l:Lemma4.OY}, if $t_1+t_\ell\le R_n$, then 
\begin{equation}  \label{e:prob.X1.others}
\P \big( d_H(X_1,X_i) \le R_n, \, i=1,\dots,\ell \, \big| \bt \big) \le Cn^{-(\ell-1)} e^{\half (d-1)(\ell-1)t_1} \prod_{i=2}^{\ell} e^{\half (d-1)t_i}, 
\end{equation}
and applying Lemma \ref{l:Lemma1.OY} $(i)$, together with \eqref{e:prob.X1.others}, 
\begin{align*}
C_{n,\ell}(u) &\le Cn^{m-\ell} \int_\wholet e^{\half (d-1)(\ell-1-2\al)t_1} \prod_{i=2}^{\ell} e^{\half (d-1)(1-2\al)t_i} \\
&\qquad \qquad\qquad  \times \prod_{i=\ell+1}^{m-1}e^{-\al(d-1)t_i} \one \{ t_1+t_\ell \le R_n, \, t_1+t_{\ell+1} >R_n \}\\
&= Cn^{m-\ell} \int_{t_1=R_n-u}^{R_n/2} e^{\half (d-1)(\ell-1-2\al)t_1} \int_{(t_2,\dots, t_\ell):\, t_1\le t_2 \le \cdots \le t_\ell \le R_n-t_1} \prod_{i=2}^{\ell} e^{\half (d-1)(1-2\al)t_i}\\
&\qquad \qquad\qquad  \times \int_{(t_{\ell+1}, \dots, t_{m-1}):\, R_n-t_1 < t_{\ell+1} \le \cdots \le t_{m-1} \le u}\prod_{i=\ell+1}^{m-1}e^{-\al(d-1)t_i}\\
&\le Cn^{m-\ell} \int_{t_1=R_n-u}^{R_n/2} e^{\half (d-1)(\ell-1-2\al)t_1} \Big(  \int_{t_1}^{R_n-t_1} e^{\half (d-1) (1-2\al)t}\dif t\Big)^{\ell-1}\\
&\qquad \qquad \qquad \qquad\qquad \qquad \qquad \qquad\times \Big(  \int_{R_n-t_1}^{u} e^{-\al(d-1)t}\dif t\Big)^{m-\ell-1}\\
&\le Cn^{m-\ell-2(m-\ell-1)\al} \int_{t_1=R_n-u}^{R_n/2} e^{(d-1) (\ell-1+(m-2\ell-1)\al)t_1}\\
&\le C n^{m-\ell-2(m-\ell-1)\al} e^{\half (d-1) (\ell-1+(m-2\ell-1)\al)R_n} \\
&= Cn^{(m-1)(1-\al)}.  
\end{align*}
By \eqref{e:bound.of.1} with $\vep=0$, it holds that 
\begin{equation}  \label{e:Cn(u).on1}
C_{n,\ell}(u)\le Cn^{(m-1)(1-\al)-1/2 } \mainu =o_n(1) \mainu. 
\end{equation}

Next, let $m\ge3$ and $u\in (R_n/2,R_n)$. It again follows from Lemmas \ref{l:Lemma1.OY} and \ref{l:Lemma4.OY} that 
\begin{align*}
D_n(u) &\le n^{m-1} \int_\wholet \P\big( d_H(X_1,X_i)\le R_n, \, i=2,\dots,m-1 \, \big| \bt \big) \\
&\qquad\qquad\qquad\qquad\qquad\qquad\times  \one \{ t_1+t_{m-1}\le R_n, \, t_1+u >R_n \} \radpdf (\bt)  \\
&\le Cn \int_\wholet e^{\half (d-1)(m-2-2\al)t_1} \prod_{i=2}^{m-1} e^{\half (d-1)(1-2\al)t_i} \one \{t_1 + u>R_n\}   \\
&=Cn \int_{t_1=R_n-u}^{R_n/2} e^{\half (d-1)(m-2-2\al)t_1} \int_{(t_2,\dots,t_{m-1}):\, t_1 \le t_2 \le \cdots \le t_{m-1}\le u}\prod_{i=2}^{m-1} e^{\half (d-1)(1-2\al)t_i}  \\
&\le Cn  \int_{t_1=R_n-u}^{R_n/2} e^{(d-1)(m-2-(m-1)\al)t_1}  \\
&\le Cne^{(d-1)(m-2-(m-1)\al)(R_n-u)}  \\
&=Cn^{2m-3-(2m-2)\al} \mainu \cdot e^{\half (d-1)((2m-2)\al-(2m-3))u}.  
\end{align*}
If $u\in (R_n/2, \gaR)$, then 
\begin{align}
D_n(u)&\le Cn^{2m-3-(2m-2)\al} \mainu \cdot e^{\half (d-1)((2m-2)\al-(2m-3))\gaR} \label{e:Dn(u).on1.no.omega}\\
&=Cn^{(1-\gamma) (2m-3-(2m-2)\al)} \mainu =o_n(1) \mainu. \notag 
\end{align}
On the other hand, if $u\in (R_n/2, R_n)$ (so $u$ may be larger than $\gaR$), then 
\begin{equation}  \label{e:Dn(u).on1.omega}
D_n(u)\le Cn^{2m-3-(2m-2)\al} \mainu \cdot e^{\half (d-1)((2m-2)\al-(2m-3))R_n} =C\mainu. 
\end{equation}
In the latter  case, if the previously omitted term $\omega_n$ is restored, a more accurate bound can be expressed as 
$$
D_n(u)\le C e^{c_0\omega_n} \cdot \mainu, 
$$
for some constant $c_0>0$. In contrast, in the former case of \eqref{e:Dn(u).on1.no.omega}, even after restoring  $\omega_n$, we still obtain 
$$
D_n(u) \le Cn^{(1-\gamma) (2m-3-(2m-2)\al)} e^{c_0\omega_n}\cdot \mainu =o_n(1)  \mainu. 
$$
Namely, the contribution of $\omega_n$ in the former case  is  negligible in the final bound for $D_n(u)$.

Finally, we observe that the indicator function in $E_n(u)$ can be bounded above by $\one \{ t_1+u \le R_n, \, t_{m-1}+u >R_n \}$, and by Lemmas \ref{l:Lemma1.OY} and \ref{l:Lemma4.OY}, we have for $u\in (R_n/2,R_n)$, 
\begin{align*}
E_n(u) &\le n^{m-1} \int_\wholet \P \big( d_H(X_1,X_i) \le R_n, \, i=2,\dots,m-1, \ d_H(X_1,p) \le R_n \, \big| \bt \big) \\
&\qquad \qquad\qquad \qquad\qquad \qquad\qquad \qquad \times \one \{ t_1+u\le R_n, \, t_{m-1}+ u >R_n \} \\
&\le C\mainu \int_\wholet e^{\half (d-1)(m-1-2\al)t_1} \prod_{i=2}^{m-1} e^{\half (d-1)(1-2\al)t_i}\\
&\qquad \qquad \qquad \qquad \qquad \qquad \qquad \times \one \{t_1 + u\le R_n, \, t_{m-1}+u >R_n \}\\
&=C\mainu \int_{t_1=0}^{R_n-u} e^{\half (d-1)(m-1-2\al)t_1} \int_{(t_2,\dots, t_{m-2}):\, t_1 \le t_2 \le \cdots \le t_{m-2}\le u}  \prod_{i=2}^{m-2} e^{\half (d-1)(1-2\al)t_i}  \\
&\qquad \qquad \qquad \qquad \qquad \qquad \qquad \qquad \times \int_{t_{m-1}=(R_n-u)\vee t_{m-2}}^u e^{\half (d-1)(1-2\al)t_{m-1}}\\
&\le C\mainu \cdot e^{\half (d-1)(1-2\al)(R_n-u)} \int_{t_1=0}^{R_n-u} e^{\half (d-1)(m-1-2\al)t_1}  \Big( \int_{t_1}^u e^{\half (d-1) (1-2\al)t}\dif t \Big)^{m-3}  \\
&\le C\mainu \cdot e^{\half (d-1)(1-2\al)(R_n-u)} \int_{t_1=0}^{R_n-u} e^{(d-1)(m-2)(1-\al)t_1}\\
&\le C\mainu \cdot e^{\half (d-1)(2m-3-(2m-2)\al)(R_n-u)}. 
\end{align*}
If $u\in (R_n/2, \gaR)$, then 
\begin{align}
E_n(u) &\le C\mainu \cdot e^{\half (d-1)(2m-3-(2m-2)\al)(1-\gamma)R_n} \label{e:En(u).on1.no.omega} \\
&\le Cn^{(1-\gamma)(2m-3-(2m-2)\al)} \mainu = o_n(1) \mainu. \notag 
\end{align}
Furthermore, if $u\in (R_n/2,R_n)$, we use $u< R_n$ (instead of $u<\gaR$ as above) and restore $\omega_n$, which yields  that 
\begin{equation}  \label{e:En(u).on1.omega}
E_n(u)\le Ce^{c_0\omega_n}\cdot \mainu. 
\end{equation}

It now follows from \eqref{e:Bn(u).on1}, \eqref{e:Cn(u).on1}, \eqref{e:Dn(u).on1.no.omega}, and \eqref{e:En(u).on1.no.omega} that for all $n\ge1$ and $u\in (0,\gaR)$, 
$$
 B_n(u) + \sum_{\ell=2}^{m-2} C_{n,\ell}(u) + D_n(u) + E_n(u) \le C\beta_n \mainu, 
$$
for some sequence $\beta_n\to0$. If $u\in (0,R_n)$, we obtain from \eqref{e:Bn(u).on1}, \eqref{e:Cn(u).on1},  \eqref{e:Dn(u).on1.omega}, and \eqref{e:En(u).on1.omega} that 
$$
 B_n(u) + \sum_{\ell=2}^{m-2} C_{n,\ell}(u) + D_n(u) + E_n(u) \le C e^{c_0\omega_n}\cdot  \mainu. 
$$
\end{proof}
\medskip

\subsection{Proof of \eqref{e:main(i).cond4.clique}}

\begin{lemma}    \label{l:Pruss.maximal.inequ.clique}
We have, as $n\to\infty$, 
$$
\P \bigg( \sup_{0\le t \le 1} \bigg| \sum_{i=1}^{[N_nt]} \Big(\mC_{m,n}(X_i)\one \{ U_i \le c_1 R_n \} -\E \big[ \mC_{m,n}(X_i)\one \{ U_i \le c_1 R_n \} \big] \Big) \bigg|>\eta b_{m,n} \bigg) \to 0. 
$$
\end{lemma}

\begin{proof}
Analogously to \eqref{e:var.Dn(X)}, we need to prove that as $n\to\infty$, 
$$
\frac{n}{b_{m,n}^2}\, \int_{0}^{c_1R_n} \E \big[ \mC_{m,n}(u)^2 \big] \radpdf (u) \dif u \to 0, 
$$
and 
\begin{align}  
&\Big(\frac{n}{b_{m,n}} \Big)^2 \int_{\mA_d}\int_{\mA_d} \int_{0}^{c_1R_n}  \int_{0}^{c_1R_n} \text{Cov} \big( \mC_{m,n}(u,\btheta), \, \mC_{m,n}(v,\bphi) \big) \label{e:cov.clique} \\
&\qquad \qquad \qquad \qquad \qquad \qquad \qquad \times  \radpdf (u)\radpdf (v) \pi(\btheta)\pi(\bphi)\dif v \dif u \dif  \bphi \dif \btheta \to 0. \notag
\end{align}
By Proposition \ref{p:exp.clique} $(i)$ and Proposition \ref{p:var.clique}, 
$$
 \E \big[ \mC_{m,n}(u)^2 \big] =\text{Var}\big(\mC_{m,n}(u)  \big)+\nu_{m,n}(u)^2 \le Ce^{\tilde c_1\omega_n}\cdot e^{(d-1)u}, 
$$
for some $\tilde c_1\in (0,\infty)$, 
from which we get that 
\begin{align*}
\frac{n}{b_{m,n}^2}\, \int_{0}^{c_1R_n} \E \big[ \mC_{m,n}(u)^2 \big] \radpdf (u) \dif u &\le Cn^{1-1/\al}e^{\tilde c_1\omega_n} \int_0^{c_1R_n} e^{(d-1)(1-\al)u}\dif u \\
&\le Cn^{(1-\al)(2c_1-\frac{1}{\al})} e^{\tilde c_1\omega_n} \to 0, \ \ \ n\to\infty. 
\end{align*}

Appealing to the Mecke formula as in \eqref{e:cov.K.nell}, one can see that \eqref{e:cov.clique} is equal to 
\begin{align}
&\sum_{q=1}^{m-1}\Big( \frac{n}{b_{m,n}} \Big)^2 \int_{\mA_d}\int_{\mA_d} \int_{0}^{c_1R_n}  \int_{0}^{c_1R_n}  \frac{1}{\big((m-q-1)!\big)^2 q!}\label{e:cov.K.nq}\\
&\quad\times \E \bigg[ \sum_{(Y_1,\dots,Y_{2(m-1)-q})\in (\Pn)_{\neq}^{2(m-1)-q}} \hspace{-25pt} h_n(p,Y_1,\dots,Y_{m-1}) h_n(p',Y_1,\dots,Y_q, Y_m,\dots,Y_{2(m-1)-q})\bigg] \notag \\
&\qquad \qquad \qquad \qquad\qquad \qquad\qquad \qquad \times  \radpdf (u)\radpdf (v) \pi(\btheta)\pi(\bphi)\dif v \dif u \dif  \bphi \dif \btheta \notag \\
&=: \sum_{q=1}^{m-1}\Big( \frac{n}{b_{m,n}} \Big)^2 \int_{\mA_d}\int_{\mA_d} \int_{0}^{c_1R_n}  \int_{0}^{c_1R_n}  K_{n,q}(p,p') \, \radpdf (u)\radpdf (v) \pi(\btheta)\pi(\bphi) \dif v \dif u \dif  \bphi \dif \btheta, \notag 
\end{align}
where we set $p=(u,\btheta)$, $p'=(v,\bphi)$ with $u=R_n-d_H(o,p)$, $v=R_n-d_H(o,p')$, and $h_n(\cdot)$ is defined in \eqref{e:def.hn}. 

As in the proof of Proposition \ref{p:var.clique},  we consider two $m$-cliques,  again denoted by $\mC_1$ and $\mC_2$.  We first define the three points $Z=(W,\Psi)$, $Y_1=(V_1,\Ta_1)$, and $Z_1=(W_1,\Phi_1)$  in the same way as  the proof of Proposition \ref{p:var.clique},    and we again restrict attention to the case $W \ge V_1 \vee W_1$. Suppose without loss of generality that $p\in \mC_1$ and $p'\in \mC_2$. We now take the path formed by the nodes in $\mC_1$: 
\begin{equation} \label{e:path.no.removal.cov}
Y_1 \to Z, \ \ Y_1\to Y_2 \to \cdots \to Y_{m-2} \to p, 
\end{equation}
where $Y_i=(V_i, \Ta_i)$, $V_i=R_n-d_H(o,Y_i)$, $i=1,\dots,m-2$, such that $V_1\le \cdots \le V_{m-2}\le u$. Note that all the $m$ nodes in \eqref{e:path.no.removal.cov} form $\mC_1$. Additionally, we consider the path consisting of nodes in $\mC_2$: 
\begin{equation} \label{e:path.removal.cov}
Z_1 \to Z, \ \ Z_1\to Z_2 \to\cdots \to Z_{m-q-1} \to p',
\end{equation}
where $Z_i=(W_i, \Phi_i)$, $W_i=R_n-d_H(o,Z_i)$, $i=1,\dots,m-q-1$, such that $W_1 \le \cdots \le W_{m-q-1}\le v$. 

Under these configurations, we can bound $K_{n,q}(p,p')$ in \eqref{e:cov.K.nq} by 
\begin{align*}
L_{n,q}(p,p') &:= n^{2m-q-2}\, \P \big( \max_{1\le i \le m-3}d_H(Y_i,Y_{i+1})\le R_n, \, d_H(Y_{m-2},p)\le R_n, \\
&\qquad \qquad \qquad \quad  \max_{1\le i \le m-q-2} d_H(Z_i, Z_{i+1})\le R_n, \, d_H(Z_{m-q-1}, p')\le R_n, \\ 
&\qquad \qquad \qquad \quad d_H(Y_1,Z)\le R_n, \, d_H(Z_1,Z)\le R_n, \\
&\qquad \qquad \qquad \quad V_1 \le \cdots \le V_{m-2} \le u, \, W_1 \le \cdots \le W_{m-q-1} \le v,\, V_1 \vee W_1 \le W \le u \wedge v \big). 
\end{align*}
Substituting $L_{n,q}(p,p')$ into \eqref{e:cov.K.nq}, we need to bound the integral 
\begin{align}
B_{n,q} &:= \Big( \frac{n}{b_{m,n}} \Big)^2 \int_{\mA_d}\int_{\mA_d} \int_{0}^{c_1R_n}  \int_{0}^{c_1R_n}  L_{n,q}(p,p') \, \radpdf (u)\radpdf (v) \pi(\btheta)\pi(\bphi) \dif v \dif u \dif  \bphi \dif \btheta \label{e:B.nq} \\
&= \frac{n^{2m-q}}{b_{m,n}^2}\, \int_{\mU_n(\bt, \bs, u, v, w)} \P \big( \max_{1\le i \le m-3}d_H(Y_i,Y_{i+1})\le R_n, \, d_H(Y_{m-2},A)\le R_n, \, \notag \\
&\qquad \qquad \qquad \qquad  \qquad\quad \max_{1\le i \le m-q-2} d_H(Z_i, Z_{i+1})\le R_n, \, d_H(Z_{m-q-1}, B)\le R_n, \, \notag \\
&\qquad \qquad \qquad \qquad  \qquad\quad  d_H(Y_1,Z)\le R_n, \, d_H(Z_1,Z)\le R_n\, \big|\, \bt, \bs, u, v, w\big) \notag \\
&\qquad \qquad \qquad \qquad  \qquad\quad  \radpdf (\bt)\radpdf (\bs)\radpdf (u)\radpdf (v) \radpdf (w), \notag
\end{align}
where, by the same abuse of notation as  in \eqref{e:main.variance.clique}, we set $Y_i=(t_i, \Ta_i)$ for $i=1,\dots,m-2$,  $Z_i=(s_i, \Phi_i)$ for $i=1,\dots,m-q-1$, and $Z=(w,\Psi)$. Furthermore,  we define $A=(u,\Ta)$ and $B=(v,\Phi)$, and
\begin{align*}
\mU_n(\bt, \bs, u, v, w) &:= \big\{  (\bt, \bs, u, v,w):  0\le t_1\le \cdots \le t_{m-2}\le u \le c_1R_n, \\
&\qquad \qquad \qquad \qquad \quad 0\le s_1\le \cdots \le s_{m-q-1}\le v \le c_1R_n, \, t_1 \vee s_1 \le w\le u \wedge v \big\}. 
\end{align*}

The configurations in \eqref{e:path.no.removal.cov} and \eqref{e:path.removal.cov} do not contain any cycles, so by  independence of the angular components, the conditional probability in \eqref{e:B.nq}  can be written in product form as in \eqref{e:prod.form.var.clique}. Applying bounds of the type \eqref{e:ignore.restriction.omega} repeatedly to this product representation, the conditional probability in \eqref{e:B.nq} can be bounded above by
$$
C (n^{-1}e^{\half (d-1)\omega_n})^{2m-q-1} \prod_{i=1}^{m-2} e^{(d-1)t_i} \prod_{i=1}^{m-q-1} e^{(d-1)s_i} \cdot e^{(d-1)w}\cdot e^{\half (d-1) u}\cdot e^{\half (d-1)v}
. 
$$
Applying this bound as well as Lemma \ref{l:Lemma1.OY}, 
\begin{align*}
B_{n,q} &\le \frac{Cn}{b_{m,n}^2} e^{(d-1)(m-1)\omega_n} \int_{\mU_n(\bt, \bs, u, v, w)}  \prod_{i=1}^{m-2} e^{(d-1)(1-\al)t_i} \prod_{i=1}^{m-q-1} e^{(d-1)(1-\al)s_i}\\
&\qquad \qquad \qquad \qquad \qquad  \times e^{(d-1)(1-\al)w}\cdot e^{\half (d-1)(1-2\al) u}\cdot e^{\half (d-1)(1-2\al)v}  \\
&= \frac{Cn}{b_{m,n}^2} e^{(d-1)(m-1)\omega_n} \int_{u=0}^{c_1R_n} e^{\half (d-1)(1-2\al) u} \int_{v=0}^{c_1R_n} e^{\half (d-1)(1-2\al) v} \int_{w=0}^{u \wedge v} e^{ (d-1)(1-\al) w} \\
&\quad \times \int_{t_1=0}^w e^{(d-1)(1-\al)t_1} \int_{s_1=0}^w e^{(d-1)(1-\al)s_1}  \int_{(t_2,\dots,t_{m-2}):\, t_1 \le t_2\le \cdots \le t_{m-2}\le u} \prod_{i=2}^{m-2} e^{(d-1)(1-\al)t_i} \\
&\quad  \times  \int_{(s_2,\dots,s_{m-q-1}):\, s_1 \le s_2\le \cdots \le s_{m-q-1}\le v}  \prod_{i=2}^{m-q-1} e^{(d-1)(1-\al)s_i} \\
&\le \frac{Cn}{b_{m,n}^2} e^{(d-1)(m-1)\omega_n} \int_{u=0}^{c_1R_n} e^{\half (d-1)(2m-5-(2m-4)\al) u}  \int_{v=0}^{c_1R_n} e^{\half (d-1)(2m-5-(2m-4)\al) v} \\
&\qquad \qquad \times  \int_{w=0}^{u \wedge v} e^{ (d-1)(1-\al) w}  \int_{t_1=0}^w e^{(d-1)(1-\al)t_1} \int_{s_1=0}^w e^{(d-1)(1-\al)s_1} \\
&\le \frac{Cn}{b_{m,n}^2} e^{(d-1)(m-1)\omega_n} \int_{u=0}^{c_1R_n} e^{\half (d-1)(2m-5-(2m-4)\al) u} \\
&\qquad\qquad\qquad\qquad\qquad\times   \int_{v=0}^{c_1R_n} e^{\half (d-1)(2m-5-(2m-4)\al) v} \cdot e^{(d-1)(3-3\al)(u\wedge v)} \\
&= \frac{2Cn}{b_{m,n}^2} e^{(d-1)(m-1)\omega_n}\int_{u=0}^{c_1R_n} e^{\half (d-1)(2m+1-(2m+2)\al) u}  \int_{v=u}^{c_1R_n} e^{\half (d-1)(2m-5-(2m-4)\al) v} \\
&\le Cn^{1-1/\al} e^{(d-1)(m-1)\omega_n}\int_{u=0}^{c_1R_n} e^{ (d-1)(2m-2-(2m-1)\al) u}. 
\end{align*}
If $\frac{2m-2}{2m-1}<\al< 1$, then as $n\to\infty$, 
\begin{align*}
B_{n,q} &\le Cn^{1-\frac{1}{\al} + 2c_1(2m-2-(2m-1)\al)}  e^{(d-1)(m-1)\omega_n} \le Cn^{-(2m-2)+\frac{2m-3}{\al}} e^{(d-1)(m-1)\omega_n} \to 0, 
\end{align*}
where the second inequality follows from the constraint \eqref{e:constraint.c1.clique}. Finally, if $\frac{2m-3}{2m-2}<\al\le \frac{2m-2}{2m-1}$, then $B_{n,q}\le Cn^{1-1/\al}  e^{(d-1)(m-1)\omega_n} \log n\to 0$ as $n\to\infty$. 
\end{proof}
\medskip


\end{document}